\documentclass[11pt]{article}

\usepackage{amssymb,amsfonts,amsmath,amsthm}
\usepackage{graphicx}
\usepackage{color}
\usepackage{enumerate}
\usepackage{amssymb, amsfonts, amsthm,amsmath}
\usepackage{amsfonts,color,amssymb,mathrsfs,stmaryrd}

\usepackage{setspace}
\usepackage{graphicx}

\usepackage{tikz}

\tikzstyle{every node}=[color=black, circle, draw, fill=black, inner sep=0pt, minimum width=2pt]
\tikzstyle{left}=[color=green, circle, draw, fill=green, inner sep=0pt, minimum width=4pt]
\tikzstyle{right}=[color=red, circle, draw, fill=red, inner sep=0pt, minimum width=4pt]
\tikzstyle{s}=[color=black, circle, draw, fill=black, inner sep=0pt, minimum width=1.5pt]
\tikzstyle{l}=[color=black, circle, draw, fill=black, inner sep=0pt, minimum width=3pt]
\tikzstyle{plain}=[color=black, fill=none, draw=none, inner sep=0pt, minimum width=2pt]
\usetikzlibrary{calc}
\usetikzlibrary{decorations.pathreplacing}

\newcommand{\tod}{\stackrel{{\cal D}}{\longrightarrow}}
\newcommand{\toP}{\stackrel{\widehat{\mathbb{P}}_n}{\longrightarrow}}

\newcommand{\txi} {\tilde{\xi}}
\newcommand{\tS} {\tilde{S}}

\newcommand{\A}{\mathcal{A}}

\newcommand{\Dom}{W}
\newcommand{\eps}{\varepsilon}
\newcommand{\Pro}[1]{\mathbb{P} \left(\,#1\,\right)}
\newcommand{\Ex}[1]{\mathbb{E} \left[\, #1\,\right]}

\newcommand{\Var}{{\rm Var}}

\newcommand{\Ball}[2]{B^{\pm}_{#2} (#1)}

\newcommand{\HBallUp}[1]{ {\mathcal{B}^+ } (#1)}
\newcommand{\HBallLow}[1]{ {\mathcal{B}^- } (#1)}
\newcommand{\HBall}[1]{\mathcal{B}(#1)}
\newcommand{\HD}[1]{\mathcal{D} (#1)}
\newcommand{\BallDown} { B^-(p) }
\newcommand{\BallUp}{ B^+(p) }

\newcommand{\Y}{Y}

\newcommand{\Space}{\mathcal{\mathbf{S}}}
\newcommand{\RR}{\mathbb{R}}

\newcommand{\OneUp}{1_{+\varepsilon}}
\newcommand{\OneLow}{1_{-\varepsilon}}
\newcommand{\OnePM}{1_{\pm\varepsilon}}
\renewcommand{\P}{\mathbb{P} }

\newcommand{\N}{\mathbb{N}}

\newcommand{\R}[1] {D (#1)}

\newcommand{\Part}[1]{\mathrm{S}_{#1}}
\newcommand{\InterSec}[1]{\mathrm{S}^{\pm}_{#1}}

\newcommand{\tP} {\tilde{\cal P}_{\alpha, D}}
\newcommand{\tPH} {\tilde{\cal P}_{\alpha, D , H}}

\newcommand{\cP}{\mathcal{P}}
\newcommand{\Rcal}{D}

\newcommand{\D}{\mathcal{D}_R}
\newcommand{\Dis}[1]{\mathcal{D}_{#1}}

\newcommand{\Gnan} { \mathcal{G}_{\alpha,n} }

\def\be{\begin{equation}}
\def\ee{\end{equation}}
\def\bea{\begin{eqnarray}}
\def\eea{\end{eqnarray}}

\DeclareMathOperator{\Cov}{Cov}

\newtheorem{theorem}{Theorem}  

\newtheorem{lemma}[theorem]{Lemma}

\newtheorem{corollary}[theorem]{Corollary}

\numberwithin{theorem}{section}
\numberwithin{equation}{section}

\title{Limit theory for  isolated  and extreme points in hyperbolic random geometric graphs
}

 \author{
Nikolaos Fountoulakis\footnote{School of Mathematics, University of Birmingham, United Kingdom  e-mail: \texttt{n.fountoulakis@bham.ac.uk}; Research partially supported by the Alan Turing Institute, grant EP/N510129/1}
\qquad
Joseph Yukich\footnote{Department of Mathematics, Lehigh University, USA e-mail: \texttt{jey0@lehigh.edu}; Research supported in part by Simons Collaboration grant 519427}
}

\date{November 2, 2020}

\begin{document}
\maketitle

\begin{abstract}   Given $\alpha \in (0, \infty)$ and $r \in (0, \infty)$, let ${\cal D}_{r, \alpha}$
 be the disc of radius $r$ in the hyperbolic plane having curvature $-
\alpha^2$.  Consider  the Poisson point process
having uniform intensity density on ${\cal D}_{R, \alpha}$, with $R = 2 \log(n/ \nu),$ $n \in \N$, and $\nu < n$ a fixed constant.
The points are projected onto ${\cal D}_{R, 1}$, preserving
polar coordinates, yielding a Poisson point process ${\cal P}_{\alpha, n}$ on  ${\cal D}_{R, 1}$.
The hyperbolic geometric graph ${\cal G}_{\alpha, n}$ on ${\cal P}_{\alpha, n}$  puts an edge
between pairs of points of ${\cal P}_{\alpha, n}$ which are  distant at most $R$.
This model has been used to express fundamental features of complex networks in terms of 
an underlying hyperbolic geometry.

For $\alpha \in (1/2, \infty)$ we establish expectation and variance asymptotics as well as asymptotic
 normality for the number of isolated and extreme points in ${\cal G}_{\alpha, n}$ as $n \to \infty.$  The
limit theory and renormalization for the number of isolated points are highly sensitive on the curvature parameter. In particular, for $\alpha \in (1/2, 1)$, the variance is super-linear, for $\alpha = 1$ the variance is linear with a logarithmic correction, whereas for $\alpha \in (1, \infty)$ the variance is linear.
The central limit theorem fails for  $\alpha \in (1/2, 1)$ but it holds for $\alpha \in (1, \infty)$. 
\end{abstract}

\section{Introduction and main results}

\subsection{Hyperbolic random geometric graphs}

We  study in this paper the random geometric graph on the hyperbolic plane $H^{2}_{-1}$, as introduced
by Krioukov et al. \cite{ar:Krioukov}.   The standard Poincar\'e disk representation of $H^{2}_{-1}$ is
the open unit disk ${\cal D} := \{(u,v) \in \mathbb{R}^2 \ : \ u^2+v^2 < 1 \}$ equipped with the hyperbolic (Riemannian) metric $d_H$ given by $ds^2 = {4}~{du^2 + dv^2\over (1-u^2-v^2)^2}$.
Recall that the arclength of the boundary of a disk ${\cal D}_r \subset \cal D$ of radius $r$ and centered at the origin is
$2\pi \sinh (r)$, whereas the area of  ${\cal D}_r$ is $2\pi (\cosh ( r) - 1)$.

Given $\nu  \in (0, \infty)$  a fixed constant and a natural number $n > \nu$, we let 
$$
R : = 2 \log (n/\nu),  
$$
 i.e.,  $ n = \nu \exp(R /2)$.
For every $\alpha \in (0, \infty)$, consider the probability density function
\begin{equation} \label{eq:pdf}
 \rho_{\alpha, n} (r): = \begin{cases}
\alpha {\sinh  (\alpha r) \over \cosh (\alpha R ) - 1} & 0\leq r \leq R \\
0 & \mbox{otherwise}
\end{cases}.
\end{equation}
Let $\theta$ be uniformly distributed on $(-\pi, \pi]$.
When $\alpha =1$ the distribution  of $(r, \theta)$ given by  ~\eqref{eq:pdf} is the uniform distribution on $\D$ under the  metric $d_H$.
For general $\alpha \in (0, \infty)$ Krioukov et al.~\cite{ar:Krioukov} call this the \emph{quasi-uniform} distribution on $\D$, since it arises as the projection of
the uniform distribution on a disc of hyperbolic radius $R$ in $H^2_{-\alpha^2}$, the hyperbolic plane having curvature $-\alpha^2$ and equipped with
the metric $\frac{4}{\alpha^2} ~{du^2 + dv^2\over (1-u^2-v^2)^2}$.


Denote by $\kappa_{\alpha, n}$ the Borel measure on $\D$ given by
\begin{equation} \label{eq:kappa-measure}
\kappa_{\alpha, n} (A) := \frac{1}{2\pi} \int_A \rho_{\alpha, n} (r) dr d \theta,
\end{equation}
where $A$ is a Borel subset of  $\D$.
We let ${\cal P}_{\alpha, n}$ denote the Poisson point process  on $\D$ with intensity measure $n \kappa_{\alpha, n}$.
Denote by $(\Omega_n,\mathcal{F}_n, \mathbb{P}_n)$ the probability space on which the point process ${\cal P}_{\alpha, n}$ is realised.
Let $\mathbb{E}:= \mathbb{E}_n$ denote expectation with respect to $\mathbb{P}:= \mathbb{P}_n$.

We join two points in ${\cal P}_{\alpha, n}$  with an edge if and only if they are within hyperbolic distance $R$ of each other.
The resulting {\em hyperbolic random geometric graph on $\D$}  is denoted by $\Gnan:= \mathcal{G}_{\alpha, n, \nu}$.
Figure~\ref{fig:tube} illustrates the disc ${\cal B}_R(v) $ of radius $R$ centered at  $v \in \D$.
\begin{figure}
\begin{center}
\includegraphics[scale=0.45]{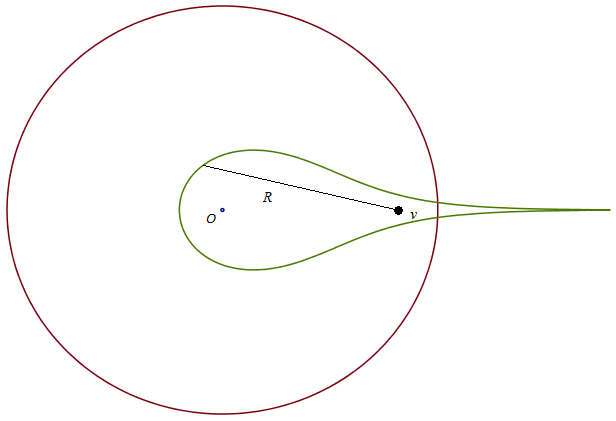}
\includegraphics[scale=0.45]{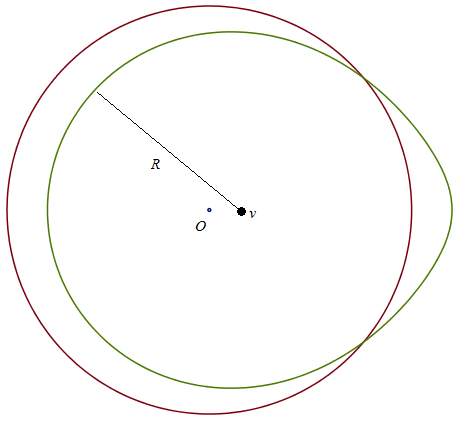}
\caption{The disc ${\cal B}_R(v) $ around the point $v \in \D$.}
\label{fig:tube}
\end{center}
\end{figure}
An equivalent construction of   $\Gnan$ goes as follows.
Given $\alpha \in (0, \infty)$ and $r \in (0, \infty)$, let ${\cal D}_{r, \alpha}$
 be a disc of radius $r$ in $H^2_{-\alpha^2}$. Consider  the Poisson point process
having uniform intensity density on ${\cal D}_{R, \alpha}$.
The points are projected onto ${\cal D}_{R, 1}$, preserving
polar coordinates, and the hyperbolic geometric graph on ${\cal D}_{R, \alpha}$ is created by putting an edge
between the points of the Poisson point process
whose projections are distant at most $R$.  The projection of this graph onto ${\cal D}_{R, 1}$ is $\mathcal{G}_{\alpha, n, \nu}$.


When  ${\cal P}_{\alpha, n}$ is replaced by $n$ i.i.d. random variables having density  $\rho_{\alpha, n}(r)/2 \pi $, we obtain the model of  Krioukov et al. ~\cite{ar:Krioukov}.
The underlying hyperbolic geometry gives rise to a power-law degree distribution tuned by the parameter $\alpha$,
whereas the parameter $\nu$ determines the average degree of $\Gnan$ \cite{ar:Krioukov}.  The model realises the assumption that there are intrinsic hierarchies in a complex network that induce a tree-like structure. This set-up provides a geometric framework describing  the inherent inhomogeneity of  complex networks and suggests that the geometry of complex networks is hyperbolic.

The graph $\Gnan$ also arises as a cosmological model.  As noted in~\cite{ar:Krioukov2},  the higher-dimensional analogue of  $\Gnan$ asymptotically coincides as $n \to \infty$ with the graph encoding the large scale causal structure of the de Sitter spacetime representation
of the universe.  Roughly speaking, the latter graph is obtained by
sprinkling Poisson points in de Sitter spacetime (the hyperboloid) and then joining two
points if they lie within each other's light cones. The light cones are then mapped to the hyperbolic plane, where they are approximated (for large times) by the hyperbolic balls of a certain radius (see Fig. 2 in~\cite{ar:Krioukov2}). Graph properties of $\Gnan$ thus yield information about the  causal structure of de Sitter spacetime.

\subsection{Main results}

For any $p \in \D$ we let $r(p)$ denote its radius (hyperbolic distance to the origin) and  
$$
y(p):= R - r(p)
$$
its defect radius. Given a  point process $\mathcal{P}$ on $\D$ and $p \in \mathcal{P} \cap \D$, we say that $p$ is {\em isolated}  with respect to $\mathcal{P}$  if and only if there is
no $p' \in \mathcal{P}$, $p'\not = p$, such that $d_H(p, p') \leq R$.  We say that $p$ is {\em extreme}  with respect to $\mathcal{P}$ if and only if there is no
 $p' \in \mathcal{P}$, $p'\not = p$,  such that $d_H(p, p') \leq R$ and $y(p') \in [0, y(p)).$

Given $p \in \mathcal{P}$, define the score $\xi^{iso}(p, \mathcal{P})$  to be $1$ if $p$ is isolated  with respect to $\mathcal{P}$ and zero otherwise.  Likewise, define $\xi^{ext}(p, \mathcal{P})$ to be $1$ if $p$ is extreme with respect to $\mathcal{P}$ and zero otherwise.
Our  main goal is to establish the limit theory for the number of isolated and extreme points in $\Gnan$, given respectively by
$$
S^{iso}({\cal P}_{\alpha, n}):= \sum_{p \in {\cal P}_{\alpha, n}} \xi^{iso}(p, {\cal P}_{\alpha, n})
$$
and
$$
S^{ext} ({\cal P}_{\alpha, n}):= \sum_{p \in {\cal P}_{\alpha, n}} \xi^{ext}(p, {\cal P}_{\alpha, n}).
$$
Isolated vertices are well-studied in the setting of Euclidean graphs, where they feature  in the connectivity
properties of certain random graph models. The paper \cite{Pen2016} elaborates on this when the graph in question is either the geometric graph on i.i.d. points in $[0,1]^d$ or even a soft version of this graph. In the cosmological set-up~\cite{ar:Krioukov2}, in the large time limit, isolated points are  precisely 
those whose past and future light cones are empty, i.e., the set of points neither accessible by the past nor having access to the future. Extreme points are those whose future light cones are empty, i.e.,  points which do not causally influence other points.

Extreme points are the analog of maximal points of a sample, of broad interest in computational geometry, networks, and analysis of linear programming. Recall that if $K \subset \mathbb{R}^d$ is a cone with apex at the origin of $\mathbb{R}^d$, then given ${\cal X} \subset \mathbb{R}^d$ locally finite, $x \in {\cal X}$ is $K$-maximal if $(K \oplus x) \cap {\cal X} = x$.  If ${\cal X}$ is an i.i.d. sample uniformly distributed on a smooth convex body $B$ in $\mathbb{R}^d$ of volume $1$, $n := \text{card}({\cal X})$,  then both the expectation and variance of the  number of maximal points in ${\cal X}$ are asymptotically $\Theta(n^{(d-1)/d})$, the order of the expected number of points close to the boundary of $B$~\cite{Yukich}.
In the present paper, the expectation and variance of the number of extreme  points are shown to grow linearly with $n := \text{card}({\cal X})$, which likewise is of the order of the expected number of points close to the boundary of $\D$.

The second order limit theory for the number of isolated points is altogether different. 
Our first main result  shows that the growth rates of the variance of  $S^{iso}({\cal P}_{\alpha, n})$ decrease with increasing $\alpha \in (1/2, \infty)$ and undergo a double jump when $\alpha$ crosses $1$.
Moreover, there is a logarithmic correction at $\alpha = 1$. 
 The variance grows faster than the expectation for $\alpha \in (1/2, 1]$, but it is always sub-quadratic with respect to input size. The asymptotics for the range $\alpha \in (1/2, 1]$  contrast markedly  with the  second order limit behavior of isolated points in the random geometric graph in the Euclidean plane ~\cite{bk:Penrose}, where asymptotics grow linearly with input.
This phenomenon, which gives rise to non-standard renormalization growth rates, appears to be linked to the high connectivity properties of $\Gnan$ for small $\alpha$, as described in Section \ref{properties}.

The limit constants appearing in our first and second order results \eqref{expected-iso}, \eqref{muformula}, and \eqref{Varext} are given in terms of expectations and covariances of scores involving isolated and extreme points of a Poisson point process on the upper half-plane, which appears to be a natural setting for studying such problems. Put $\gamma:= 8 \nu \alpha/ \pi(2 \alpha -1)$.

\begin{theorem} \label{prop:var_Shat}
We have for all $\alpha \in (1/2, \infty)$
\be \label{expected-iso}
\lim_{n \to \infty} \frac{  \mathbb{E} [ S^{iso}({\cal P}_{\alpha, n})] } { n} = 2 \alpha \int_0^{\infty}  \exp
\left(- \gamma e^{y/2}  \right) \exp ( - \alpha y)dy,
\ee
and
\be \label{Variso}
 \Var [ S^{iso}({\cal P}_{\alpha, n}) ] = \begin{cases}
\Theta(n^{3-2\alpha}) & \alpha \in ( \frac{1} {2}, 1) \\
\Theta(nR)  = \Theta(n \log n)  & \alpha =1 \\
\Theta(n) & \alpha \in (1, \infty)
\end{cases}.
\ee
\end{theorem}

On the other hand, for all $\alpha \in (1/2, \infty)$, the expectation and variance asymptotics for the number of extreme points exhibit linear scaling in $n$, that is to say the renormalization is the standard one in stochastic geometric models.

\begin{theorem} \label{extVar}
We have for all $\alpha \in (1/2, \infty)$
\begin{equation} \label{muformula}
\lim_{n \to \infty} \frac{ \mathbb{E} [S^{ext}({\cal P}_{\alpha, n}) ]  }{n} = \mu,
\end{equation}
and
\be \label{Varext}
\lim_{n \to \infty} \frac{ \Var [S^{ext}({\cal P}_{\alpha, n})] } {n} = \sigma^2,
\ee
where $\mu, \sigma^2 \in (0, \infty)$ are given by ~\eqref{eq:extreme_mu} and \eqref{sigmaformula}, respectively,
 below.
\end{theorem}

Denote by $N$  the standard normal random variable with mean zero and variance one. One might expect that $S^{iso}({\cal P}_{\alpha, n})$, after centering and renormalizing, converges in distribution to $N$ for {\em all}  $\alpha \in (1/2, \infty)$.  The next result shows that this is false.

\begin{theorem} \label{mainCLT}
As $n \to \infty$, for any $\alpha \in (1, \infty)$ we have
\begin{equation} \label{CLTiso}
\frac{ S^{iso}({\cal P}_{\alpha, n}) - \mathbb{E} [S^{iso}({\cal P}_{\alpha, n}) ] } { \sqrt {\Var [S^{iso}({\cal P}_{\alpha, n})]} } \tod N.
\end{equation}
The limit \eqref{CLTiso}  fails for $ \alpha \in (1/2,1)$.
As $n \to \infty$, for any $\alpha \in (1/2, \infty)$ we have
\begin{equation}\label{CLText}
\frac{ S^{ext}({\cal P}_{\alpha, n}) - \mathbb{E} [S^{ext}({\cal P}_{\alpha, n}) ] } { \sqrt {\Var [S^{ext}({\cal P}_{\alpha, n})]} } \tod N.
\end{equation}
\end{theorem}

The proofs of these  results depend on mapping the point process ${\cal P}_{\alpha, n}$ in the disc $\D$ to a Poisson point process hosted by a rectangle in the upper half-plane. This mapping,  introduced in~\cite{ar:FounMull2017},   transforms the graph $\Gnan$ into an analytically more tractable graph, as seen in Section 2.  In particular, it facilitates the evaluation of the probability content of the intersection of two radius $R$ balls, which is essential to evaluating the covariance of scores at distinct points.
Variance calculations are based on the covariance formula for two points. In the case of isolated points, the upper bound is based on the Poincar\'e inequality.  The lower bound, which turns out to be tight, is
based on a careful analysis of the intersection of the balls of radius $R$ around two typical points. We refer the reader to Sections 3 and 4 for all details.

The determination of variance asymptotics for $S^{ext}( {\cal P}_{\alpha, n})$ is handled by extending stabilization methods.  That is, for a given point $p$, we define a radius of stabilization $R^{\xi}:= R^{\xi^{ext}}$ for $\xi^{ext}$,  in the sense that points at distance farther than  $R^{\xi}$ from $p$ do not affect the property of $p$ being extreme.  We show that the covariance of two points (depending on their interpoint distance and heights) is rather small.  By stabilization, the covariance converges to the covariance of two points in the infinite hyperbolic plane.  We show in Section 5 that though the constants describing the tail behavior of $R^{\xi}$ grow exponentially fast with the height of $p$, it is still possible to extract an explicit integral formula for the scaled variance.

To prove the asymptotic normality \eqref{CLTiso} we use the Poincar\'e inequality for  Poisson functionals  \cite{ar:LPS16}, which bounds the Wasserstein distance in terms of first- and second-order difference operators.  When $\alpha \in (1, \infty)$ there is a high probability event on which these difference operators may be controlled, as vertices of high degree are fewer in this regime.  For $\alpha \in (1/2, 1)$, conditional under the likely event of having no vertex  sufficiently close to the origin  (equivalently having no vertex of significantly high degree), the variance is much smaller than the unconditional variance, and the convergence to the standard normal fails in this regime. Intuitively,  vertices close to the origin generate radius $R$ balls which cover a relatively large part of $\D$ and any vertex lying therein will not be isolated. Surprisingly, when $\alpha \in (1,\infty)$, this phenomenon stops having a significant effect and, therefore one can deduce the asymptotic normality of $S^{iso}({\cal P}_{\alpha, n})$.

The case of extreme points is different, since the extremality status of a point is  influenced only by points of larger radius lying in the ball of radius $R$ around it. This region is typically quite small and makes the
corresponding score functions almost independent.
To prove the central limit theorem \eqref{CLText}, we cut the plane into rectangles and define a dependency graph on the vertex set of such rectangles, so that no points in non-adjacent vertices in this dependency graph can be connected, and we use the central limit theorem of
 Baldi and Rinott  \cite{BR}.

\vskip.3cm

\noindent{\bf Remarks.} (i) We are unaware of  results treating the limit theory for statistics
of $\Gnan$ in the regime $\alpha \in (1/2, 1).$
The paper ~\cite{OY} establishes variance asymptotics and asymptotic normality for the number of copies of trees in $\Gnan$ with at least two vertices, but the authors require  $\alpha \in (1, \infty)$,   save for when counting trees close to the boundary of $\D$.
 The methods of ~\cite{OY} do not appear to treat the limit theory of $S^{iso}({\cal P}_{\alpha, n})$ and $S^{ext}({\cal P}_{\alpha, n})$, as $n \to \infty$.

\vskip.3cm

\noindent (ii) It is  an interesting  problem whether the number of isolated points asymptotically follows a normal law when $\alpha = 1$.  The methods in this paper do not apply, as they
give estimates that are useless. To deal with this case, one likely needs a more detailed treatment of the variance of $S^{iso} ({\cal P}_{\alpha, n})$, giving not only the order of magnitude but the multiplicative constant. 

\vskip.3cm

\noindent (iii) As seen in~\cite{ar:BlaFriedKrohmer18,ar:FriedKrohmer15},  the expected number of cliques of order $k\geq 2$ is $\Theta(n)$ if $\alpha \in (1 - 1/k, \infty)$, whereas the expected number of cliques of order $k \geq 3$ is $\Theta(n^{(1 - \alpha) k })$ if $\alpha \in (1/2, 1 - 1/k).$

\vskip.3cm

\noindent (iv) In dimension $d \geq 3$ we expect that the central limit theorem \eqref{CLText} holds for all $\alpha \in (1/2, \infty)$, where $R$ is suitably defined so that the average degree of the random graph is $\Theta(1)$.
 It is unclear for which $\alpha$ the central limit \eqref{CLTiso} holds in dimension $d \geq 3$.

\vskip.3cm

\noindent (v) It is unclear whether there exists a limiting distribution for $S^{iso}({\cal P}_{\alpha, n})$
when $\alpha \in (1/2,1)$. As we are going to see in Section 6, the variance of  $S^{iso}({\cal P}_{\alpha, n})$ is highly sensitive when
conditioning  on the high probability event of having no points within a certain radius in $\D$. It is plausible that
a central limit theorem holds in such a conditional space.

\vskip.3cm

\noindent (vi)  As elaborated upon in the next subsection, the degree distribution in $\Gnan$ follows a power-law with exponent $2 \alpha + 1$ when $\alpha \in (1/2, \infty)$.  In particular, the degree distribution belongs to $L^2$ when $\alpha \in (1, \infty)$.  It would be worthwhile to better understand the connection between  asymptotic normality of the number of isolated vertices and moments of the degree distribution. 

\vskip.3cm

\noindent{\bf Notation and terminology.}
We say that a sequence of events $E_n \in \mathcal{F}_n$ occur \emph{asymptotically almost surely (a.a.s.)}
if $\lim_{n\to \infty}  \mathbb{P}(E_n) =1$.
Given  $a_n$ and $b_n$  two sequences of positive real numbers,
we  write $a_n \sim b_n$ to denote that $a_n / b_n \rightarrow 1$, as $n\rightarrow \infty$.

\subsection{Degree and connectivity properties of the graph $\Gnan$} \label{properties}

For $\alpha \in (1/2, \infty)$, the tails of the distribution of the degrees in
$\Gnan$  follow a power law with exponent $2\alpha  + 1$; see Krioukov et al.~\cite{ar:Krioukov}.  This was verified rigorously
in~\cite{ar:Gugel}.  
For $ \alpha \in (1/2, 1)$, the exponent is between 2 and 3, as is the case in a number of networks arising  in applications (see for example~\cite{BarAlb} for a list of experimental observations).
The paper  ~\cite{ar:Krioukov} observes that the average degree of $\Gnan$ is determined through the parameter $\nu$
for $\alpha \in (1/2, \infty)$. This was rigorously shown in~\cite{ar:Gugel}. In particular, they show that
the average degree tends to  $8 \alpha^2 \nu / \pi(2\alpha-1)^2$ in probability.
 However, when $\alpha \in (0, 1/2]$, the average degree tends to infinity as $n\to \infty$.
 Thus, in this sense, the regime $\alpha \in (1/2,\infty)$ corresponds to the thermodynamic regime in the
 context of random geometric graphs on the Euclidean plane~\cite{bk:Penrose}.
In~\cite{ar:Foun13+} the degree distribution of a \emph{soft} version of this model is determined. Here, pairs of points that are distant at most $R$ are joined with some probability that is not identically equal to 1.

When $\alpha$ is small, one expects  more points of ${\cal P}_{\alpha, n}$ to be near the origin and one may expect increased graph connectivity. The paper
~\cite{BFMgiantEJC} establishes 
that $\alpha = 1$ is the critical point for the emergence of a giant component in $\Gnan$. In particular, when $\alpha \in (0,1)$, the fraction of the vertices
contained in the largest component is bounded away from 0 a.a.s.~\cite{BFMgiantEJC}, whereas if $\alpha \in (1, \infty)$, the largest
component is sublinear in $n$ a.a.s.
For $\alpha = 1$, the component structure depends on $\nu$. If $\nu$ is large enough,
then a giant component exists a.a.s., but if $\nu$ is small enough, then a.a.s. all components have sublinear size ~\cite{BFMgiantEJC}. Figures~\ref{fig:simulations_I} and~\ref{fig:simulations_II} illustrate these transitions.

\begin{figure}[h!]
\begin{center}
\includegraphics[scale=0.4]{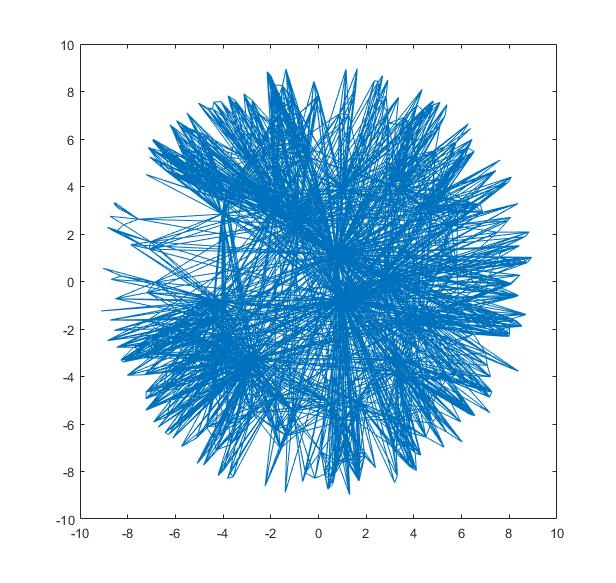}
\includegraphics[scale=0.4]{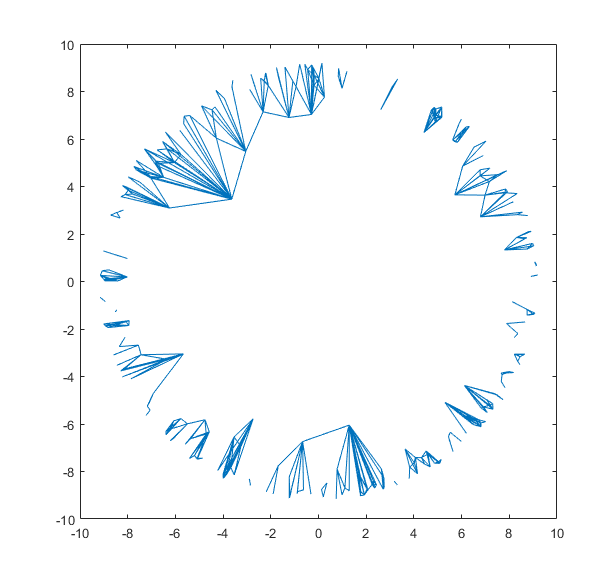}
\caption{Samples of $\Gnan$ for $n=300$, $\nu = 3$ and $\alpha =0.7$ and $2$, respectively, from left to right. Average degree decreases as $\alpha$ increases.}
\label{fig:simulations_I}
\end{center}
\end{figure}

\begin{figure}[h!]
\begin{center}
\includegraphics[scale=0.4]{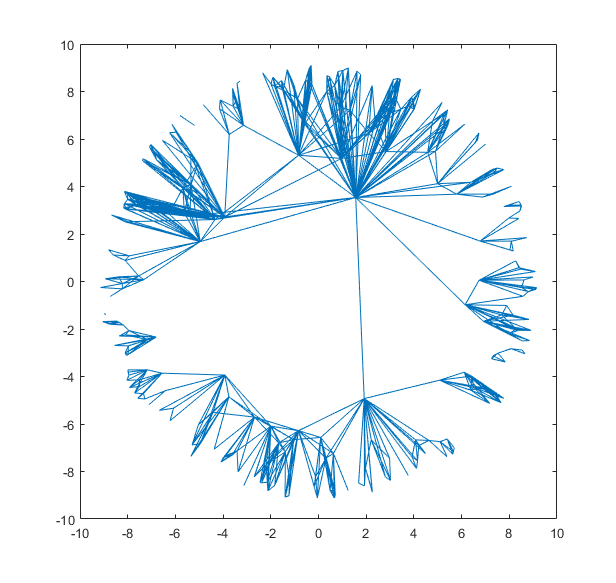}
\includegraphics[scale=0.4]{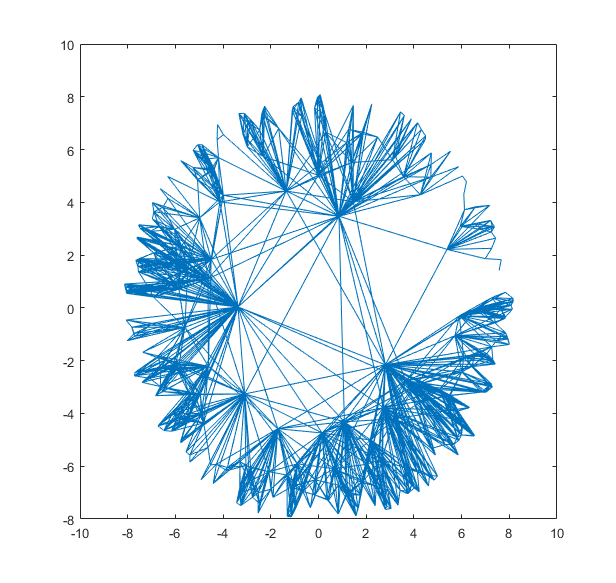}
\caption{Samples of $\Gnan$ for $n=300$, $\alpha = 1$ and $\nu =3$ and $5$, respectively, from left to right. Average degree increases as $\nu$ increases.}
\label{fig:simulations_II}
\end{center}
\end{figure}

 The paper ~\cite{ar:FounMull2017} strengthens these results and shows the fraction of vertices belonging to the largest component converges in probability to a constant which depends on $\alpha$ and $\nu$. Furthermore, for $\alpha =1$, there exists a critical value $\nu_0 \in (0, \infty)$ such that when $\nu$ crosses $\nu_0$ a giant component emerges a.a.s. \cite{ar:FounMull2017}. For $\alpha \in (0, 1)$ the second largest component has polylogarithmic order a.a.s.;  see  \cite{KiwiMit} and \cite{KiwiMit2017+}.
 For $\alpha \in (0, 1/2)$ we have that $\Gnan$ is a.a.s. connected, whereas $\Gnan$ is disconnected for
$\alpha \in (1/2, \infty)$  ~\cite{BFMconnRSA}.  For $\alpha =1/2$, the probability of
connectivity tends to a certain constant  given explicitly in ~\cite{BFMconnRSA}.

Apart from the component structure, the geometry of this model has been also considered.
In~\cite{KiwiMit} and~\cite{ar:FriedKrohmerDiam} polylogarithmic upper bounds on the diameter are
shown. These were improved shortly afterwards in~\cite{ar:MullerDiam} where a logarithmic upper bound on the diameter is established.
Furthermore, in~\cite{ar:AbdBodeFound} it is shown that for $\alpha \in (1/2,1)$ the largest component has doubly logarithmic typical distances and it forms what is called an ultra-small world.

\section{Auxiliary results}

\subsection{Approximating a hyperbolic ball}

We characterize when two points in  $\D$ are within hyperbolic distance $R$. In particular the next lemma approximates hyperbolic balls by analytically more tractable sets, reducing a statement about hyperbolic distances between two points to a statement about their relative angle.
For a point $p \in \D$, we let $\theta (p) \in (-\pi, \pi]$ be the angle
$\widehat{p O q}$ between $p$ and a (fixed) reference point  $q \in \D$ (where positive angle is determined 
by moving from $q$ to $p$ in the
anti-clockwise direction).
For points $p,p'\in \D$ we denote by $\theta (p,p')$ their relative angle:
\begin{equation*}
\theta (p,p') := \min \{ |\theta(p) - \theta(p')|, \ 2\pi - |\theta (p) - \theta (p')| \}.
\end{equation*}
For any $p \in \D$ recall that $r(p)$ denotes its radius (hyperbolic distance to the origin) whereas  $y(p):= R - r(p)$, or
more succinctly $y := R - r$.  Thus for $p \in \D$, we shall write
$p := (\theta(p), y(p))$. The hyperbolic law of cosines relates the relative angle $\theta (p,p')$ between two points with their hyperbolic
distance:
\begin{equation} \label{RE}
\begin{split}
\cosh (d_H(p,p'))  =
& \cosh (r(p) ) \cosh (r(p'))  - \sinh (r(p)) \sinh (r(p'))
\cos ( \theta (p,p') ).
\end{split}
\end{equation}
For $r,r' \in [0,R]$, we let $\theta_R (r,r')$ be the value of $\theta (p,p') \in [0,\pi]$ satisfying \eqref{RE}, having set $d_H(p,p') = R$, for two points $p,p' \in \D$ with $r(p)=r$ and $r(p')=r'$.
As $\cos (\cdot )$ is decreasing in $[0,\pi]$, it follows that
$d_H (p,p') \in [0,R]$ if and only if $\theta (p,p') \leq \theta_R (r(p),r(p'))$.

When $y(p)$ and $y(p')$ are not too large, our next result  estimates $\theta_R(r(p),r(p'))$ as a function
of $y(p)$ and $y(p')$. To prepare for mapping $\D$ to a rectangle in $\mathbb{R} \times \mathbb{R}^+$ having length proportional to $\frac12 e^{R/2}$, we re-scale $\theta_R(r(p),r(p'))$ by a factor of $\frac12 e^{R/2}$.
The following lemma appears in a stronger form in ~\cite{ar:FounMull2017}.  Here and elsewhere we put
\be \label{defH}
H := 4 \log R.
\ee

The proof of the next lemma is in Section~\ref{sec:basic_lemmas}.
\begin{lemma}\label{lem:relAngle_generic}  Given $p$ and $p'$ in $\D$,  $y:= y(p) := R - r$, and $y':= y(p') := R - r'$, with
$r, r' \in [0,R]$ we set
$$
\Delta(r,r') := \frac12 e^{R/2} \theta_R(r, r') =   \frac12 e^{R/2} \arccos\left( \frac{ \cosh r \cosh r' - \cosh R} { \sinh r \sinh r' } \right).
$$
For every $\eps \in (0, 1/3)$ there exists a $C:= C(\eps) \in (0, R)$ such that the following holds.
\begin{enumerate}
\item[ (i)] If  $r + r' \in (R + C,2R]$, i.e., if  $y(p) + y(p') \in [0, R- C)$, then
\begin{equation}\label{eq:asymp1.0}
(1- \eps) e^{\frac12(y+y')} \leq \Delta( r, r') \leq (1 + \eps) e^{\frac12(y+y')}.
\end{equation}
\item[(ii)] If $r,r' \in [R- H,R]$, i.e., if $y(p), y(p') \in [0, H]$, then 
\be \label{lam1}
\Delta (r,r' ) = (1 + \lambda_n (r,r')) e^{\frac12(y+y')}
\ee
where $\lambda_n (r,r') = o(1)$ as $n\to \infty$, uniformly over all $r,r'\in [R-H,R]$.
\item[(iii)] In part (i) above, one can take $\eps := \eps (n) \to 0$ and $C := C (n) \to \infty$ as $n\to \infty$.
In  particular we may relate $\eps$ and $C$ by $\eps = \Theta(e^{-C})$.
\end{enumerate}
\end{lemma}

Recall that $\Dis{r}$ denotes the disc of (hyperbolic) radius $r$ centered at the origin $O$. For any $h \in [0,R)$ we let $\A_{h}$ denote the
annulus  $\D \setminus \Dis{R-h}$. Throughout we shall use caligraphic letters to denote subsets of $\Dis{R}$.
For $p \in \D$ we let $\HBall{p}:= {\cal B}_R(p) \cap \D$.
We now approximate $\HBall{p}$ whenever $r(p) \in (C,R]$, with $C$ as in Lemma \ref{lem:relAngle_generic}.  This goes as follows.

By the triangle inequality, given $p \in \Dis{R}$, any point with defect radius at most $y(p):= R - r(p)$ is also within distance $R$ from $p$.
To  approximate $\HBall{p}$ from above, we will take a superset of this set, namely the set of
points of radius at most $y(p) - C$,  with  $C:=C(\eps)$ as in Lemma~\ref{lem:relAngle_generic}.  We set $\OneUp := 1 + \eps$ and $\OneLow:= 1-\eps$
and put
\begin{align*} \label{superset}
\HBallUp{p} & := \HBallUp{p, \eps} \\
& := \{p' \in \D : y(p') + y(p) \in [0, R - C),  \  \theta (p,p') \leq \OneUp \cdot 2e^{\frac{1}{2}(y(p)+y(p')- R)}\} \nonumber \\
& \ \ \ \ \ \ \ \ \cup \{p' \in \D : y(p') + y(p) \in [R - C, 2R] \}
\end{align*}
and
\begin{align*} 
\HBallLow{p} & := \HBallLow{p, \eps} \\
&:= \{p' \in \D \ : y(p') + y(p) \in [0, R - C),  \  \theta (p,p') \leq \OneLow \cdot 2e^{\frac{1}{2}(y(p)+y(p')- R)}\}.
\end{align*}
For $\eps \in (0, 1/3)$, $C: = C(\eps) >0$ as in Lemma~\ref{lem:relAngle_generic}(i),
and  $p \in \D$ with  $r(p) \in (C,R]$, the inequality \eqref{eq:asymp1.0} yields
the following inclusions:
\begin{equation} \label{eq:ball_inclusion}
\HBallLow{p} \subset \HBall{p} \subset \HBallUp{p}.
\end{equation}

In our calculations for $\mathbb{E} [\xi^{ext}(p,{\cal P}_{\alpha, n})]$ we will need the 
truncated subset of $\HBall{p}$ consisting of points with height coordinates at most $y(p)$, namely 
\be \label{trunc}
\HD{p} := \{ p' \in \HBall{p} \ : \ y(p') \leq y(p)\}.
\ee
A point $p\in \D \cap {\cal P}_{\alpha, n}$ is extreme with respect to  ${\cal P}_{\alpha, n}$ if and only if $\HD{p} \cap  {\cal P}_{\alpha, n} = \{p\}$.
Lemma~\ref{lem:relAngle_generic}(ii) implies that if $y(p) \in [0, H]$, then
\begin{equation}\label{eq:HD_def}
\HD{p} := \{p': \ y(p') \leq y(p), \ \theta (p,p') \leq (1+\lambda_n (p,p')) e^{\frac12 (y(p)+ y(p') - R)} \}.
\end{equation}

\subsection{Properties of $\Gnan$}

The density of the defect radius is close to the exponential density with parameter $\alpha$.  The proof of this fact is based on elementary algebraic manipulations and appears in  Section~\ref{sec:basic_lemmas}.

\begin{lemma}\label{lem:dist}
Let $\bar\rho_{\alpha, n} (y):= \rho_{\alpha, n} (R-y), \ y \in [0,R),$ be the probability density  of the defect radii.
For all  $\alpha \in (1/2, \infty)$ we have
\be \label{approx}
|\bar{\rho}_{\alpha, n}(y) - \alpha e^{-\alpha y}| < \frac{2\alpha }{e^{\alpha R}-2} = O(n^{-2\alpha}), \ y \in [0,R].
\ee
\end{lemma}

One does not expect to observe isolated and extreme points close to the origin.  
The following lemma makes this precise and shows  that the isolated and extreme points
a.a.s. have  defect radii less than $H:=4\log R$.
\begin{lemma} \label{lem:no_isolated_high_points}
Let $p \in \D$.  If $y(p) \in (H, R]$ then for all $\alpha \in (1/2, \infty)$ we have
$$ \mathbb{E}[ {\xi ^{iso}}(p, {\cal P}_{\alpha, n} \cup \{p\})] = n^{-\Omega(\log n)}, \ \   \mathbb{E} [{\xi ^{ext}}(p, {\cal P}_{\alpha, n} \cup \{p\})] = n^{-\Omega(\log n)}.$$
\end{lemma}
\begin{proof} Let $\eps \in (0, 1/3)$ and $C:= C(\eps)$ be as in Lemma~\ref{lem:relAngle_generic}.
First assume that $y(p) \in [0, R-C)$.
We may bound $ \mathbb{E}[ \xi^{iso}(p, {\cal P}_{\alpha, n}\cup \{p\})] $  by the probability that
$\HBall{p} \cap \A_{H}$ is empty.
Using the first inclusion in ~\eqref{eq:ball_inclusion} and recalling the definition of $\kappa_{\alpha, n}$ at \eqref{eq:kappa-measure}, we have (for $n$ sufficiently large):
\begin{equation*}
\begin{split}
 \mathbb{E} [\xi ^{iso}(p, {\cal P}_{\alpha, n}\cup \{p\})] & \leq \exp \left(  -n \kappa_{\alpha, n}
 (\HBall{p} \cap \A_{H} )\right) \\
&= \exp \left( - n \frac{2(1-\eps)}{2\pi}  e^{-R/2} e^{y(p)/2} \int_0^{4\log R} e^{y/2} \bar{\rho}_{\alpha, n}(y) dy \right) \\
& <  \exp \left( - \frac{\nu (1 -2\eps)}{\pi} e^{y(p)/2} \int_0^{4\log R} e^{(1/2 - \alpha ) y}dy \right)  \\
&= \exp \left(-  \Omega (e^{2 \log R})\right) = \exp \left( - \Omega (R^2  )\right) = n^{-\Omega(\log n)},
\end{split}
\end{equation*}
where the inequality follows by Lemma~\ref{lem:dist}.

Suppose now that $y(p) \in [R-C, R]$, i.e, $r(p) \in [0,C]$.  By the triangle inequality  any point of $\D$ of radius less than $R-C$ is within hyperbolic distance $R$ from $p$.
This implies that
\begin{equation*}
 \mathbb{E} [\xi ^{iso}(p, {\cal P}_{\alpha, n}\cup \{p\})]   \leq \exp \left(  -n \kappa_{\alpha, n} (\Dis{R-C})\right).
\end{equation*}
Recalling \eqref{eq:kappa-measure} we obtain
\begin{equation*}
\begin{split}
 \kappa_{\alpha, n} (\Dis{R-C}) = \int_0^{R-C} \alpha \frac{\sinh \alpha r}{\cosh(\alpha R) - 1} dr = \frac{\cosh(\alpha (R-C) - 1}{\cosh(\alpha R) - 1} = \Theta (1),
\end{split}
\end{equation*}
whereby $ \mathbb{E} [ \xi ^{iso}(p, {\cal P}_{\alpha, n}\cup \{p\})] = \exp \left(  -\Omega (n ) \right).$

These upper bounds are also valid for $  \mathbb{E} [ \xi ^{ext}(p, {\cal P}_{\alpha, n}\cup \{p\})] $, with the exception of the last integral which would start from $r(p) \in (0, R]$ instead of from $r(p) = 0$. However, the asymptotic growth of this integral is still $\Theta (1)$.
\end{proof}

\subsection{Mapping $\D$ to $\mathbb{R}^2$}

To further simplify our calculations, we will transfer our analysis from $\D$ to $\mathbb{R}^2$, making  use of a mapping introduced in~\cite{ar:FounMull2017}.
We set
$$
I_n:= \frac{\pi}{2}e^{R/2} =  \frac{\pi}{ 2 \nu} \cdot n.
$$
For any subset $E \subseteq [0, R]$, define the rectangular domain
$$
{D} (E) :=  (-I_n, I_n] \times E
$$
and we put
\be \label{defiD}
D := D([0, R]) := (-I_n, I_n] \times [0,R].
\ee
For $p \in \D$, recall that we write $p:=(y(p), \theta(p))$, with $y(p)$ the defect radius and $\theta(p)$ the angle with respect to a reference point.
We re-scale the angle $\theta (p)$ by $2 e^{-R/2}$, setting $x(p) := \frac12 \theta (p) e^{R/2}$.
This defines the map $\Phi : \D \to {D}$, mapping $(\theta (p),  y(p))  \mapsto (x (p), y(p)) $.

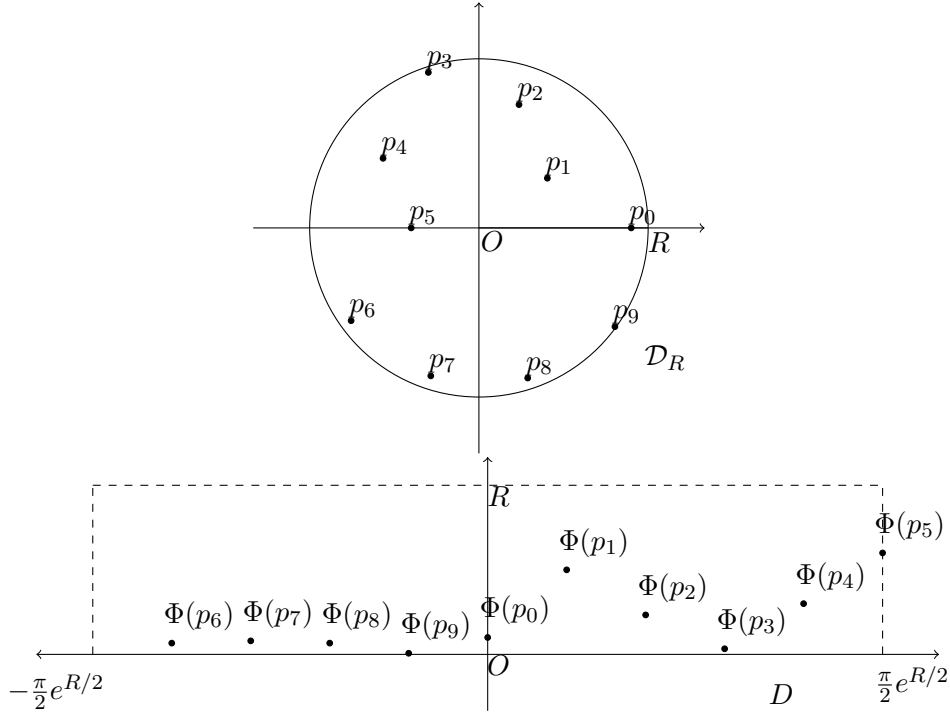
\begin{figure}
\begin{center}
\begin{tikzpicture}[scale=0.75]
\draw (0,0) circle [radius = 3]; 
\draw [->] (-4,0) -- (4,0);
\draw [->] (0,-4) -- (0,4);
\node [fill=none, below right, draw=none] at (3,-2) {$\mathcal{D}_R$};
\node [fill=none, below right, draw=none] at (-1:3) {$R$};
\node [fill=none, below right, draw=none] at (0.04,-0.04) {$O$};
\draw (0,0) -- (0:3);
\foreach \r [count=\a from 0] in {2.7,1.5,2.3,2.9,2.1,1.2,2.8,2.76, 2.8,2.98}
{\node[fill=none, above right, draw=none]  at (\a*36:\r) {$p_{\a}$};
\draw [fill = black] (\a*36:\r) circle [radius  =0.05];}
\end{tikzpicture}

\begin{tikzpicture}[scale=0.75]
\draw [->] (0,0) -- (8,0);
\draw [->] (0,0) -- (-8,0);
\draw [dashed] (7,0) -- (7,3);
\draw [dashed] (-7,0) -- (-7,3);
\draw [dashed] (-7,3) -- (7,3);
\draw [->] (0,-1) -- (0,3.5); 
\node[fill=none, below right, draw=none]  at  (0,3)  {$R$};
\node[fill=none, below right, draw=none]  at  (0,0)  {$O$};
\node[fill=none, below right, draw=none]  at  (7,0)  {$\frac{\pi}{2}e^{R/2}$};
\node[fill=none, below left, draw=none]  at  (-7,0)  {$-\frac{\pi}{2} e^{R/2}$};
\foreach \r [count=\a from 0] in {2.7,1.5,2.3,2.9,2.1,1.2}
{\node[fill=none, above right, draw=none]  at  (14*\a/10,3-\r)  {$\Phi(p_{\a})$};
\draw [fill = black] (14*\a/10,3-\r) circle [radius  =0.05];}
\foreach \r [count=\a from 6] in {2.8,2.76, 2.8,2.98}
{\node[fill=none, above right, draw=none]  at  (14*\a/10-14,3-\r)  {$\Phi(p_{\a})$};
\draw [fill = black] (14*\a/10-14,3-\r) circle [radius  =0.05];}

\node [fill=none, below right, draw=none] at (5,-0.5) {$D$};
\end{tikzpicture}
\end{center}
\caption{The mapping $\Phi : \mathcal{D}_R \to D$.}
\end{figure}

Put $\beta := {2\nu \alpha/ \pi}$.
The map $\Phi$ sends ${\cal P}_{\alpha, n}$ to the Poisson point process $\tilde{\cal P}_{\alpha, n}$ on ${D}$ with intensity density
\be \label{Pnav}
d \tilde{\cal P}_{\alpha, n}(x,y) = ( \beta e^{- \alpha y} + \epsilon_n) dy dx,  \  (x,y) \in {D} ,
\ee
where, recalling Lemma \ref{lem:dist}, we have $\epsilon_n =  O(n^{ - 2 \alpha}) = o(n^{-1})$, since $\alpha \in  (1/2, \infty)$.

The analogue of the relative angle is defined as follows.
For $x, x' \in ( -I_n ,  I_n]$, we let
$$
|x-x'|_{\Phi} := \min \left\{ |x-x'|, \ 2 I_n - |x -x'| \right\}.
$$
When considering the geometry of hyperbolic balls inside $\Rcal$, it will
be convenient to use arithmetic on the $x$-axis modulo $2I_n$.
In particular, for $x_1, x_2 \in (-I_n, I_n]$, we write $x_1 <_\Phi x_2$, if $x_1  < x_2$
and $|x_1-x_2|_\Phi = |x_1-x_2|$ or $x_1 > x_2$ and $|x_1-x_2|_\Phi =
2 I_n - |x_1 -x_2|$. This definition naturally extends to all other types of inequalities.
Also, for any $x_1, x_2 \in \mathbb{R}$, we write $x_1=_\Phi x_2$, if $x_1 = x_2 \mod 2I_n$.

\noindent {\bf Mapping balls in ${\cal D}$ to balls in $D$}.
We set
$$
B(p):=\Phi (\HBall{p} ),  \ \BallDown := \Phi (\HBallLow{p} ), \ \text{and} \ \BallUp := \Phi (\HBallUp{p} ).
$$

Thus, for  $p \in {D}$ with $y(p) \in [0,  R - C)$ and $\eps > 0$, we have
\begin{equation*}
B^{-}(p) : =   B^{-}(p, \eps) :=  \{ (x,y): \ y + y(p) \in [0, R- C), \ |x-x(p)|_{\Phi} <  \OneLow \cdot e^{\frac12 (y+ y(p))} \}
\end{equation*}
and
\begin{equation*}
B^{+}(p) : =   B^{+}(p, \eps) : = \{ (x,y): \ |x-x(p)|_\Phi < \OneUp \cdot e^{\frac12 (y+ y(p))}\} \cup \{ (x,y): \ y + y(p) \in (R- C, 2R] \}.
\end{equation*}
For $p \in \D$ with $y(p) \in [0, R - C)$ note that $\Phi$ transforms the set inclusion
 ~\eqref{eq:ball_inclusion} into
\begin{equation} \label{eq:inclusion_mapped}
B^{-}(p) \subset B(p)  \subset B^{+}(p).
\end{equation}

\medskip

\noindent
{\bf Approximating $S^{iso}({\cal P}_{\alpha, n})$ and $S^{ext}({\cal P}_{\alpha, n})$ on ${D}$.}
Let $\tilde{p} : = (x( \tilde{p}), y( \tilde{p}))$
be the image of $p$ by $\Phi$. For  $\tilde{p} \in \tilde{{\cal P}}_{\alpha, n}$ we define
$\tilde \xi^{iso}( \tilde{p}, \tilde{{\cal P}}_{\alpha, n}) = \xi^{iso}( p, {\cal P}_{\alpha, n})$. In other words,
\begin{equation*}
 \tilde{S}^{iso}( \tilde{{\cal P}}_{\alpha, n} ) := \sum_{\tilde{p}  \in \tilde{{\cal P}}_{\alpha, n} }
 \tilde \xi^{iso} ( \tilde{p} , \tilde{{\cal P}}_{\alpha, n} ) = S^{iso}({\cal P}_{\alpha, n} ).
\end{equation*}

\noindent
Regarding $S^{ext}( {\cal P}_{\alpha, n})$, recall from \eqref{eq:HD_def} that $p \in {\cal P}_{\alpha, n}$ is extreme if and only if
$\HD{p} \cap {\cal P}_{\alpha, n} = \{p \}$.  The image under $\Phi$ of the truncated ball $\HD{p}$
at \eqref{trunc} is
$$
D(y(\tilde{p})):= \{  (x,y) \in B( (x( \tilde{p}) , y(\tilde{p}))) : \ y  \in [0, y(\tilde{p})] \}.
$$
Note that $\xi^{ext}(p, {\cal P}_{\alpha, n}) = 1$ if and only if $D(y(\tilde{p})) \cap  \tilde{{\cal P}}_{\alpha, n} = \{\tilde{p}\}.$  For $\tilde{p} \in \tilde{{\cal P}}_{\alpha, n}$ define
$$
 \txi^{ext}(\tilde{p},  \tilde{{\cal P}}_{\alpha, n}) : = \begin{cases} 1 &  \tilde{{\cal P}}_{\alpha, n} \cap D(y( \tilde{p})) =  \{\tilde{p}\}  \\
 0   &  \tilde{{\cal P}}_{\alpha, n} \cap D(y( \tilde{p})) \neq \{\tilde{p}\}
 \end{cases}.
 $$
By definition, we thus have $\xi^{ext}(p, {\cal P}_{\alpha, n}) = \txi^{ext}(\tilde{p},  \tilde{{\cal P}}_{\alpha, n})$ and
$$
\tilde S^{ext}( \tilde{\cal P}_{\alpha, n}) =  \sum_{\tilde p \in \tilde{\cal P}_{\alpha, n} }  \tilde \xi^{ext}(\tilde p, \tilde {\cal P}_{\alpha, n}) =:  S^{ext}( {\cal P}_{\alpha, n}).
$$

From now on, when the context is clear, we write $p$ instead of $\tilde p$ for a generic point in $\tilde{{\cal P}}_{\alpha, n}$.

\begin{lemma} \label{lemm1} We have for all $\alpha \in (1/2, \infty)$
$$
\mathbb{E} [  \txi^{iso}( p, \tilde{{\cal P}}_{\alpha,n}\cup \{p\}) {\bf 1}(p \in\R{(H, R)}) ] = \exp( - \Omega(R^2)) = n^{-\Omega( \log n )},
$$
and similarly when  $\txi^{iso}$ is replaced by  $\txi^{ext}$.
\end{lemma}

\noindent{\em Proof.} This follows from Lemma \ref{lem:no_isolated_high_points}.  \qed

\vskip.5cm

We put
\begin{equation} \label{eq:tilde_S-iso_def}
\tilde{S}^{iso}_H ( \tilde{\cal P}_{\alpha,n} )  := \sum_{p  \in \tilde{\cal P}_{\alpha,n} \cap D([0,H])  }
\txi^{iso}( p , \tilde{\cal P}_{\alpha,n} )
\end{equation}
and
\begin{equation}
  \tilde{S}^{ext}_H ( \tilde{\cal P}_{\alpha,n} ) := \sum_{p  \in \tilde{\cal P}_{\alpha,n} \cap D([0,H])  }
\txi^{ext}( p , \tilde{\cal P}_{\alpha,n} ).
\end{equation}

\begin{lemma} \label{lem:variances_tilde_checked}
We have for all $\alpha \in (1/2, \infty)$
$$ \Var [S^{iso}( {\cal P}_{\alpha, n}) ] = \Var [\tilde{S}^{iso}_H (\tilde{\cal P}_{\alpha,n})]+ o(1),  \ \  \Var [S^{ext}( {\cal P}_{\alpha, n}) ] = \Var [\tilde{S}^{ext}_H (\tilde{\cal P}_{\alpha,n})] + o(1),
$$
as well as
$$
\mathbb{E} [S^{iso}( {\cal P}_{\alpha, n}) ] = \mathbb{E} [\tilde{S}^{iso}_H (\tilde{\cal P}_{\alpha,n})]+ o(1), \ \  \mathbb{E} [S^{ext}( {\cal P}_{\alpha, n}) ] = \mathbb{E} [\tilde{S}^{ext}_H (\tilde{\cal P}_{\alpha,n})]
+ o(1).
$$
\end{lemma}

\begin{proof} For brevity we write $S_n$ for $S^{iso}( {\cal P}_{\alpha, n})$ and $\tS_n$ for $\tilde{S}^{iso}_H (\tilde{\cal P}_{\alpha,n
})$.
We first assert that
\be \label{assert}
\P( S_n \neq \tS_n) = O(n^{-15}).
\ee
To see this, we condition on the event that $| \tilde{\cal P}_{\alpha,n} | \leq 2n$ (note that the complementary event
has probability which is generously bounded by  $O(n^{-16})$) and then use Boole's
inequality together with Lemma \ref{lemm1}.

Now write
$$
\Var S_n = \Var [  S_n {\bf 1}(S_n = \tS_n ) + S_n
{\bf 1}(S_n \neq \tS_n) ].
$$
By H\"older's inequality  and \eqref{assert}, we have  $$\Var [ S_n {\bf 1}(S_n \neq \tS_n)] \leq \mathbb{E} [ S_n^2 {\bf 1}(S_n \neq \tS_n)] $$
$$
\leq (\mathbb{E} [ |S_n|^3])^{2/3} \P (S_n \neq \tS_n)^{1/3} = O(n^{-3}).
$$

\vskip.5cm

Using the inequality $\Var[X + Y] \leq  \Var X + \Var Y + 2 \sqrt{ \Var X} \sqrt {\Var Y}$,
together with $\Var [S_n {\bf 1}(S_n = \tS_n)] = O(n^2)$, we see that $ \Var S_n =
\Var \tS_n + o(1),$ which proves the first assertion in Lemma \ref{lem:variances_tilde_checked}.
The remaining assertions are proved similarly.
\end{proof}

Define the Poisson point process $\tilde{\cal P}_{\alpha}$ on $\mathbb{R} \times \mathbb{R}^+$ with intensity measure $\mu_\alpha$ given by
\be \label{Pnavdef}
\mu_{\alpha}  (S) := \beta \int_S  e^{-\alpha y} dx dy,
\ee
where $S \subseteq \mathbb{R} \times \mathbb{R}^+$ is measurable. 
Recall from \eqref{defiD} that $D = (-I_n, I_n] \times [0, R].$ Put
\begin{equation} \label{eq:sum_score_R^2}
\tS^{iso}_H( \tilde{\cal P}_{\alpha} \cap D) := \sum_{p \in \tilde{\cal P}_{\alpha} \cap D([0,H]) } \txi^{iso}(p, \tilde{\cal P}_{\alpha} \cap D )
\end{equation}
and
\begin{equation} \label{eq:sum_score_R^2_extr}
\tS^{ext}_H ( \tilde{\cal P}_{\alpha} \cap D) :=  \sum_{p \in \tilde{\cal P}_{\alpha} \cap D([0,H])  } \txi^{ext}(p, \tilde{\cal P}_{\alpha} \cap D).
\end{equation}

The following lemma, together with Lemma \ref{lem:variances_tilde_checked},  shows that to prove Theorems \ref{prop:var_Shat} and \ref{extVar}, it is enough to establish expectation and variance asymptotics for $\tS^{iso}_H ( \tilde{\cal P}_{\alpha} \cap D )$ and
$\tS^{ext}_H( \tilde{\cal P}_{\alpha} \cap D)$.

\begin{lemma}\label{lem:exp_var_equiv}
We have  for all $\alpha \in (1/2, \infty)$
$$
\left| \mathbb{E} [ \tS^{iso}_H( \tilde{\cal P}_{\alpha,n} ) -  \tS^{iso}_H( \tilde{\cal P}_{\alpha} \cap D
)]  \right| = o(1), \ \  \left| \mathbb{E} [ \tS^{ext}_H( \tilde{\cal P}_{\alpha,n} ) -  \tS^{ext}_H( \tilde{\cal P}_{\alpha} \cap D )] \right| = o(1)
$$
and
$$
|\Var [ \tS^{iso}_H ( \tilde{\cal P}_{\alpha,n} )]  - \Var [\tS^{iso}_H ( \tilde{\cal P}_{\alpha} \cap D )] | = o(n), \ \ |\Var [\tS^{ext}_H( \tilde{\cal P}_{\alpha,n}  )]  - \Var [ \tS^{ext}_H( \tilde{\cal P}_{\alpha} \cap D )] | = o(n).
$$
\end{lemma}

We will  show that $\Var [ \tS^{iso}_H( \tilde{\cal P}_{\alpha} \cap D )]$ and
$ \Var [\tS^{ext}_H( \tilde{\cal P}_{\alpha} \cap D )]$  are both $\Omega (n)$. This, together with the following corollary of Lemmas \ref{lem:variances_tilde_checked} and \ref{lem:exp_var_equiv}, implies that the leading terms of $\Var [S^{iso}({\cal P}_{\alpha, n})]$ and $\Var [S^{ext}({\cal P}_{\alpha, n})]$ are given by $\Var [\tS^{iso}_H( \tilde{\cal P}_{\alpha} \cap  D )]$ and $\Var [\tS^{ext}_H( \tilde{\cal P}_{\alpha} \cap D)]$, respectively.

\begin{corollary}  \label{cor1}  We have for all $\alpha \in (1/2, \infty)$
\begin{equation} \label{eq:var_iso_equiv}
\left| \Var [S^{iso}({\cal P}_{\alpha, n})] -  \Var [\tS^{iso}_H( \tilde{\cal P}_{\alpha} \cap  D )]  \right| = o(n)
\end{equation}
and
\begin{equation} \label{eq:var_ext_equiv}
\left| \Var [ S^{ext}({\cal P}_{\alpha, n})] -  \Var [\tS^{ext}_H( \tilde{\cal P}_{\alpha} \cap D)] \right| = o(n).
\end{equation}
\end{corollary}

\vskip.3cm

\noindent{\em Proof of Lemma \ref{lem:exp_var_equiv}}. Let $\check{X}_n$ be  $\tS^{iso}_H ( \tilde{\cal P}_{\alpha,n} )$ and let $\widehat{X}_n$ be $\tS^{iso}_H( \tilde{\cal P}_{\alpha} \cap  D )$. We denote by $F_n$ the event that
$\tilde{\cal P}_{\alpha, n}  \not =  \tilde{\cal P}_{\alpha} \cap D$. By Lemma~\ref{lem:dist}, there is a coupling of the point processes $\tilde{\cal P}_{\alpha, n}$ and
$\tilde{\cal P}_{\alpha} \cap D$ such that $\mathbb{P} (  F_n ) = O(n^{-2 \alpha }) = o(n^{-1})$, since $\alpha \in (1/2, \infty)$.
We let $A_n:= \{ |\tilde{\cal P}_{\alpha, n}| > 2n \}$ and $B_n:=\{| \tilde{\cal P}_{\alpha} \cap D | > 2n \}$.
Then $\mathbb{P}(A_n \cup B_n) = o(n^{-1})$.
Setting  $\check{Y}_n:= \check{X}_n - \mathbb{E} \check{X}_n$ and $\widehat{Y}_n := \widehat{X}_n - \mathbb{E} \widehat{X}_n$ gives
\begin{equation*}
\begin{split}
\Var \check{X}_n &= \mathbb{E} [ \check{Y}_n^2 {\bf 1}( F_n^c )]
+ \mathbb{E} [ \check{Y}_n^2 {\bf 1}(F_n)] \\
& = \mathbb{E} [ \widehat{Y}_n^2 {\bf 1}(F_n^c)]
+  \mathbb{E}  [\check{Y}_n^2 {\bf 1}(F_n) ]  \\
&\leq \mathbb{E}  \widehat{Y}_n^2 + \mathbb{E} [ \check{Y}_n^2 {\bf 1}(F_n)].
\end{split}
\end{equation*}
Now,
\begin{equation*}
\begin{split}
\mathbb{E} [ \check{Y}_n^2 {\bf 1}(F_n)]  &\leq  \mathbb{E} [ |\tilde{\cal P}_{\alpha, n}|^2 {\bf 1} (F_n)] \\
& =  \mathbb{E} [ |\tilde{\cal P}_{\alpha, n}|^2 {\bf 1} (F_n){\bf 1} ( A_n^c) ]   +
\mathbb{E} [ |\tilde{\cal P}_{\alpha, n}|^2 {\bf 1} (F_n){\bf 1} (A_n)] \\
& \leq 4n^2 \P (F_n) + \mathbb{E} [ |\tilde{\cal P}_{\alpha, n}|^2 {\bf 1} (A_n)] \\
& = o(n) + \mathbb{E} [ |\tilde{\cal P}_{\alpha, n}|^2 {\bf 1} (A_n)] = o(n),
\end{split}
\end{equation*}
since $|\tilde{\cal P}_{\alpha, n} |$ is Poisson-distributed with parameter equal to $n$.
Thus $ \Var \check{X}_n \leq \Var \widehat{X}_n  + o(n).$
The bound remains valid if we interchange $\check{X}_n$ with $\widehat{X}_n$, $\tilde{\cal P}_{\alpha, n}$ with $\tilde{\cal P}_{\alpha} \cap D$, and $A_n$ with $B_n$.
We thus obtain $|\Var\check{X}_n - \Var \widehat{X}_n | = o(n).$  The proof of the bound for
$|\mathbb{E} [\check{X}_n -  \widehat{X}_n] |$ is identical, except that
second moments are replaced by first moments and this yields $| \mathbb{E}[\check{X}_n -  \widehat{X}_n] | = o(1)$.  This completes the proof of the estimates involving
$\tS^{iso}_H ( \tilde{\cal P}_{\alpha,n} )$ and $\tS^{iso}_H( \tilde{\cal P}_{\alpha} \cap D )$.  The proofs of the assertions involving $\tS^{ext}_H( \tilde{\cal P}_{\alpha,n} )$ and $\tS^{ext}( \tilde{\cal P}_{\alpha} \cap D )$ are  identical.  \qed

\vskip.3cm

\section{Preparing for the proof of Theorem ~\ref{prop:var_Shat}}  \label{Prep}

We  provide several  lemmas needed to estimate  $\Var[\tS^{iso}_H( \tilde{\cal P}_{\alpha} \cap D )]$.

\subsection{A covariance formula for  $\txi^{iso}$}

We establish a basic covariance formula needed for the calculation of $\Var[\tS^{iso}_H( \tilde{\cal P}_{\alpha} \cap D )]$.
If $\xi (p, (\tilde{\cal P}_{\alpha} \cap D) \cup \{p\})$ is a score function, we define
the \emph{covariance of $\xi$ at the points $p_1,p_2$} as
\begin{align}
c^\xi( p_1, p_2) &:= c^\xi( p_1, p_2;  \tilde{\cal P}_{\alpha} \cap D)  \nonumber \\
& :=  \mathbb{E} \left[ \xi( p_1, (\tilde{\cal P}_{\alpha} \cap D) \cup \{p_1,p_2\} ) \cdot \xi( p_2, (\tilde{\cal P}_{\alpha} \cap D) \cup \{p_1,p_2\} ) \right] \nonumber  \\
& \hspace{1.5cm} - \mathbb{E} \left[ \xi( p_1, (\tilde{\cal P}_{\alpha} \cap D)\cup \{p_1\})\right] ~\mathbb{E} \left[\xi(p_2, (\tilde{\cal P}_{\alpha} \cap D )\cup\{p_2\})\right]. \label{Cov}
\end{align}

We will give an expression for $c^{\tilde{\xi}^{iso}}$.
The number of points of $(\tilde{\cal P}_{\alpha} \cap D)$ inside
$B(p)$ is Poisson-distributed with parameter $\mu_{\alpha}  (B(p) )$, implying that
$ \mathbb{E} [ \txi^{iso}( p, (\tilde{\cal P}_{\alpha} \cap D) \cup \{p\} )] = \exp \left( - \mu_{\alpha}  (B(p) ) \right).$  Thus
$$
  \mathbb{E} [ \txi^{iso}( p_1, (\tilde{\cal P}_{\alpha} \cap D) \cup \{p_1\} )]   \mathbb{E} [ \txi^{iso}( p_2, (\tilde{\cal P}_{\alpha} \cap D) \cup \{p_2\})]  =   \exp \left( - \left( \mu_{\alpha}  (B(p_1)) + \mu_{\alpha}  (B(p_2))\right) \right).
$$
Put
\be \label{defSp}
 \Part{p_1p_2} := B( p_1) \cap B(p_2).
 \ee
If  $p_1 \not \in B (p_2)$ and $p_2 \not \in B( p_1)$, then we have
\begin{align*}
 & \mathbb{E} [ \txi^{iso} (p_1,(\tilde{\cal P}_{\alpha} \cap D)  \cup \{p_1,p_2 \} ) \txi^{iso}(p_2,  (\tilde{\cal P}_{\alpha} \cap D)  \cup \{p_1,p_2\} ) ]  \\
& =  \exp \left( -\mu_{\alpha}  (B(p_1) \cup B(p_2)) \right) \\
&=  \exp(- \left(\mu_{\alpha}  (B(p_1))
+ \mu_{\alpha}  (B(p_2))\right) +\mu_{\alpha}  ( \Part{p_1p_2})  )\\
&=  \mathbb{E} [ \txi^{iso}( p_1, (\tilde{\cal P}_{\alpha} \cap D) \cup \{p_1\} ) ]
 \mathbb{E} [ \txi^{iso}( p_2, (\tilde{\cal P}_{\alpha} \cap D)  \cup \{p_2\}) ]
 \cdot \exp \left(\mu_{\alpha}  ( \Part{p_1p_2}) \right).
\end{align*}

Therefore, given $p_1, p_2 \in D([0,H])$, we have the {\em following basic covariance formula}:
\begin{equation} \label{eq:cov_intersection}
\begin{split}
c^{\tilde{\xi}^{iso}} (p_1,p_2) =
\begin{cases}
\mathbb{E} [ \txi^{iso}( p_1, \tilde{\cal P}_{\alpha} \cap D )]  \mathbb{E} [ \txi^{iso}( p_2,  \tilde{\cal P}_{\alpha} \cap D )]
\left(  \exp \left(\mu_{\alpha}  ( \Part{p_1p_2}) \right)  -1 \right)  & p_1 \not \in
B(p_2 ) \\
-\mathbb{E} [ \txi^{iso}( p_1,  (\tilde{\cal P}_{\alpha} \cap D) \cup \{p_1\} )] \mathbb{E} [ \txi^{iso}( p_2,  (\tilde{\cal P}_{\alpha} \cap D) \cup \{p_2\} )]  & p_1 \in B(p_2)
\end{cases}.
\end{split}
\end{equation}

Consider the second case in ~\eqref{eq:cov_intersection}, where the covariance is negative.
By Lemma~\ref{lem:relAngle_generic}(ii),  given  points $p_1$ and $p_2$ with
$y(p_1),y(p_2) \in [0, H]$, we have
$$
p_1 \in B (p_2)\ \mbox{if and only if} \ |x(p_2 )  -x(p_1)|_\Phi < (1+\lambda_n (p_1,p_2)) e^{ \frac{1}{2} (y(p_1) + y(p_2))},
$$
where  $\lambda_n (p_1,p_2)=o(1)$ uniformly for all $p_1,p_2 \in \R{[0,H]}$.
Setting
\be \label{Yonetwo}
Y_{12}:= (1+\lambda_n (p_1,p_2)) e^{ \frac{1}{2} (y(p_1) + y(p_2))}
\ee
we may re-state the above as
\begin{equation}\label{eq:condition_inc}
p_1 \in B (p_2)\ \mbox{if and only if} \ |x(p_2 )  -x(p_1)|_\Phi < Y_{12}.
\end{equation}
Before focussing on the first case in  ~\eqref{eq:cov_intersection} we need some geometric preliminaries.

\subsection{The geometry of balls with  height coordinate at most $H$}

Our aim now is to estimate $\mu_{\alpha}  ( \Part{p_1p_2})$.
The set inclusion  $B^{-}(p) \subseteq B(p) \subseteq B^{+}(p)$ at \eqref{eq:inclusion_mapped} implies
$$
\mu_{\alpha}  ( B^{-}(p_1)  \cap B^{-}(p_2) )
\leq \mu_{\alpha}  ( \Part{p_1p_2}) \leq
\mu_{\alpha}  ( B^{+}(p_1)  \cap B^{+}(p_2) ).
 $$

Given  $l \in [0, R]$, $\eps > 0$ and $p \in \R{[0,H)} $ we set
\be
B_l^-(p):= B_l^-(p, \eps):= \{ (x,y) \ : \ y \in [0, l), \ |x-x(p)|_{\Phi}  <  \OneLow \cdot e^{\frac12 (y+ y(p))} \}
\ee
\be \label{eq:ball_plus}
B_l^+ (p) := B_l^+(p, \eps):= \{ (x,y) \ : \ y \in [0, l), \  |x-x(p)|_{\Phi} < \OneUp \cdot e^{\frac12 (y+ y(p))}\}
\ee
and
\be\label{eq:ball_upper}
 Z_l(p) :=  \{ (x,y)\in \Rcal \ : \ y\geq l \}.
\ee
We continue to assume that $p_1$ and $p_2$ belong to $D([0,H])$.
We assume without loss of generality that $x (p_1) <_\Phi x (p_2)$ and $y(p_1) \in  (y(p_2),H]$.
Henceforth, for $C:=C(\eps)$ as in Lemma~\ref{lem:relAngle_generic}, we put
\be \label{defh}
h:= h(p_1) := R - y(p_1) - C.
\ee
Notice that $B^{-}(p_1) = B_h^{-}(p_1)$ and $B_h^{-}(p_2) \subseteq B^{-}(p_2)$. See Figure~\ref{fig:ball}.
\begin{figure}[h!]
\begin{center}
\scalebox{0.7}{\input{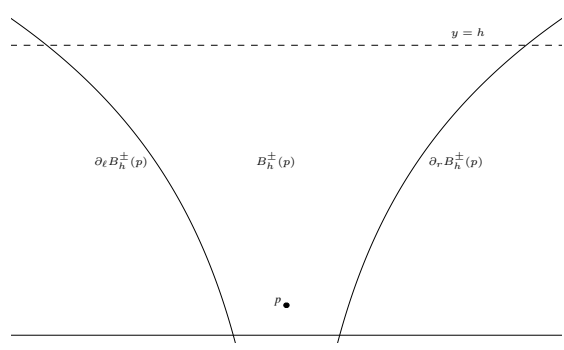}}
\caption{The ball $B_h^{\pm}(p) $ around a point $p$.  \label{fig:ball}}
\end{center}
\end{figure}

Furthermore, the definitions of $\BallUp$ and $\BallDown$ and the assumption $y(p_1 ) > y(p_2)$
imply
 $$
 B_h^{-} (p_1)\cap B_h^{-}(p_2)  \subseteq B^{-}(p_1)  \cap B^{-}(p_2)
 $$
 and
$$
B^{+}(p_1)  \cap B^{+}(p_2) \subseteq  \left(B_h^{+} (p_1)\cap B_h^{+}(p_2)  \right) \cup Z_h (p_1).
$$
\noindent
These inclusions yield
\begin{equation} \label{eq:inter_upper-bound}
  \mu_{\alpha}  ( \Part{p_1p_2}) \leq
\mu_{\alpha}  ( B_h^+ (p_1) \cap B_h^+ (p_2)) + \mu_{\alpha}  (Z_h (p_1))
\end{equation}
and
\begin{equation} \label{eq:inter_lower-bound}
\mu_{\alpha}  (B_h^{-} (p_1) \cap B_h^- (p_2)) \leq   \mu_{\alpha}  (B^{-}(p_1)  \cap B^{-}(p_2)) \leq  \mu_{\alpha}  ( \Part{p_1p_2}),
\end{equation}
whence
\be \label{Spm}
\mu_{\alpha}  (B_h^{-} (p_1) \cap B_h^- (p_2)) \leq   \mu_{\alpha}  ( \Part{p_1p_2}) \leq
\mu_{\alpha}  ( B_h^+ (p_1) \cap B_h^+ (p_2)) + \mu_{\alpha}  (Z_h (p_1)).
\ee

First, we notice that the definitions of $h, \beta,$ and $I_n$ give
\begin{equation} \label{eq:meas_Z}
\begin{split}
\mu_{\alpha}  (Z_h (p_1)) &= \beta \cdot 2I_n
\int_{R-y(p_1)-C}^R e^{-\alpha y} dy \\
&= 2\nu   e^{ \frac{R}{2}} \left( e^{-\alpha (R - y(p_1) -C)} - e^{-\alpha R}\right)
\\
&= 2\nu  e^{( \frac{1} {2}  - \alpha )R} \left(e^{\alpha y(p_1) + \alpha C} -1
\right) \\
&=\Theta (1) \cdot e^{( \frac{1} {2}  - \alpha )R + \alpha y(p_1)}.
\end{split}
\end{equation}

Denote by $B_h^{\pm} (p)$ either of the balls $B_h^+(p)$ or $B_h^- (p)$ and denote by $\OnePM$ either $\OneUp$ or
$\OneLow$, depending on which of the two cases we are considering.
The following lemma  characterises when two balls are disjoint.

\begin{lemma} \label{4.1}  Fix $\eps \in (0,1)$ and assume $x(p_1) <_\Phi x(p_2)$. With $h$ at \eqref{defh} we
have $B_h^{\pm} (p_1) \cap\Ball{p_2}{h}=\emptyset $ if and only if
 $$ |x(p_2) - x(p_1)|_{\Phi}  >  \OnePM\cdot e^{ \frac{h} {2} }  ( e^{ \frac{y(p_1)} {2}}  + e^{ \frac{y(p_2)} {2} } ). $$
\end{lemma}

\begin{proof}
By the definition of $B_h^{\pm} (p_1)$, the right-most point of $B_h^{\pm} (p_1)$, denoted by $p' := (x(p'), y(p'))$ satisfies
$x(p') =_\Phi x(p_1) + \OnePM \cdot e^{\frac12 (y(p_1) + h)}$.
Similarly, the left-most point of $\Ball{p_2}{h}$ (denoted by $p''$)
satisfies $x(p'') =_\Phi x(p_2) -\OnePM \cdot  e^{\frac12 (y(p_2) + h)}$.
Note that $y(p') = y(p'')$.
Then $x(p') = x(p'')$ if and only if $x(p_2) - x(p_1) =_\Phi \OnePM \cdot  e^{h/2} ( e^{y(p_1)/2} + e^{y(p_2)/2})$.
If $|x(p_2) - x(p_1)|_\Phi > \OnePM \cdot  e^{h/2} ( e^{y(p_1)/2} + e^{y(p_2)/2})$, then $B_h^{\pm}(p_1) \cap B_h^{\pm} (p_2)=\emptyset $.
Likewise,  if $|x(p_2) - x(p_1)|_\Phi \leq \OnePM \cdot  e^{h/2} ( e^{y(p_1)/2} + e^{y(p_2)/2})$, then $B_h^{\pm}(p_1) \cap B_h^{\pm} (p_2)\not = \emptyset$.
\end{proof}


We still assume $x(p_1) <_\Phi x(p_2)$ and we set $t :=_\Phi x (p_2) - x(p_1)$. With $h$ at \eqref{defh} we consider in the remainder of this sub-section the case $0 <_\Phi t <_\Phi  \OnePM \cdot e^{h/2}( e^{y(p_1)/2} + e^{y(p_2)/2})$.
For $t$ in this domain, Lemma \ref{4.1} implies that $\Ball{p_1}{h}\cap \Ball{p_2}{h} \not =\emptyset$.  Given $p \in \R{(0,H)}$, denote the left and right boundaries of $\Ball{p}{h}$ by
$$
\partial_{\ell}  \Ball{p}{h} := \{ p' \in \R{(0,h)}  : \ x (p') =_\Phi x (p) - \OnePM \cdot e^{\frac{1}{2}(y(p) + y(p'))}\},
$$
and
$$
\partial_{r} \Ball{p}{h}:=  \{  p' \in \R{(0,h)} : \
x(p') =_\Phi x(p) + \OnePM \cdot e^{\frac{1}{2}(y(p) + y(p'))}\};
$$
cf. Figure~\ref{fig:ball}.   The first part of the next lemma shows that
$\partial_{r} \Ball{p_1}{h}$ and $\partial_{\ell} \Ball{p_2}{h}$ intersect whenever the $x$-coordinates of $p_1$ and $p_2$ are
far enough apart with respect to the exponentiated height coordinates.

\begin{lemma} \label{geometrylemma}
(i) If $\partial_{r} \Ball{p_1}{h} \cap \partial_{\ell} \Ball{p_2}{h} \neq \emptyset$ then
\be \label{notempty}
x(p_2) - x(p_1) >_\Phi  \OnePM \cdot \left( e^{  \frac{y(p_1)} {2}}  + e^{ \frac{y(p_2)} {2}}  \right).
\ee
\noindent (ii) $\partial_{\ell} \Ball{p_1}{h} \cap \partial_{\ell} \Ball{p_2}{h} = \emptyset$.
\vskip.1cm
\noindent (iii) If $p_{12}' \in \partial_{r} \Ball{p_1}{h} \cap \partial_{r} \Ball{p_2}{h} \neq \emptyset$ then
\begin{equation} \label{eq:intersection_point_upper}
e^{  \frac{ y(p_{12}') } {2}}  = \frac{1}{\OnePM} \cdot \frac{|x(p_2) - x(p_1)|_\Phi}{e^{y(p_1)/2} -e^{y(p_2)/2}}.
\end{equation}
\end{lemma}


\noindent{\em Proof.} (i) If $p_{12} := \partial_{r} \Ball{p_1}{h} \cap \partial_{\ell} \Ball{p_2}{h} \neq \emptyset$ (cf. Figure~\ref{fig:Balls-intersection}), then
$y(p_{12})$ satisfies the following equations:
\begin{eqnarray*}
x(p_{12}) - x(p_1) &=_\Phi &\OnePM \cdot e^{ \frac12 (y(p_1) + y(p_{12}) )}   \\
\text{and} & & \\
x(p_2) - x(p_{12})  &=_\Phi&\OnePM \cdot  e^{ \frac12 (y(p_2) + y(p_{12}))}.
\end{eqnarray*}
Therefore
\begin{equation} \label{eq:intersection_point_low}
e^{ \frac{y(p_{12} ) } {2} } = \frac{1}{\OnePM} \cdot \frac
{|x (p_2) - x(p_1)|_\Phi}{ e^{y(p_1)/2} + e^{y(p_2)/2}}.
\end{equation}
Hence, $p_{12}$ exists provided that $y(p_{12}) >0$, which implies \eqref{notempty}, as desired.

\vskip.3cm

\noindent (ii) On the contrary,
assume that there exists $p \in \partial_{\ell} \Ball{p_1}{h} \cap \partial_{\ell} \Ball{p_2}{h}$. Using the definition
of the left boundary we have
$$ x(p_1) -\OnePM \cdot e^{\frac{1}{2}(y(p_1) + y(p))}=_\Phi x(p_2) -\OnePM \cdot e^{\frac{1}{2}(y(p_2) + y(p))}.$$
We deduce that
$$ e^{y(p)/2} = \frac{1}{\OnePM} \cdot \frac{|x(p_2) - x(p_1)|_\Phi}{e^{y(p_2)/2} -e^{y(p_1)/2}} < 0, $$
since $y(p_2) < y(p_1)$, which is impossible. Thus, such $p$ cannot exist and $\partial_{\ell} \Ball{p_1}{h} \cap \partial_{\ell}
\Ball{p_2}{h} = \emptyset$, proving (ii) as desired.

\vskip.3cm

\noindent (iii) Assume   $p_{12}' = \partial_{r} \Ball{p_1}{h} \cap \partial_{r} \Ball{p_2}{h} \neq \emptyset$  (cf. Figure~\ref{fig:Balls-intersection}). Then $p_{12}'$  satisfies
$$ x(p_1) + \OnePM \cdot e^{\frac{1}{2}(y(p_1) + y(p_{12}' ))}=_\Phi x(p_2) +\OnePM \cdot
 e^{\frac{1}{2}(y(p_2) + y(p_{12}' ))},$$
which yields \eqref{eq:intersection_point_upper} and completes the proof of Lemma \ref{geometrylemma}.  \qed

\vskip.3cm

Note that \eqref{eq:intersection_point_upper} and \eqref{eq:intersection_point_low} imply that $y(p_{12}) < y(p_{12}')$. For convenience, we will set
$$
y(p_{12}):= y_L  \ \ \ \text{and} \ \ \ y(p_{12}'):= y_U.
$$

\begin{figure}
\begin{center}
\includegraphics[scale=0.35]{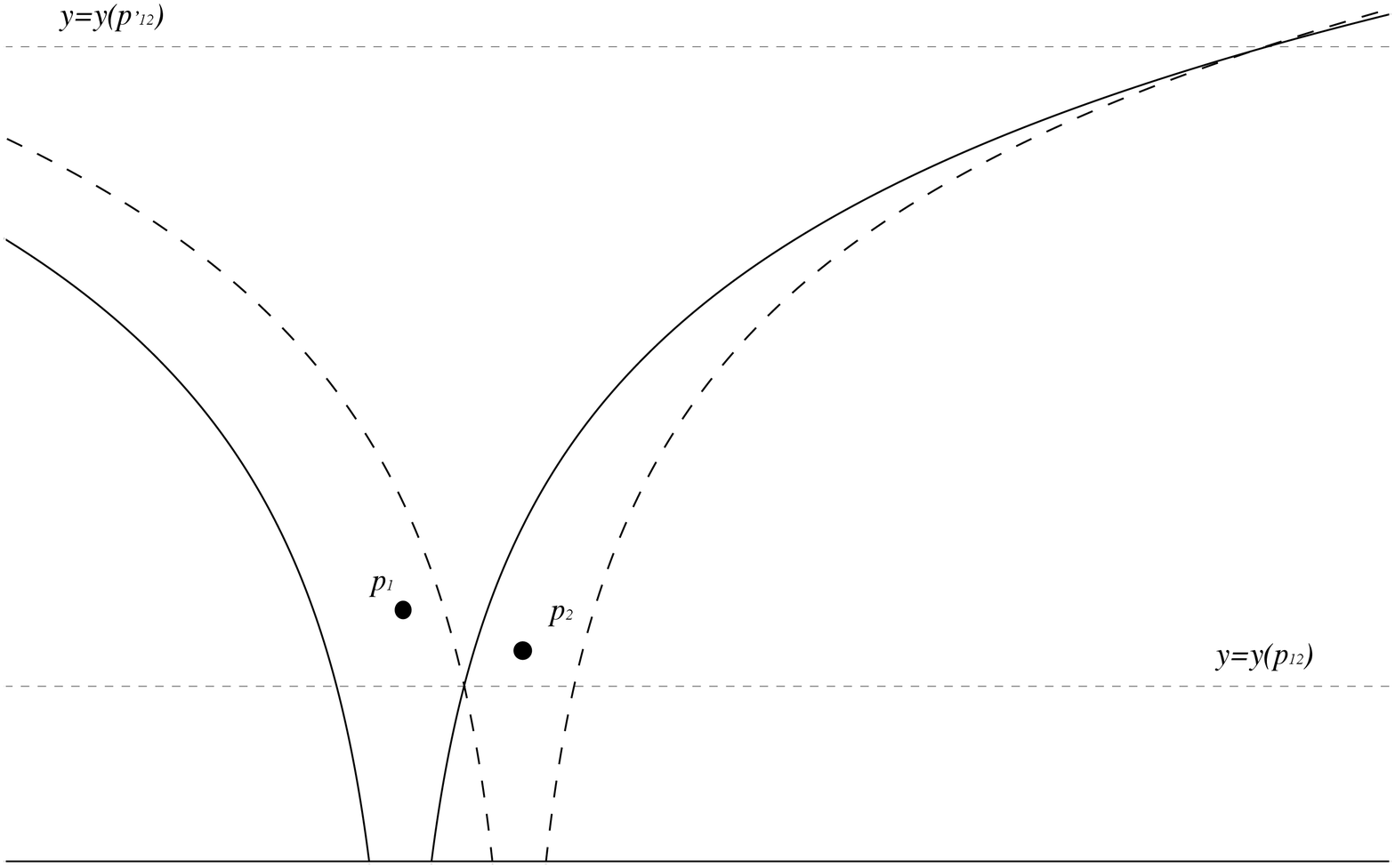}
\caption{The intersection of two balls  \label{fig:Balls-intersection}}
\end{center}
\end{figure}

Consider now the union of the two balls $\Ball{p_1}{h}\cup \Ball{p_2}{h}$.
For any $x >_\Phi x(p_{12}')$ let $p \in \partial_{r} \Ball{p_1}{h}$ and $p' \in \partial_{r} \Ball{p_2}{h}$ be such that
$x(p) = x(p') = x$. Then $y(p') >  y(p)$.
Now for $x<_\Phi x(p_1) - \OnePM \cdot e^{y(p_1)/2}$, consider two points $p \in \partial_{\ell} \Ball{p_1}{h}$ and $p' \in \partial_{\ell} \Ball{p_2}{h}$ such that $x(p) = x(p') = x$. Since $\partial_{\ell} \Ball{p_1}{h} \cap \partial_{\ell} \Ball{p_2}{h} = \emptyset$, it follows that
$y(p) < y(p')$. In other words, the curves $\partial_{\ell} \Ball{p_1}{h}$ and $\partial_{\ell} \Ball{p_2}{h}$ do not intersect and
$\partial_{\ell} \Ball{p_2}{h}$ stays ``above'' $\partial_{\ell} \Ball{p_1}{h}$.

\subsection{ The $\mu_{\alpha} $-content of $\Ball{p_1}{h} \cap \Ball{p_2}{h}$}

We now focus on $\mu_{\alpha}  ( B_h^+ (p_1) \cap B_h^+ (p_2))$ and $\mu_{\alpha}  ( B_h^- (p_1) \cap
B_h^- (p_2))$.  The calculations of the $\mu_{\alpha} $-measure of these two intersections are similar, as the considered sets differ only by
constant factors $\OneUp$ and $\OneLow$.  We provide a generic calculation covering both cases.
The inequality  ~\eqref{Spm} shows
that the $\mu_{\alpha}$-content of
$$
\InterSec{p_1p_2}:= \Ball{p_1}{h} \cap \Ball{p_2}{h}
$$
controls the growth of $\Cov(p_1,p_2)$. The following lemma gives quantitative bounds on $\mu_{\alpha} (\InterSec{p_1p_2})$. We will use the first part of the lemma to lower bound  $\Var [\tS^{iso}_H( \tilde{\cal P}_{\alpha} \cap D )]$. It turns out that this gives the main contribution to the variance bound of Theorem~\ref{prop:var_Shat}. We will give a matching upper bound on the variance through  the
Poincar\'e inequality. The second part of the lemma gives an upper bound on the intensity measure of $\InterSec{p_1p_2}$, which will be used in the proof of the central limit theorem for  $\tS^{iso}_H( \tilde{\cal P}_{\alpha} \cap D ), \alpha \in (1, \infty)$.  Recall from \eqref{defh} that
we have set $h:= R - y(p_1) - C.$ Also, recall that we set $\gamma := 4 \beta/(2 \alpha - 1)$.

\begin{lemma} \label{lem:mu_intersection}  Let $p_i := (x(p_i), y(p_i))$, $i = 1,2$ as above.
 For $i=1,2$  we put $\Y_i := e^{y(p_i)/2}$, we set $\eta (Y_2):= \gamma \Y_2 e^{(1/2 - \alpha)h}$, and we suppose that  $t := x(p_2) - x(p_1) >_\Phi 0$.
\begin{enumerate}
\item[ (i)]
If  $ \max \{ \OnePM \cdot \left( \Y_1 + \Y_2 \right), Y_{12} \}  < t \leq \OnePM \cdot e^{ \frac{h} {2} }(\Y_1 - \Y_2 )$,
then
\be \label{caseii}
 \mu_{\alpha}  ( \InterSec{p_1p_2})  = \kappa t^{1-2\alpha} - \OnePM \eta (Y_2),
  \ee
where
\begin{equation} \label{kap}
\kappa := \OnePM \cdot \frac{\gamma}{4\alpha} \left(  \left( \Y_1 + \Y_2 \right)^{2\alpha}
- \left( \Y_1 - \Y_2\right)^{2\alpha} \right). \end{equation}
\item[ (ii)]
If $e^{R/4} (Y_1 + Y_2) <_\Phi t <_\Phi I_n$, then
\begin{equation} \label{eq:far_apart}
\begin{split}
\mu_{\alpha}  (\InterSec{p_1p_2} )
&= O(1) \cdot \left(t^{1-2\alpha} \cdot (Y_1 + Y_2)^{2\alpha} + n^{1-2\alpha}\cdot  Y_1^{2\alpha}\right).
\end{split}
\end{equation}
\end{enumerate}
\end{lemma}

\begin{proof}  Part (i). We express $\InterSec{p_1p_2}$ as the disjoint union of  the sets $\R{(y_U,R)} \cap \InterSec{p_1p_2}$
and $\R{[y_L,y_U]} \cap \InterSec{p_1p_2}$.
The above analysis implies that $\R{(y_U,R)} \cap \InterSec{p_1p_2} = \R{(y_U,R]} \cap \Ball{p_2}{h}$.
Let us consider the region $\R{[y_L,y_U]}$.
Let $y \in [y_L, y_U]$. Then any point $p$ with $y(p) =y$  belongs to $\InterSec{p_1p_2}$ if and only if
$ x(p_2) - \OnePM \cdot e^{\frac12 (y(p_2) + y)} \leq_\Phi x(p) \leq_\Phi x(p_1) + \OnePM \cdot e^{\frac12 (y(p_1) + y)}$.
Note that $\R{(0,y_{L})} \cap \InterSec{p_1p_2} = \emptyset$.

Recalling $\gamma = 4 \beta/(2 \alpha - 1)$, these observations imply
\begin{equation} \label{eq:mu_intersection_U}
\begin{split}
& \mu_{\alpha} (\R{(y_U,R)} \cap \InterSec{p_1p_2} ) \\
& = \OnePM \cdot 2\beta e^{ \frac{y(p_2)} {2}}
\int_{y_U}^{h} e^{( \frac{1}{2}  -\alpha) y} dy \\
&=\OnePM \cdot \gamma \cdot   \left( e^{ \frac{ y(p_2)} {2} +( \frac{1} {2}  - \alpha) y_U}
-e^{ \frac{y(p_2)} {2} }e^{(  \frac{1} {2}  - \alpha)h} \right)  \\
&= \OnePM \cdot \gamma e^{ \frac{y(p_2)} {2} +( \frac{1} {2}  - \alpha) y_U}  -\OnePM \cdot\gamma e^{ \frac{y(p_2)} {2} }e^{( \frac{1} {2} - \alpha)h}.
\end{split}
\end{equation}
We also have
\begin{equation} \label{eq:mu_intersection_UL}
\begin{split}
\lefteqn{\mu_{\alpha}  ( \R{[y_L,y_U]} \cap \InterSec{p_1p_2} ) } \\
& =\beta \int_{y_L}^{y_U} \left( (x(p_1) - x(p_2)) + \OnePM \cdot e^{ \frac{y} {2}}  \left( e^{ \frac{y(p_1)} {2}}  + e^{ \frac{y(p_2)} {2}}  \right) \right) e^{-\alpha y} dy \\
&=\frac{\beta}{\alpha}(x(p_1) - x(p_2)) \left( e^{-\alpha y_L} - e^{-\alpha y_U} \right)
\\
&\hspace{0.5cm} +\OnePM \cdot \frac{\gamma}{2} \left( e^{  \frac{y(p_1)} {2}}  + e^{ \frac{ y(p_2)} {2}}  \right)  \left( e^{( \frac{1} {2}  - \alpha) y_L} - e^{( \frac{1} {2}  - \alpha) y_U} \right).
\end{split}
\end{equation}
Hence, by \eqref{eq:mu_intersection_U}  and \eqref{eq:mu_intersection_UL} we have
\begin{equation} \label{eq:mu_intersection}
\begin{split}
\mu_{\alpha}  ( \InterSec{p_1p_2}) & = \mu_{\alpha}  (\R{(y_U,R]} \cap \InterSec{p_1p_2} )  +
\mu_{\alpha}  ( \R{[y_L,y_U]} \cap \InterSec{p_1p_2} ) \\
& = \frac{\beta}{\alpha}(x(p_1) - x(p_2)) \left( e^{-\alpha y_L} - e^{-\alpha y_U} \right)   \\
&\hspace{1cm} + \OnePM \cdot \frac{ \gamma} {2} e^{ \frac{y(p_1)} {2}}  \left( e^{( \frac{1} {2} - \alpha) y_L} - e^{(\frac{1} {2} - \alpha) y_U} \right)\\
&\hspace{1cm} + \OnePM \cdot \frac{ \gamma}{2} e^{ \frac{y(p_2)} {2}}   \left( e^{( \frac{1} {2} - \alpha) y_L} + e^{(\frac{1} {2} - \alpha) y_U} \right) - \OnePM \cdot \gamma e^{ \frac{y(p_2)} {2}} e^{(\frac{1} {2} - \alpha)h}.
\end{split}
\end{equation}
By \eqref{eq:intersection_point_upper} and \eqref{eq:intersection_point_low} we have
\begin{equation}   \label{eq:powers_expressed}
e^{  \frac{y_L} {2}} = \frac{1}{\OnePM} \cdot\frac{t}{\Y_1 + \Y_2} \ \ \ \  \ \text{and} \ \ \ \ \
e^{ \frac{y_U} {2}} = \frac{1}{\OnePM} \cdot\frac{t}{\Y_1 - \Y_2}.
\end{equation}
Therefore,
\begin{equation}  \label{eq:powers_alpha}
e^{-\alpha y_L} =  ( \OnePM)^{2\alpha}    \cdot \left( \frac{t}{\Y_1 + \Y_2} \right)^{-2\alpha} \ \text{and} \ e^{-\alpha y_U} =
( \OnePM)^{2\alpha} \cdot
\left( \frac{t}{\Y_1 - \Y_2} \right)^{-2\alpha}.
\end{equation}

Combining \eqref{eq:powers_expressed} and \eqref{eq:powers_alpha} yields
\begin{equation}  \label{eq:powers_total}
e^{(\frac{1} {2} -\alpha) y_L} = ( \OnePM)^{2\alpha-1 } \cdot
 \left( \frac{t}{\Y_1 + \Y_2} \right)^{1-2\alpha} \ \text{and} \ e^{( \frac{1} {2} -\alpha) y_U} = (\OnePM)^{2\alpha -1} \cdot \left(
\frac{t}{\Y_1 - \Y_2}\right)^{1-2\alpha}.
\end{equation}
Substituting  ~\eqref{eq:powers_alpha} and ~\eqref{eq:powers_total}
into ~\eqref{eq:mu_intersection} we have
\begin{equation} \label{eq:with_powers_in}
\begin{split}
\mu_{\alpha}  ( \InterSec{p_1p_2} ) &  = - ( \OnePM )^{2\alpha} \cdot \frac{\beta}{\alpha} t \left(   \left( \frac{t}{\Y_1 + \Y_2} \right)^{-2\alpha}
 - \left( \frac{t}{\Y_1 - \Y_2}\right)^{-2\alpha}\right) \\
& \ \ \ \ \ \ \ + ( \OnePM )^{2\alpha} \cdot \frac{ \gamma}{ 2 } \left(  \Y_1 \left( \left( \frac{t}{\Y_1 + \Y_2} \right)^{1-2\alpha}
-\left( \frac{t}{\Y_1 - \Y_2} \right)^{1-2\alpha} \right) \right. \\
& \hspace{1cm} \left.+ \Y_2 \left( \left( \frac{t}{\Y_1 + \Y_2} \right)^{1-2\alpha}
+ \left( \frac{t}{\Y_1 - \Y_2}\right)^{1-2\alpha} \right) \right) -\OnePM \cdot \gamma \Y_2 e^{(  \frac{1}{2} - \alpha)h} \\
&=  -( \OnePM )^{2\alpha} \cdot\frac{\beta}{\alpha} t^{1- 2\alpha} \left(   (\Y_1 + \Y_2)^{2\alpha}  - (\Y_1 - \Y_2)^{2\alpha} \right) \\
&   \ \ \ \ \ \ \  +  ( \OnePM )^{2\alpha}  \cdot  \frac{\gamma }{2 } t^{1-2\alpha}
\left[  \Y_1 \left(  (\Y_1 + \Y_2)^{2\alpha - 1}
- ( \Y_1 - \Y_2)^{2\alpha - 1} \right) \right. \\
& \hspace{1cm} \left. + \Y_2 \left( (\Y_1 + \Y_2)^{2\alpha - 1}
+ (\Y_1 - \Y_2)^{2\alpha - 1} \right) \right]  - \OnePM \cdot \gamma \Y_2 e^{( \frac{1}{2} - \alpha)h} .
\end{split}
\end{equation}
Notice that
\begin{equation*}
\begin{split}
 & \Y_1 \left( ( \Y_1 + \Y_2)^{2\alpha - 1} - ( \Y_1 - \Y_2)^{2\alpha - 1} \right) \\
& \hspace{1cm} + \Y_2 \left( ( \Y_1 + \Y_2)^{2\alpha - 1} - ( \Y_1 - \Y_2)^{2\alpha - 1} \right) \\
&=(\Y_1 + \Y_2) ( \Y_1 + \Y_2)^{2\alpha- 1}  -
(\Y_1 - \Y_2) ( \Y_1 - \Y_2)^{2\alpha - 1}  \\
&=   ( \Y_1 + \Y_2)^{2\alpha}  - (\Y_1 - \Y_2)^{2\alpha}.
\end{split}
\end{equation*}
Substituting this into \eqref{eq:with_powers_in} yields
\begin{equation*}
\begin{split}
\mu_{\alpha}  ( \InterSec{p_1p_2})  & = -( \OnePM )^{2\alpha}\cdot \frac{\beta}{\alpha} t^{1- 2\alpha}  \left( ( \Y_1 + \Y_2)^{2\alpha}  - (\Y_1 - \Y_2)^{2\alpha} \right)  \\
& + ( \OnePM )^{2\alpha} \cdot \frac{ \gamma}{2} t^{1-2\alpha}
\left( ( \Y_1 + \Y_2)^{2\alpha}  - (\Y_1 - \Y_2)^{2\alpha} \right) - \OnePM \cdot \gamma \Y_2 e^{(\frac{1}{2}  - \alpha)h} \\
&= ( \OnePM)^{2\alpha} \cdot \frac{ \gamma } {4 \alpha}  t^{1-2\alpha} \left( (\Y_1 + \Y_2)^{2\alpha}
- ( \Y_1 - \Y_2)^{2\alpha} \right) - \OnePM \cdot \gamma \Y_2 e^{(\frac{1}{2}  - \alpha)h}.
\end{split}
\end{equation*}
Hence, the proof of part (i) is complete.

\vskip.3cm

Part (ii). We will consider three different subsets of the interval $(e^{R/4} \cdot \left( \Y_1 + \Y_2 \right), I_n)$. For the case where
$ e^{R/4} \cdot \left( \Y_1 + \Y_2 \right)  <_\Phi t \leq_\Phi \OnePM \cdot e^{h/2}(\Y_1 - \Y_2 )$ we will use part (i). (Note that
$\max\{\OnePM \cdot \left( \Y_1 + \Y_2 \right), Y_{12} \} <  e^{R/2} \cdot \left( \Y_1 + \Y_2 \right) $, since
$Y_{12} < 2 Y_1 Y_2 < 2 Y_1^2 \leq 2 R^4$.)
 Indeed, the expression for $\mu_{\alpha}  ( \InterSec{p_1p_2})$
 immediately implies
that for any such $t$ we have
$$\mu_{\alpha}  ( \InterSec{p_1p_2})  =O(1) \cdot t^{1-2\alpha} (Y_1 + Y_2)^{2\alpha}.$$

Now, assume that $\OnePM \cdot e^{h/2} (Y_1 - Y_2) <_\Phi t \leq_\Phi \OnePM \cdot
e^{h/2} (Y_1 + Y_2)$.
In this case, we have $y_U \in (h, R]$. Thus, any point $p$ with $y(p) =y$ and $y \in [y_L, h]$ belongs to $\InterSec{p_1p_2}$ if and only if
$ x(p_2) - \OnePM \cdot e^{\frac12 (y(p_2) + y)} \leq_\Phi x(p) \leq_\Phi x(p_1) + \OnePM \cdot e^{\frac12 (y(p_1) + y)}$.
Hence, we will use a modified version of~\eqref{eq:mu_intersection_UL}:
\begin{align} \label{eq:mu_intersection_Case_3}
\mu_{\alpha}  (\InterSec{p_1p_2}) &  = \mu_{\alpha}  ( \R{[y_L,R]}  \cap \InterSec{p_1p_2} )  \nonumber \\
& = \beta \int_{y_L}^{h} \left( (x(p_1) - x(p_2)) + \OnePM \cdot
e^{  \frac{y} {2}  } \left( e^{  \frac{y(p_1)} {2}} + e^{ \frac{ y(p_2)} {2}}  \right) \right) e^{-\alpha y} dy \nonumber \\
& = \frac{\beta}{\alpha}(x(p_1) - x(p_2)) \left( e^{-\alpha y_L} - e^{-\alpha h} \right)
\\
&\hspace{0.5cm} + \OnePM \cdot \frac{2\beta}{2\alpha -1}\left(  e^{  \frac{y(p_1)} {2}} + e^{ \frac{ y(p_2)} {2}}  \right)  \left( e^{( \frac{1} {2}  - \alpha) y_L} - e^{( \frac{1} {2}  - \alpha) h} \right).  \nonumber
\end{align}
Using~\eqref{eq:powers_alpha} and~\eqref{eq:powers_total}, the above becomes:
\begin{align*}
\mu_{\alpha}  ( \InterSec{p_1p_2} )
& = -\frac{\beta}{\alpha}t \left( ( \OnePM)^{2\alpha} \cdot \left(\frac{t}{Y_1 + Y_2} \right)^{-2\alpha} - e^{-\alpha h} \right) \\
& \hspace{0.5cm} +\OnePM \cdot \frac{\gamma} {2} \left( Y_1 + Y_2 \right)
\left( (\OnePM)^{2\alpha -1} \cdot \left(\frac{t}{Y_1 + Y_2} \right)^{1-2\alpha}- e^{(1/2 - \alpha) h} \right) \\
& =  (\OnePM)^{2\alpha} \cdot t^{1-2\alpha} (Y_1 + Y_2)^{2\alpha}\beta \left(-\frac{1}{\alpha} + \frac{2}{2\alpha -1} \right)
+\frac{\beta}{\alpha}t e^{-\alpha h}\\
&  \hspace{0.5cm} - \OnePM \cdot \frac{\gamma }{2 }\left( Y_1 + Y_2 \right)  e^{( \frac{1} {2} -\alpha ) h}.
\end{align*}
(Note that when $t= (Y_1 + Y_2) e^{h/2}$, the above expression is equal to 0.)
Now, since $-\frac{1}{\alpha} + \frac{2}{2\alpha -1}=\frac{1}{\alpha (2\alpha -1)}>0$
we obtain
$$
\mu_{\alpha}  ( \InterSec{p_1p_2} )
=O(1) \cdot \left( 
t^{1-2\alpha} (Y_1 + Y_2)^{2\alpha} 
+
t e^{-\alpha h} \right). 
$$
Recalling that $h= R-y(p_1)-C$ we deduce that $e^{-\alpha h} =O(1) \cdot e^{-\alpha R} Y_1^{2\alpha}= O(1) \cdot n^{-2\alpha} \cdot Y_1^{2\alpha}$. So
$$
t e^{-\alpha h} = O(1) \cdot t \cdot n^{-2\alpha} \cdot Y_1^{2\alpha} \stackrel{t<I_n = O(n)}{=} O(1) \cdot
t^{1-2\alpha} (Y_1 + Y_2)^{2\alpha},
$$
which yields \eqref{eq:far_apart} when $t$ satisfies $\OnePM \cdot e^{h/2} (Y_1 - Y_2) <_\Phi t \leq_\Phi \OnePM \cdot
e^{h/2} (Y_1 + Y_2)$.
\vskip.3cm

Finally, assume that $ \OnePM \cdot e^{h/2}( Y_1 + Y_2) <_\Phi t <_\Phi  I_n$. By Lemma~\ref{4.1} we have that $\InterSec{p_1p_2} = \emptyset$ and therefore
\begin{equation} \mu_{\alpha}  ( \Part{p_1p_2}) \leq
\mu_{\alpha}  (Z_h (p_1)) =\Theta (1) \cdot e^{(  \frac{1} {2}  -\alpha)R + \alpha y(p_1)}.
\end{equation}
Since $n = \nu e^{R/2}$, the above expression is $O(1) \cdot n^{1-2\alpha} \cdot Y_1^{2\alpha}$, which also yields \eqref{eq:far_apart}.
Combining the three cases together we deduce part (ii).
\end{proof}

\section{Proof of Theorem  ~\ref{prop:var_Shat} }  \label{sect4}
A central tool in the proof of our main results
is the Palm theory for Poisson processes (see~\cite{bk:Penrose, bk:StKendMecke, bk:LastPenrose}).
Let $\Space$ be a measurable space and $\mathcal{N} (\Space)$ the set of all
locally finite point configurations on $\Space$.
For a Poisson point process $\mathcal{P}$ on $\Space$ with intensity $\rho$ and a measurable non-negative function
$h: \Space^r \times\mathcal{N} (\Space) \rightarrow [0, \infty)$ the Campbell-Mecke formula (cf. Theorems 4.1, 4.4 of ~\cite{bk:LastPenrose}) gives
\begin{equation} \label{eq:Campbell-Mecke}
\begin{split}
& \mathbb{E}  \left[  \sum_{x_1,\ldots, x_r \in \mathcal{P}}^{\neq} h (x_1, \ldots, x_r,
\mathcal{P}) \right] \\
& = \int_{\Space} \cdots \int_{\Space}
\Ex {h(x_1,\ldots, x_r,\mathcal{P}\cup \{x_1,\ldots, x_r\})} \rho (x_1) \cdots \rho (x_r) dx_1 \cdots dx_r,
\end{split}
\end{equation}
where the sum ranges over all pairwise distinct $r$-tuples of points of $\mathcal{P}$.

Equation~\eqref{eq:Campbell-Mecke} can be used to calculate
$\Var [X]$ where $X = \sum_{p \in \mathcal{P}} \xi (p,\mathcal{P})$ for some score function
$\xi (p,\mathcal{P}\cup \{p\})$ on $\Space$ taking values in $\{0,1\}$. With $c^{\xi} (x_1,x_2) = c^{\xi}(x_1,x_2;\mathcal{P})$ (cf.~\eqref{Cov}), the definition of the variance together with ~\eqref{eq:Campbell-Mecke} yield:
\begin{align} \label{eq:variance}
\Var [X] &= \int_{\Space} \int_{\Space} c^{\xi} (x_1,x_2) \rho(x_1) \rho(x_2) dx_1dx_2
 +  \int_{\Space} \Ex{\xi (x, \mathcal{P}\cup \{x\})^2}\rho(x) dx  \nonumber \\
 &=\int_{\Space} \int_{\Space} c^{\xi} (x_1,x_2) \rho(x_1) \rho(x_2) dx_1dx_2
+ \mathbb{E}[X],
\end{align}
where the last equality holds since $\xi^2 = \xi$ (in fact the first equality does not require that the score function is an indicator random variable, but this is the case throughout our paper).

\subsection{Proof of expectation asymptotics ~\eqref{expected-iso}}

\begin{lemma} \label{muasym} 
 Uniformly for $p \in {D}([0, H])$ we have
\begin{equation} \label{muident}
\mu_{\alpha}  ( B(p) ) = \gamma  e^{y(p)/2} +o(1).
\end{equation}
\end{lemma}

\noindent{\em Proof.}  We use the inclusions in~\eqref{eq:inclusion_mapped}.  By Lemma \ref{lem:relAngle_generic}(iii)
 we may put $C := C(n) := 5 \log R$ and  $\eps = O(e^{-5 \log R})$.  Now \eqref{eq:inclusion_mapped} yields:
\begin{equation} \label{eq:mu_ball_bound}
\begin{split}
&\mu_{\alpha}  ( B(p) ) \leq \mu_{\alpha}  (\BallUp)\\
&=  2\cdot  \OneUp\cdot \beta \int_0^{R-y(p)-C} e^{ \frac{1}{2} (y(p) + y) } e^{-\alpha y} dy  + \beta \cdot 2I_n \int_{R-y(p)-C}^R e^{-\alpha y} dy \\
& =  2\cdot \OneUp\cdot \beta e^{ \frac{y(p)} {2} } \int_0^{R-y(p)-C} e^{( \frac{1} {2}  -\alpha ) y} dy +\frac{\beta \pi}{\alpha}
e^{ \frac{R} {2} } \left(e^{-\alpha (R-y(p) -C)} - e^{-\alpha R}\right) \\
&=\OneUp \cdot \gamma e^{ \frac{ y(p)} {2} }\left(1- e^{(\frac{1} {2}  - \alpha) (R- y(p)-C)}\right) +\frac{\beta
\pi}{\alpha} e^{ \frac{R} {2} } \left(e^{-\alpha (R-y(p) -C)} - e^{-\alpha R}\right).
\end{split}
\end{equation}
where we recall that $2I_n = \pi e^{R/2}$ and where $\gamma$ and $\beta$ are related via $\gamma := 4 \beta/ (2\alpha - 1)$.
Recalling $\alpha \in (1/2, \infty)$, $y(p) \in (0, H), H = o(R)$, and $C= o(R)$, it follows
that uniformly over all such $p$
$$e^{ \frac{y(p)} {2}  + (\frac{1} {2}  - \alpha) (R- y(p)-C)}=o(1), \ \ e^{ \frac{R} {2} -\alpha (R-y(p) -C)} =o(1), \ \text{and} \  e^{\frac{R}{2} - \alpha R} =o(1).$$
We conclude that $\mu_{\alpha}  (B(p) ) \leq \OneUp \cdot \gamma  e^{ \frac{ y(p)} {2} } +o(1).$
(Recall that $\OneUp  := 1 + \eps$ and $\eps = O(e^{-5\log R})$.)
Notice that $y(p)/2 \in [0, 2 \log R)$ since $y(p) \in [0,  H)$.
Thus  $\eps \cdot e^{y(p)/2} = o(e^{-\frac12 \log R})$, uniformly over all such $p$,
whereby
\begin{equation} \label{eq:expec_score_Pab_up}
\mu_{\alpha}  (B(p) ) \leq \gamma  e^{ \frac{ y(p)} {2} } +o(1).
\end{equation}

To obtain a lower bound, we use the first inclusion in ~\eqref{eq:inclusion_mapped}:
\begin{align*}
\mu_{\alpha}  ( B(p) ) & \geq \mu_{\alpha}  (\BallDown )\\
& =  2\cdot \OneLow \cdot \beta \int_0^{R-y(p)-C} e^{ \frac{1}{2} (y(p) + y) } e^{-\alpha y} dy  \\
&= 2 \cdot \OneLow\cdot \beta e^{ \frac{ y(p)} {2} } \int_0^{R-y(p)-C} e^{( \frac{1} {2}  -\alpha ) y} dy \\
& = \OneLow \cdot  \gamma  e^{ \frac{y(p)} {2} }\left(1- e^{( \frac{1} {2}  - \alpha) (R- y(p)-C)}\right).
\end{align*}
Using again that $\eps = O(e^{-5\log R})$ and  $y(p)/2 \in [0, 2 \log R)$ we deduce
a matching lower bound:
\begin{equation} \label{eq:expec_score_Pab_low}
\mu_{\alpha}  ( B(p) ) \geq
\gamma  e^{  \frac{y(p)} {2} } +o(1).
\end{equation}
Combining~\eqref{eq:expec_score_Pab_up} with~\eqref{eq:expec_score_Pab_low} shows \eqref{muident}, as desired.  \qed

\vskip.2cm
We now establish expectation asymptotics ~\eqref{expected-iso}.  Since
$$
{\mathbb{E}} [\txi^{iso}(p, (\tilde{\cal P}_{\alpha}  \cap D) \cup \{p\})] = \exp( - \mu_{\alpha}  ( B(p) ) )
$$
it follows from \eqref{muident} that
\begin{equation} \label{eq:expec_score_Pabaa}
{\mathbb{E}}[ \txi^{iso}(p, (\tilde{\cal P}_{\alpha}  \cap D) \cup \{p\} )]  \sim \exp \left( - \gamma  e^{ \frac{y(p)} {2} }\right)
\end{equation}
uniformly over all $p \in \R{[0,H]}$. The Campbell-Mecke formula  \eqref{eq:Campbell-Mecke} and \eqref{eq:expec_score_Pabaa} yield
\begin{equation*}
\begin{split}
 \mathbb{E} [\tS^{iso}_H( \tilde{\cal P}_{\alpha}  \cap D) ] & =  \beta \int_{-I_n}^{I_n}\int_0^H  {\mathbb{E}}
[ \txi^{iso} ((x,y), (\tilde{\cal P}_{\alpha}  \cap D) \cup\{ (x,y)\} )]  e^{-\alpha y} dx dy \\
& \sim \beta \cdot 2I_n \int_0^H e^{-\gamma e^{ y/2} } e^{-\alpha y}dy
\sim \beta   \cdot \pi e^{R/2} \int_0^\infty e^{-\gamma e^{y/2 }} e^{-\alpha y} dy \\
& =2\alpha  n \int_0^\infty e^{-\gamma e^{y/2 }} e^{-\alpha y} dy,
\end{split}
\end{equation*}
since  $\beta = 2\nu \alpha/ \pi$. 
By Lemmas \ref{lem:variances_tilde_checked} and \ref{lem:exp_var_equiv}, we deduce ~\eqref{expected-iso} as desired. \qed

\subsection{Upperbounding  $\Var [\tilde{S}^{iso}_H( \tilde{\cal P}_{\alpha}  \cap D)]$ }
We derive the asymptotics for $\Var [\tilde{S}^{iso}_H( \tilde{\cal P}_{\alpha}  \cap D)]$ in two steps.
First, in this subsection we provide an upper bound via  the Poincar\'e inequality.  It turns out that
this is tight up to multiplicative constants. The next subsection provides a matching lower bound
for $\Var [\tilde{S}^{iso}_H( \tilde{\cal P}_{\alpha}  \cap D)] $ using the geometry of the intersection of hyperbolic balls
obtained in Section \ref{Prep}.

Let $F$ be a functional on a space $\Space$ hosting a Poisson process $\cP$ of intensity measure
$\lambda$. For a point $p \in \Space$ we define the \emph{first order linear operator}
$\nabla_p F:= F(\cP \cup \{ p\}) - F(\cP)$.
Then  the Poincar\'e inequality (inequality (1.1) in~\cite{ar:LPS16}) states that
\begin{equation}\label{eq:Poincare} \Var F  \leq \Ex{\int_\Space (\nabla_p F(\cP))^2 \lambda(dp)}.\end{equation}
We now put
$$
F: = \tilde{S}^{iso}_H (\tilde{P}_{\alpha} \cap D), \ \Space:=\Rcal, \ \cP := \tilde{\cal P}_{\alpha} \cap D \ \mbox{and} \ \lambda:= \mu_\alpha.
$$
Note that $|\nabla_p F|$ is stochastically dominated from above by the number of points of
$( \tilde{\cal P}_{\alpha} \cap D)  \cup \{p\}$ in $B(p)$.
By Lemma~\ref{muasym} we deduce that $|B(p) \cap (\tilde{\cal P}_{\alpha} \cap D) |$ is Poisson-distributed with parameter equal to
$$ \mu_{\alpha}  ( B(p) ) \leq \gamma'  e^{ \frac{y(p)} {2}} ,$$
uniformly over all $p \in D$, for some constant $\gamma'>0$.
Thus,
$$\Ex{ (\nabla_p F(\cP))^2} = O(1) \cdot e^{y(p)}
$$
which implies that
\begin{align*}
\Ex{ \int_\Space (\nabla_p F(\cP))^2 \lambda (dp)}  & \leq  \int_\Space \Ex{(\nabla_p F(\cP))^2} \lambda(dp)
 = O(1) \cdot n \int_0^{R} e^{(1-\alpha )y}dy.
 \end{align*}
In other words, evaluating the integral and using $e^{(1 - \alpha)R} = n^{2(1 - \alpha)}$ in the range $\alpha \in (1/2, 1)$, we get  
\begin{align}
\Var [\tilde{S}^{iso}_H (\tilde{P}_{\alpha} \cap D)]
&= \begin{cases} \label{ubound}
O (n^{3-2\alpha}) & \ \alpha \in ( \frac{1} {2} ,1) \\
O (n R) & \ \alpha =1 \\
O (n) & \ \alpha \in (1, \infty)
 \end{cases}.
\end{align}


\subsection{Lowerbounding  $\Var [\tilde{S}^{iso}_H (\tilde{P}_{\alpha} \cap D)]$}
Recall the definition of $ c^{\tilde{\xi}^{iso}}(p_1,p_2)$ at \eqref{Cov}.  By ~\eqref{eq:variance}, we have $\Var [\tilde{S}^{iso}_H(\tilde{\cal P}_{\alpha} \cap D)  ] = V_1 + V_2$ with
$$
V_1 : =\beta^2  \int_{-I_n}^{I_n}\int_0^H \int_{-I_n}^{I_n}\int_0^H   c^{\tilde{\xi}^{iso}} ((x_1,y_1),(x_2,y_2))
e^{-\alpha y_1} e^{-\alpha y_2} dy_2 dx_2 dy_1 dx_1
$$
and
$$
V_2 :=\mathbb{E} [ \tilde{S}^{iso}_H(\tilde{\cal P}_{\alpha} \cap D)]  = \mathbb{E} \left[ \sum_{p \in \tilde{\cal P}_{\alpha} \cap D([0,H])  }
\txi^{iso}(p,  \tilde{\cal P}_{\alpha} \cap D) \right] .
$$
Since $V_2 \geq 0$ it  suffices to provide a lower bound on $V_1$ matching that at \eqref{ubound}. Put
\begin{align*}
& \Cov^- (x_1, y_1, y_2, t) \\
& :=  \mathbb{E} [ \txi^{iso}(p_1, (\tilde{\cal P}_{\alpha} \cap D) \cup \{p_1\})]
\mathbb{E} [ \txi^{iso}(p_2, (\tilde{\cal P}_{\alpha} \cap D)  \cup \{p_2\})]
\left(  \exp \left(\mu_{\alpha}  ( \Part{p_1p_2}^-) \right)  -1 \right),
\end{align*}
where $x(p_1 ) = x_1$, $y (p_1) = y_1$, $y(p_2) = y_2$ and $| x (p_2 ) - x (p_1) |_\Phi = t$. By symmetry, it suffices to consider the case where $y_2 \leq y_1$ and $x(p_2) >_\Phi x(p_1)$.
Indeed, this is one of four possible cases regarding the relative positions of $p_1,p_2 \in \R{[0,H]}$ and it accounts for the pre-factor 4 appearing  in front of our upcoming lower bounds.

Note that if $x(p_1)=0$ and $x (p_2 ) - x (p_1)  >_\Phi  0$, then in fact
$|x(p_2) - x(p_1)|_\Phi = x(p_2) - x(p_1)$.
Considering points $p_1$ and $p_2$ such that
$x(p_1 ) = 0$, $y (p_1) = y_1$, $y(p_2) = y_2$ and $x (p_2 ) - x (p_1)  = t \in (0, I_n)$ we have
by \eqref{eq:cov_intersection} the bound
$ c^{\tilde{\xi}^{iso}} (p_1,p_2) \geq \Cov^- (x_1, y_1, y_2, t).$

Therefore,
\begin{equation*}
V_1 \geq V_1^- :=4 \beta^2 \int_{-I_n}^{I_n} \int_0^H \int_0^{y_1}\int_{0}^{I_n}
\Cov^- (x_1, y_1, y_2, t) e^{-\alpha y_2} e^{-\alpha y_1} dt dy_2 d y_1dx_1.
\end{equation*}

\vskip.3cm
We will drop the $-$  sign and write $\Cov (x_1, y_1, y_2, t) := \Cov^- (x_1, y_1, y_2, t).$
Note that  $\Cov (x_1, y_1, y_2, t) $ does not depend on $x_1$ as the Poisson process $\tilde{\cal P}_{\alpha} \cap  D$  is stationary with respect to the spatial $x$-coordinate. Therefore, we can write
 \begin{equation*}
V_1^- = 8 \beta^2  I_n \int_0^H \int_0^{y_1}\int_{0}^{I_n}
\Cov (0, y_1, y_2, t) e^{-\alpha y_2} e^{-\alpha y_1}  dt dy_2 d y_1.
\end{equation*}
We change variables and, as before, put $\Y_i = e^{y_i/2}, i = 1,2$. Hence, $2 d \Y_i =  e^{y_i/2} dy_i$ and $dy_i =
2 \Y_i^{-1}  d\Y_i$. Also, $e^{-\alpha y_i} = \Y_i^{-2\alpha}$.
Moreover, as $y_1$ ranges from $0$ to $H$, the variable $\Y_1$ ranges from $1$ to $e^{H/2} = e^{2\log R} = R^2$.
Thus,
 \begin{equation} \label{eq:variance_total}
 V_1^- =32 \beta^2 I_n \int_1^{R^2} \int_1^{\Y_1}\int_{0}^{I_n}
\Cov (0, \Y_1, \Y_2, t) \frac{1}{\Y_1^{2\alpha +1}} \frac{1}{\Y_2^{2\alpha +1}}  dt d\Y_2 d \Y_1.
\end{equation}
To simplify notation we shall write
$$e^{h/2} := \OneLow \cdot e^{h/2}. $$
This amounts to transferring the term $\OneLow$ inside $h$ changing the constant $C$ to $C-2 \ln \OneLow$. It will make no difference.

Let us observe that
\be
\int_{0}^{I_n}  \Cov (0, \Y_1, \Y_2, t) dt \geq J_1 + J_2,
\ee
where, recalling $Y_{12}$ defined at \eqref{Yonetwo}, we have
\begin{align*}
J_1 &  :=   \int_0^{Y_{12}}  \Cov (0, \Y_1, \Y_2, t) dt,\ \mbox{and} \
J_2  :=  \int_{Y_{12} \vee \OneUp\cdot (\Y_1 + \Y_2)}^{e^{h/2}(\Y_1 - \Y_2)} \Cov (0, \Y_1, \Y_2, t) dt.
\end{align*}
By ~\eqref{eq:condition_inc}  the covariance is negative only when $t$ belongs to the range covered by the $J_1$ integral.
For $t \in (Y_{12}, I_n]$, the covariance is positive. Thus, for the range $(Y_{12}, I_n]$ it suffices to use the subset given by the smaller range of $J_2$ which in turn is covered by Lemma~\ref{lem:mu_intersection}(i).

For $k=1,2$ we set
\begin{equation} \label{eq:gen_integrals}
L_k  : = \int_1^{R^2} \int_1^{\Y_1} J_k \alpha^2 Y_1^{-2\alpha-1} Y_2^{-2\alpha-1} dY_2 dY_1.
\end{equation}

We  now  show that $|L_1| =O (1)$ for all $\alpha \in (1/2, \infty)$, whereas we derive lower bounds on $L_2$ which match the upper bounds in \eqref{ubound}.

\subsubsection{Calculating integral $L_1$}

Formula \eqref{eq:expec_score_Pabaa} and the second part of the covariance formula \eqref{eq:cov_intersection} give for all $t \in [0, \Y_{12}]$
\be
\Cov (0, \Y_1, \Y_2, t) \sim  - \exp (-\gamma e^{y_1/2}) \exp (-\gamma e^{y_2/2}) = - \exp \left(-\gamma \left(Y_1 + Y_2 \right) \right),
\ee
uniformly for all $Y_2\leq Y_1\leq R^2$, whereby
$$
J_1=\int_0^{Y_{12}}  \Cov (0, \Y_1, \Y_2, t) dt \sim
- Y_{12} \exp \left(-\gamma \left(Y_1 + Y_2 \right) \right).
$$
Therefore, since $Y_{12} = Y_1 Y_2 (1+o(1))$ by~\eqref{Yonetwo}, we eventually obtain:
\begin{equation*}
\begin{split}
L_1 &= \int_1^{R^2} \int_1^{\Y_1} J_1 \frac{2}{\Y_1^{2\alpha+1}} \frac{2}{\Y_2^{2\alpha+1}}
 d\Y_2 d\Y_1
 \sim  -  4\int_1^{R^2} \int_1^{\Y_1}  \frac{e^{-\gamma Y_1}}{\Y_1^{2\alpha}} \frac{e^{-\gamma Y_2}}{\Y_2^{2\alpha}}
d\Y_2 d\Y_1. 
\end{split}
\end{equation*}
Hence, we deduce for all $\alpha \in (1/2, \infty)$ that
$|L_1| = O(1).$

\subsubsection{The lower bound on integral $L_2$}
Let $s:= \max \{ 5\alpha/ (2\alpha -1),5\}$ and $r:= R^s e^{-h/2}$.
Given the domain $\Dom:=\{ (Y_1,Y_2) \ : \ 1 \leq Y_2 \leq Y_1 \leq R^2  \}$ consider the sub-domain
$$
\Dom' :=\{ (Y_1,Y_2) \in \Dom \ : \ Y_1 - Y_2 \geq r \}.
$$

It suffices to consider the contribution to $L_2$ that comes from  the domain  $\Dom'$.
That is, we will bound from below the integral
$$
L_2':=\int_{\Dom'} J_2 \frac{1}{\Y_1^{2\alpha+1}} \frac{1}{\Y_2^{2\alpha+1}}d\Y_2 d \Y_1.
$$

Combining ~\eqref{eq:cov_intersection},   Lemma~\ref{lem:mu_intersection}, \eqref{eq:expec_score_Pabaa} and recalling  $\eta (Y_2):= \gamma Y_2 e^{(1/2 -\alpha) h}$ we obtain
\begin{equation*}
\begin{split}
&\Cov (0, \Y_1, \Y_2, t) \\
& = \mathbb{E} [ \txi^{iso}(p_1, (\tilde{\cal P}_{\alpha} \cap D)  \cup\{p_1\} )]
 \mathbb{E} [ \txi^{iso}(p_2, (\tilde{\cal P}_{\alpha} \cap D) \cup \{p_2\} )]
\left( \exp \left(\kappa t^{1-2\alpha} - \eta (Y_2) \right) - 1\right)  \\
&  {\sim}  \exp (-\gamma (Y_1 +Y_2))  \left( \exp \left(\kappa t^{1-2\alpha} - \eta(Y_2) \right) - 1\right).
\end{split}
\end{equation*}
For simplicity, we set $t_1 := Y_{12} \vee   1_{+ \epsilon} (\Y_1 + \Y_2)$ and $t_2:=e^{h/2}(\Y_1 - \Y_2)$,
whereby
\begin{equation} \label{eq:J_3}
\begin{split}
 J_2 &=\int_{t_1}^{t_2}  \Cov (0, \Y_1, \Y_2, t) dt  \\
 & \sim
 \exp (-\gamma (Y_1 +Y_2))
  \int_{t_1}^{t_2}
  \left( \exp \left(\kappa t^{1-2\alpha} -\eta (Y_2) \right) - 1\right) dt.
\end{split}
\end{equation}

Consider the integral in ~\eqref{eq:J_3} when  $(Y_1,Y_2) \in \Dom'$.
The following lemma shows that its value changes radically as $\alpha$ crosses 1.  The regimes for this lemma induce three regimes for $L_2'$.

\begin{lemma} \label{clm:integral_cases_lb}  There is a $\delta > 0$ such that 
for any $n$ sufficiently large, any $(Y_1,Y_2) \in \Dom'$ and $\alpha \in (1/2, \infty)$, we have
\begin{equation*}
 \int_{t_1}^{t_2} t^{1-2\alpha}dt \geq
\begin{cases}
2 \alpha (1+\delta) t_2^{2(1-\alpha)} & \alpha \in (\frac{1} {2} ,1) \\
\frac{1}{5} \ln t_2 & \alpha =1 \\
\frac{1}{4(\alpha - 1)} t_1^{2(1-\alpha)} & \alpha \in (1, \infty)
\end{cases}.
\end{equation*}
\end{lemma}
\noindent
{\em Proof.} Elementary integration gives the three different cases:
\begin{equation*}
\begin{split}
\int_{t_1}^{t_2} t^{1 - 2\alpha} dt =
\begin{cases}
\frac{1}{2(1-\alpha)} \left( t_2^{2(1-\alpha)} - t_1^{2(1-\alpha)}\right) & \alpha \in ( \frac{1} {2} ,1)  \\
\ln \frac{t_2}{t_1} & \alpha =1 \\
\frac{1}{2(\alpha - 1)} \left( t_1^{2(1-\alpha)} - t_2^{2(1-\alpha)}\right) & \alpha \in (1, \infty)
\end{cases}.
\end{split}
\end{equation*}
By definition, for any $(Y_1,Y_2) \in \Dom'$, we have that $t_2 > t_1 + R^5/2$, whereas $t_1 = Y_1 Y_2 \vee (Y_1 +Y_2) \leq R^4$, for $n$ sufficiently large.
Thus, $t_2 / t_1 \geq R/2 \to \infty$ as $n\to \infty$.
These facts imply that if $\alpha \in (1/2, 1)$, then for some $\delta \in (0, \infty)$
$$t_2^{2(1-\alpha)} - t_1^{2(1-\alpha)} \stackrel{4 \alpha (1-\alpha) <1}{>} (1+\delta) 4\alpha (1-\alpha) t_2^{2(1-\alpha)}, $$
whereas if $\alpha =1$, then
$$ \ln \frac{t_2}{t_1} = \ln t_2 \left( 1 - \frac{\ln t_1}{\ln t_2}\right) >
\ln t_2 \left( 1 - \frac{\ln R^4}{\ln{R^5}}\right) = \frac{1}{5}\ln t_2.
$$
Finally if $\alpha \in (1, \infty)$, then $ t_1^{2(1-\alpha)} - t_2^{2(1-\alpha)} > \frac{1}{2} t_1^{2(1-\alpha)}$, 
for $n$ sufficiently large. The lemma follows.  \qed

\vskip.3cm

For $(Y_1, Y_2) \in \Dom'$  and $n$ sufficiently large  we have $Y_{12} \vee \OneUp \cdot (\Y_1 + \Y_2) \leq R^4$. In this domain
the definition of $s$ gives
\begin{equation}\label{eq:b2-b1_diff}
 e^{h/2}(\Y_1 - \Y_2) - (Y_{12} \vee \OneLow \cdot (\Y_1 + \Y_2) )\geq R^5 - R^4> R^5/2.
\end{equation}

\subsubsection{Three regimes for integral $L_2'$}

\noindent \emph{4.3.3 (a) The integral $L_2'$,   $\alpha \in (1/2,1)$}.
Recalling the definitions of $t_1$ and $t_2$ and appealing to Lemma \ref{clm:integral_cases_lb}, we deduce the following lower bound
\begin{equation*}
\begin{split}
&\int_{t_1}^{t_2}
  \left( \exp \left(\kappa t^{1-2\alpha} -\eta(Y_2) \right) - 1\right) dt \geq  \int_{t_1}^{t_2}
  \left( \kappa t^{1-2\alpha} -\eta (Y_2)\right) dt \\
  &\geq 2\alpha (1+\delta) \kappa e^{h(1-\alpha )} (\Y_1 - \Y_2)^{2(1-\alpha)}
  - \gamma (Y_1 -Y_2)Y_2 e^{h(1-\alpha)}.
\end{split}
\end{equation*}
Recall from \eqref{kap} that $\kappa: =\OneLow\cdot  \frac{\gamma}{4\alpha} ((Y_1+Y_2)^{2\alpha}-(Y_1- Y_2)^{2\alpha})$.
For any $\eps>0$ sufficiently small (in terms of $\delta$),
we have  $\OneLow \cdot (1+\delta) > 1+ \delta/2$.
Therefore,
\begin{equation} \label{eq:J_3low}
\begin{split}
&2 \alpha(1+ \frac{\delta} {2} ) \kappa  (\Y_1 - \Y_2)^{2(1-\alpha)}  -\gamma (Y_1 -Y_2)Y_2 \\
&  \geq \gamma (Y_1- Y_2)\left( \frac{1+\delta/2}{2}
\left( \left(\frac{Y_1 +Y_2}{Y_1- Y_2}\right)^{2\alpha} (Y_1- Y_2) - (Y_1-Y_2) \right)-Y_2 \right) \\
&\geq \gamma (Y_1- Y_2)\left( \frac{1+\delta/2}{2}
\left( \left(\frac{Y_1 +Y_2}{Y_1- Y_2}\right) (Y_1- Y_2) - (Y_1-Y_2) \right)-Y_2 \right) \\
&=\gamma (Y_1- Y_2)\left( \frac{1+\delta/2}{2}
\left( Y_1 +Y_2- (Y_1-Y_2) \right)-Y_2 \right) \\
&=\gamma (Y_1- Y_2)\left( \frac{2(1+\delta/2)}{2} Y_2 -Y_2 \right)
 =
\frac{\gamma \delta}{2} (Y_1- Y_2)Y_2
 := Q_1(Y_1, Y_2) > 0.
\end{split}
\end{equation}

We then deduce that
\begin{equation}  \label{Qbd}
\frac{J_2} {e^{h(1-\alpha)} \exp \left( -\gamma (Y_1 + Y_2)\right)} \geq Q_1(Y_1,Y_2).
\end{equation}
Recall that $h=R- y(p_1) - C + 2\ln \OneLow$. This implies that
$$e^{h(1-\alpha)} = e^{R(1-\alpha)} e^{-(y(p_1)+C - 2\ln \OneLow)(1-\alpha)} \stackrel{Y_1 = e^{y(p_1)/2}}{=}   \Omega(1) \cdot
e^{R(1-\alpha)} \cdot Y_1^{-2(1-\alpha)}.$$
The above bounds imply that
\begin{equation*}
L_2' = \Omega(1) \cdot e^{R(1-\alpha)} \int_{\Dom'}
\exp \left( -\gamma (Y_1 + Y_2)\right) Q_1 (Y_1,Y_2) \frac{Y_1^{2(1-\alpha)}}{\Y_1^{-2\alpha+1}}
\frac{1}{\Y_2^{2\alpha+1}} d\Y_2 d \Y_1.
\end{equation*}
Thus, for $\alpha \in (1/2, 1)$ we have
\begin{equation} \label{eq:L_3-a<1-final} L_2 \geq L_2' = \Omega \left( e^{R(1-\alpha)}\right).
\end{equation}

\noindent \emph{4.3.3 (b) The integral $L_2'$, $\alpha = 1$.} Note first that
\begin{equation} \label{eq:kappa_alpha=1}
\begin{split}
\kappa &=  \OneLow \cdot \frac{\gamma}{4\alpha} ((Y_1+Y_2)^{2\alpha}-(Y_1- Y_2)^{2\alpha}) = \OneLow \cdot \frac{\gamma}{4}
((Y_1+Y_2)^2-(Y_1-Y_2)^2) \\
&=\OneLow \cdot \gamma Y_1 Y_2.
\end{split}
\end{equation}

In this case, the integral in~\eqref{eq:J_3} is bounded  below as follows:
\begin{equation*}
\begin{split}
&\int_{t_1}^{t_2} (e^{ \frac{\kappa} {t}  - \eta (Y_2)} -1 )dt  \geq \int_{t_1}^{t_2}(\frac{\kappa} {t}  - \eta (Y_2) )dt \\
& \geq \frac{\kappa}{5}\ln \left( e^{\frac{h} {2} }(\Y_1 - \Y_2)\right)
-\gamma Y_2 e^{-\frac{h} {2} } \int_0^{e^{  \frac{h} {2}}(\Y_1 - \Y_2)} dt  \\
& = \frac{\kappa}{5}\ln \left( e^{\frac{h} {2} }(\Y_1 - \Y_2)\right)   -\gamma Y_2 (Y_1- Y_2) \\
&= \kappa \frac{h}{10} + \frac{\kappa}{5} \ln \left(\Y_1 - \Y_2\right) -
\gamma Y_2 (Y_1 - Y_2) \\
& \stackrel{\eqref{eq:kappa_alpha=1},Y_1\geq Y_2}{\geq} \OneLow \cdot
 \frac{\gamma}{10} Y_1 Y_2 \frac{h}{2} +\OneLow \cdot  \frac{\gamma}{5} Y_1 Y_2
\ln \left( Y_1 - Y_2\right) -
\gamma Y_1 Y_2\\
&= \OneLow \cdot \frac{\gamma}{5} Y_1 Y_2 \left(\frac{h}{4} +
\ln \left( Y_1 - Y_2 \right) - 5 \cdot \OneLow^{-1} \right) \\ 
&\stackrel{\eps < 1/8}{\geq} \OneLow \cdot \frac{\gamma}{5} Y_1 Y_2 \left(\frac{h}{4} +
\ln \left( Y_1 - Y_2 \right) - 6 \right) \
\end{split}
\end{equation*}
where the last inequality holds for $n$ sufficiently large, if we put 
$\eps = O(e^{-5 \log R})$ (cf.~Lemma \ref{lem:relAngle_generic}(iii)).
In particular, if we let $(Y_1,Y_2)\in \Dom'' := \{(Y_1,Y_2) \in \Dom' \ : \ Y_1 - Y_2 > e^6 \}$,
then
$$\int_{t_1}^{t_2} (e^{ \frac{\kappa} {t} - \eta (Y_2)} -1 )dt  \geq \OneLow \cdot
\gamma Y_1 Y_2 \frac{h}{20}.
$$
Also, recall that $h:= R- y(p_1) - C + 2\ln \OneLow$.  Since $y(p_1) <H$, it follows that for $n$ sufficiently large we have
$h  \in (R/2, R]$.
Combining this observation with the above lower bound,~\eqref{eq:J_3} yields
$$ J_2 \geq \OneLow \cdot \frac{h}{2} e^{-\gamma (Y_1 + Y_2)}  \frac{\gamma}{20} Y_1 Y_2
\geq   \frac{R \gamma }{80} e^{-\gamma (Y_1 + Y_2)}   Y_1 Y_2.$$
Therefore
\begin{equation*}
\begin{split}
L_2' &=   \int_{\Dom'}  J_2  \frac{1}{Y_1^{2\alpha+1}}
\frac{1}{Y_2^{2\alpha+1}} dY_2 d Y_1
\geq  \int_{\Dom''}  J_2  \frac{1}{Y_1^{2\alpha+1}}
\frac{1}{Y_2^{2\alpha+1}}dY_2 d Y_1
\\
& \geq  \frac{R \gamma}{80}  \int_e^{R^2} \int_1^{Y_1 -e}
e^{-\gamma (Y_1 + Y_2)}   Y_1 Y_2
\frac{1}{Y_1^{2\alpha+1}} \frac{1}{Y_2^{2\alpha+1}} dY_2 d Y_1  =\Omega \left( R \right).
\end{split}
\end{equation*}

\noindent \emph{4.3.3 (c) The integral $L_2'$, $\alpha \in (1, \infty).$} Recall by \eqref{eq:b2-b1_diff}
that for any $(Y_1,Y_2) \in \Dom'$, we have
$$ e^{h/2}(Y_1 - Y_2) >  Y_{12} \vee \OneLow \cdot (\Y_1 + \Y_2) + R^s,$$
for some $s= s(\alpha ) \geq 5$.
Since $Y_{12} \vee \OneLow \cdot (Y_1 + Y_2) \leq R^4$ when $n$ is sufficiently large, we get
$$ e^{h/2}(Y_1 - Y_2) > 2  (Y_{12} \vee \OneLow \cdot (Y_1 + Y_2)). $$

Using the third  case in Lemma  ~\ref{clm:integral_cases_lb},
the integral in~\eqref{eq:J_3} is bounded from below as follows:
\begin{equation*}
\begin{split}
&\int_{t_1}^{t_2}
  \left( \exp \left(\kappa t^{1-2\alpha} -\eta(Y_2) \right) - 1\right) dt \geq
   \int_{t_1}^{t_2}
  \left( \kappa t^{1-2\alpha} -\eta (Y_2)\right) dt \\
  &\geq
  \frac{\kappa}{4(\alpha-1)} (Y_{12} \vee (Y_1 + Y_2))^{2(1-\alpha)}
  - \gamma (Y_1 -Y_2)Y_2 e^{h(1-\alpha)}  \\
  &\geq   \frac{\kappa}{4(\alpha-1)} (Y_{12} \vee (Y_1 + Y_2))^{2(1-\alpha)}
  - \gamma R^4 e^{h(1-\alpha)}  \\
  &= \frac{\kappa}{4(\alpha-1)} (Y_{12} \vee (Y_1 + Y_2))^{2(1-\alpha)} -o(1).
\end{split}
\end{equation*}
Hence,
\begin{equation*}
\begin{split}
L_2' & =   \int_{\Dom'}  J_2  \frac{1}{Y_1^{2\alpha+1}}
\frac{1}{Y_2^{2\alpha+1}} dY_2 d Y_1 \\
& \geq
\frac{1}{4(\alpha-1)}
\int_{\Dom'} \kappa \left[(Y_{12} \vee (\OneLow \cdot (Y_1 + Y_2))\right]^{2(1-\alpha)}
e^{-\gamma (Y_1 + Y_2)}
\frac{1}{Y_1^{2\alpha+1}} \frac{1}{Y_2^{2\alpha+1}} dY_2 d Y_1 \\
& \ \ \ \ \ - o(1)
\int_{\Dom'}
e^{-\gamma (Y_1 + Y_2)}
\frac{1}{Y_1^{2\alpha+1}} \frac{1}{Y_2^{2\alpha+1}} dY_2 d Y_1 =\Omega(1).
\end{split}
\end{equation*}

\subsection{Proof of growth rates  for $\Var[S^{iso}( {\cal P}_{\alpha, n} )]$}
We have now estimated the two summands that bound the
main term of $\Var[\tilde{S}^{iso}_H( \tilde{P}_{\alpha} \cap D )]$ from below. Our findings are summarised as follows: 
\begin{equation} \label{eq:L_3}
|L_1|  = \Theta (1),   \ \mbox{and} \  L_2= \begin{cases}
\Omega \left( e^{R(1-\alpha )}\right) & \alpha \in ( \frac{1} {2} ,1) \\
\Omega \left( R \right) &\alpha =1 \\
\Omega \left( 1 \right) & \alpha \in (1, \infty)
\end{cases}.
\end{equation}

By \eqref{eq:variance_total} we have $V_1= \Omega \left( I_n (L_1 + L_2) \right),$
and $2I_n= \pi e^{R/2} = \Theta (n)$.
Therefore,
$$ V_1 = \begin{cases}
\Omega \left( n^{1 + 2(1-\alpha)}\right) & \alpha \in ( \frac{1} {2} , 1) \\
\Omega \left( n R \right) & \alpha =1 \\
\Omega \left( n \right) & \alpha \in (1, \infty)
\end{cases}.$$
As $ \Var[ \tilde{S}^{iso}_H( \tilde{P}_{\alpha} \cap D  )] \geq V_1$, we finally deduce that
\begin{equation} \label{eq:variances_hat_Siso}
 \Var[ \tilde{S}^{iso}_H( \tilde{P}_{\alpha} \cap D  )]
 = \begin{cases}
\Omega \left( n^{3-2\alpha}\right) & \alpha \in (\frac{1} {2} , 1) \\
\Omega \left( nR \right) & \alpha =1 \\
\Omega \left( n \right) & \alpha \in (1, \infty)
\end{cases}.
\end{equation}

Combining \eqref{ubound} and \eqref{eq:variances_hat_Siso}, and recalling Corollary \ref{cor1}, we thus establish the desired growth rates  for $\Var[S^{iso}( {\cal P}_{\alpha, n} )]$, completing the proof of \eqref{Variso}.

\section{Proof of Theorem \ref{extVar}}

\subsection{Proof of expectation asymptotics for $S^{ext}( {\cal P}_{\alpha, n})$ }

By Lemmas \ref{lem:variances_tilde_checked} and \ref{lem:exp_var_equiv} it suffices to compute
$\lim_{n \to \infty} n^{-1}  \mathbb{E} [\tilde{S}^{ext}_H( \tilde{P}_{\alpha} \cap D )]$.
Given  a point $p \in D([0,H])$ we have
\begin{align*}
\mu_{\alpha}  ( B(p) \cap {D}([0,y(p)])) & \leq \mu_{\alpha}  (\BallUp
\cap {D}([0,y(p)]))\\
&=  2\cdot  \OneUp\cdot \beta \int_0^{y(p)} e^{ \frac{y(p)} {2}  + \frac{y} {2} } e^{-\alpha y} dy \\
&= 2\cdot \OneUp \cdot \frac{2\beta}{2\alpha -1} e^{ \frac{y(p)} {2}}  (1- e^{( \frac{1} {2} -\alpha)y(p)}).
\end{align*}
Similarly, we have the lower bound
\begin{align*}
\mu_{\alpha}  ( B(p) \cap {D}([0,y(p)])) &  \geq \mu_{\alpha}  (\BallDown
\cap {D}([0,y(p)]))\\
&= 2\cdot \OneLow \cdot \frac{2\beta}{2\alpha -1} e^{  \frac{y(p)} {2}}  (1- e^{( \frac{1} {2}  -\alpha)y(p)}).
\end{align*}
Taking again $\eps = O(e^{-5 \log R})$, and $\gamma = 4 \beta/(2\alpha -1)$,
we then deduce that uniformly over all $p \in \R{[0,H]}$
$$
 \mu_{\alpha}  ( B(p) \cap {D}([0,y(p)])) = \gamma \cdot
e^{\frac{ y(p)} {2}} (1- e^{(\frac{1} {2}  -\alpha)y(p)}) +o(1).
$$
Therefore,
\begin{equation} \label{eq:expec_score_Pabaa_extreme}
\mathbb{E}[ \txi^{ext}(p, \tilde{\cal P}_{\alpha, D} \cup \{p\})]  \sim \exp \left( - \gamma  e^{\frac{y(p)} {2}} (1- e^{( \frac{1} {2} -\alpha)y(p)})\right)
\end{equation}
uniformly over all $p \in \R{[0,H]}$.
Hence, the Campbell-Mecke formula  \eqref{eq:Campbell-Mecke} yields
\begin{align}
n^{-1} \mathbb{E}[\tilde{S}^{ext}_H( \tilde{P}_{\alpha} \cap D )] &=  n^{-1} \beta \int_{-I_n}^{I_n}\int_0^H \mathbb{E}[
\txi^{ext} ((x,y), \tilde{\cal P}_{\alpha, D} \cup \{ (x,y)\} ) ] e^{-\alpha y} dx dy \nonumber  \\
& \sim \beta \cdot 2I_n n^{-1} \int_0^H e^{-\gamma e^{y/2}(1-e^{( \frac{1} {2} -\alpha)y}) } e^{-\alpha y}dy  \nonumber  \\
&\sim \beta   \cdot \pi e^{R/2} n^{-1} \int_0^\infty e^{-\gamma e^{y/2}(1-e^{(\frac{1} {2} -\alpha)y})} e^{-\alpha y} dy \nonumber  \\
& =2\alpha  \int_0^\infty e^{-\gamma e^{y/2}(1-e^{(\frac{1} {2} -\alpha)y})} e^{-\alpha y} dy
:= \mu  \label{eq:extreme_mu}
\end{align}
as desired.

\subsection{Proof of variance  asymptotics for $S^{ext}( {\cal P}_{\alpha, n})$ }

The determination of variance asymptotics for $S^{ext}( {\cal P}_{\alpha, n})$ is handled by extending existing stabilization methods. We show that when the constants describing the tail behavior of the stabilization radius at a point $p$ are allowed to grow exponentially fast with the height of $p$,
as at \eqref{stab-bounds} below, then one may nonetheless establish explicit variance asymptotics as $n \to \infty$, as shown in the analysis between \eqref{Sliv} and  \eqref{sigma3} below.

We first require several auxiliary lemmas.
For all $r > 0$ and $p:= (x(p), y(p)) \in \RR \times \RR^+$   we let $B(p, r)$ denote the closed Euclidean ball of radius $r$ centered at $p$.

The identity  \eqref{trunc} implies that for $y(p) \in [0, H]$ we have
$$
D(p) = \{ p' \ : \ y(p')  \leq y(p), \ |x(p') - x(p)|_\Phi < (1+ \lambda_n(p',p)) e^{\frac12 (y(p')+ y(p) - R)} \}. $$
We set $s_n:=s_n(p):=1+ \lambda_n(p',p)$, where $\lambda_n (p',p) = o(1)$. 

Put $p_0 := (0, y_0)$. We let $d_0$ be the Euclidean distance between $p_0$ and the point in $(\tilde{\cal P}_{\alpha} \cap D)  \cap D(p_0)$ which is closest to $p_0$. 
Set $d_n:= \text{diam}(D(p_0))$ and note $d_n \leq 2 s_n e^{y_0}$.
Now we put
$$
 R^\xi:= R^\xi( p_0,  (\tilde{\cal P}_{\alpha} \cap D) ) : = \begin{cases} d_0  & (\tilde{\cal P}_{\alpha} \cap D) \cap D(p_0)  \neq \{p_0 \}     \\
 d_n  &   (\tilde{\cal P}_{\alpha} \cap D) \cap D(p_0) = \{p_0\}
 \end{cases}.
 $$
 The extremality status of $p_0$ depends only on the point set  $(\tilde{\cal P}_{\alpha} \cap D)   \cap B(p_0, R^\xi )$ in the sense that points outside this set will not modify 
 $\txi^{ext}(p_0, \tilde{\cal P}_{\alpha} \cap D ) $.  In other words, 
$$
\txi^{ext}(p_0, \tilde{\cal P}_{\alpha} \cap D )  = \txi^{ext} \left(p_0, (\tilde{\cal P}_{\alpha} \cap D)   \cap B(p_0, R^\xi )\right),
$$
that is to say that $R^\xi$ is \emph{a radius of stabilization} for $\xi:= \txi^{ext}$.

Clearly, for $t \in (d_n, \infty)$, we have $\P (R^\xi(p_0,  \tilde{\cal P}_{\alpha} \cap D ) \geq  t) = 0$.  We seek to control
$\P (R^\xi(p_0, \tilde{\cal P}_{\alpha} \cap D ) > t), t \in [0, d_n]$, as a function of both $t$ and the height parameter $y_0$.
Put $c_0 := \sqrt{3} \beta(1 - e^{- 8 \alpha} )/\alpha$ and set
 $\phi(t):=  \min \{ \alpha t/4, c_0 \sqrt{t/3} \}$ for $t \in (0, \infty)$.  We assert there is a constant $c_1$ such that  for $y_0 \in [0, H]$ we have
\be \label{stab-bound}
\P(  R^\xi(p_0,  \tilde{\cal P}_{\alpha} \cap D ) \geq t ) \leq  c_1  \exp \left( {  \alpha y_0  \over 2} \right) \exp(-  \phi(t)  ), \ \ t \in [0, \infty).
\ee

\vskip.2cm

We first compute lower bounds on the $\mu_{\alpha} $ probability content of the regions
$$
R_t(p_0) :=  B(p_0,t) \cap D(p_0),  \ t \in [2 y_0, d_n].
$$

\begin{lemma} \label{lemm3} Let $y_0 \in [8, H]$. For all $n$ large we  have
$$
\mu_{\alpha} ( R_t(p_0)) \geq c_0 \sqrt{t}, \ t \in [2 y_0, d_n].
$$
\end{lemma}

\noindent{\em Proof.} First assume $t \in [2y_0, e^{ \frac{y_0} {2} }].$
Notice that $B(p_0, t)$ meets the positive $x$-axis at points $\pm \sqrt{t^2 - y_0^2}$ which have absolute value exceeding  $\sqrt{3} t/2$ when $t \geq 2y_0$.  In other words, we have
$B(p_0, t) \supseteq [ -  \sqrt{3} t/ 2,  \sqrt{3} t/ 2 ] \times [0, y_0]$.  
We also have $D(p_0) \supseteq  [- e^{ \frac{y_0} {2} }, e^{ \frac{y_0} {2} }] \times [0, y_0]$, implying 
$R_t(p_0) \supseteq [ -  \sqrt{3} t/ 2,  \sqrt{3} t/ 2 ] \times [0, y_0]$. 
Consequently, we have
$$\
\mu_{\alpha}  ( R_t(p_0)) \geq  \sqrt{3} t \beta \int_0^{8}   e^{- \alpha y'} dy' \geq c_0 t \geq c_0 \sqrt{t}.
$$

Now  assume $t \in [ e^{ \frac{y_0} {2} }, d_n].$
Since $y_0$ exceeds $8$, we have $e^{y_0/2} \geq 2y_0$. As above, it follows that
 $$
 R_t(p_0) \supseteq \left[ - { \sqrt{3} e^{\frac{y_0} {2} }  \over 2},  { \sqrt{3} e^{\frac{y_0} {2} } \over 2} \right] \times [0, y_0].$$
Hence
$$
\mu_{\alpha} ( R_t(p_0)) \geq  \int_0^{y_0} \sqrt{3} e^{\frac{y_0} {2}} \beta e^{- \alpha y'} dy' \geq c_0 e^{\frac{y_0} {2}} > c_0 \sqrt{t/3},
$$
where the last inequality uses $t \leq d_n \leq  2s_n e^{y_0}$.
Hence, $\mu_{\alpha}  ( R_t(p_0)) \geq  c_0 \sqrt{t/3},$
as desired.   \qed

\vskip.3cm

Now we note that $R^\xi( p_0,  \tilde{\cal P}_{\alpha} \cap D ) \geq t$ iff $d_0 \geq t$ which happens if $(\tilde{\cal P}_{\alpha} \cap D) \cap D(p_0) = \{p_0\}$. 
Lemma \ref{lemm3}  shows that
for $y_0 \in [8, H ]$
$$
\P(  R^\xi( p_0,  \tilde{\cal P}_{\alpha} \cap D ) \geq t ) \leq  \exp(-\mu_{\alpha} ( R_t(p_0))) \leq    \exp(-c_0 \sqrt{t/3} ), \ \ t \in [2 y_0, d_n].
$$
For $t \in [0, 2 y_0]$ we have the trivial bound
$$
\P(  R^\xi( p_0,  \tilde{\cal P}_{\alpha} \cap D ) \geq t ) \leq  \exp \left( {  \alpha y_0  \over 2} \right)  \exp
\left( - { \alpha t \over 4}  \right).
$$
Put  $\phi(t):=  \min \{ \alpha t/4, c_0 \sqrt{t/3}\}$ for $t \in (0, \infty)$.  Summarizing the above we have shown for $y_0 \in [8, H]$ that
\be \label{stab-boundsa}
\P(  R^\xi( p_0,  \tilde{\cal P}_{\alpha} \cap D ) \geq t ) \leq   \exp \left( {  \alpha y_0  \over 2} \right) \exp(-  \phi(t)  ), \ \ t \in [0, \infty).
\ee
It remains to show that \eqref{stab-boundsa} holds for $y_0 \in [0, 8]$.
Recall  that the left-hand side of  \eqref{stab-boundsa} vanishes for $t$ in the range $t \in [d_n, \infty)$.
If $y_0 \in [0, 8]$ and  $t \in [0, d_n]$  then $R^\xi \leq e^8$. For $c$ large enough we thus have
\be \label{stab-bounds}
\P(  R^\xi( p_0, \tilde{\cal P}_{\alpha} \cap D ) \geq t ) \leq  c \exp \left( {  \alpha y_0  \over 2} \right) \exp(-  \phi(t)  ), \ \ t \in [0, \infty).
\ee
Thus we have shown \eqref{stab-bound} as desired.

\vskip.3cm

Recall that $2I_n = \pi e^{R/2}$.  Since $\tilde{\cal P}_{\alpha} \cap D$ is stationary with respect to the spatial $x$-coordinate it follows that for all $(x(p), y(p)) \in [- I_n, I_n] \times [0, H]$, we have
\be \label{general-stab-bounds}
\P(  R^\xi( ( x(p), y(p)),   \tilde{\cal P}_{\alpha} \cap D ) \geq t ) \leq  c    \exp \left( {  \alpha y(p)  \over 2} \right) \exp(-  \phi(t)  ), \ \ t \in [0, \infty).
\ee

Given the  bound \eqref{general-stab-bounds}, we now find asymptotics for $n^{-1} \Var [\tilde{S}^{ext}_H( \tilde{\cal P}_{\alpha} \cap D)]$.
In the remainder of this section we continue to abbreviate $\txi^{ext}$ by $\xi$.
In the next lemma, we will bound the covariance of $\xi$ with respect to
$(x_1, y_1), (x_2,y_2)\in \mathbb{R} \times [0, R]$, namely
\begin{align*}
& c^\xi ( (x_1, y_1), (x_2, y_2) ) \\
& = \Ex{ \xi( (x_1, y_1), (\tilde{\cal P}_{\alpha} \cap D) \cup \{(x_1,y_1),(x_2, y_2)\} ) \cdot \xi( (x_2, y_2), (\tilde{\cal P}_{\alpha} \cap D) \cup \{(x_1, y_1),(x_2,y_2\} )} \\
& \ \ \ \ - \Ex {\xi( (x_1, y_1), (\tilde{\cal P}_{\alpha} \cap D) \cup \{(x_1,y_1)\})}\cdot \Ex{\xi( (x_2, y_2), (\tilde{\cal P}_{\alpha} \cap D)  \cup \{(x_2, y_2)\} )}.
\end{align*}

\begin{lemma} There is a constant $c \in (0, \infty)$ such that for all $(x_1, y_1), (x_2, y_2) \in (- I_n, I_n] \times [0, H]$, we have
\be \label{correlationbound}
c^\xi( (x_1, y_1), (x_2, y_2) ) \leq c \left( \exp \left( {\alpha y_1 \over 2} \right) + \exp \left( {\alpha y_2 \over 2 } \right) \right) \exp \left( - \phi \left( {|x_1 - x_2| \over 3} \right)\right).
\ee
\end{lemma}
\noindent{\em Proof.}  Write $M := \max \{  R^\xi( ( x_1, y_1),  \tilde{\cal P}_{\alpha} \cap D ),  R^\xi( ( x_2, y_2),  \tilde{\cal P}_{\alpha} \cap D ) \}$.  Put $r:= |x_1 - x_2|/3$ and define $E:= \{M \leq r \}$.
 Since $\xi$ is bounded by $1$, we note that
$$
\mathbb{E}\left[ \xi( (x_1, y_1), (\tilde{\cal P}_{\alpha} \cap D)  \cup \{(x_1,y_1),(x_2, y_2)\})\cdot
\xi( (x_2, y_2),  (\tilde{\cal P}_{\alpha} \cap D)  \cup \{(x_1,y_1),(x_2, y_2)\} ) \right]
$$
differs from
$$
 \mathbb{E} \left[ \xi( (x_1, y_1),  (\tilde{\cal P}_{\alpha} \cap D) \cup \{(x_1,y_1),(x_2, y_2)\} )\cdot  \xi( (x_2, y_2),  (\tilde{\cal P}_{\alpha} \cap D) \cup \{(x_1,y_1),(x_2, y_2)\} )
\times  {\bf 1}(E)  \right]
$$
by at most $\P(E)$.

Notice that
\begin{align*}
& \mathbb{E} [ \xi( (x_1, y_1),  (\tilde{\cal P}_{\alpha} \cap D) \cup \{(x_1,y_1),(x_2, y_2)\} ) \\
& \ \ \ \ \ \ \ \cdot  \xi( (x_2, y_2),  (\tilde{\cal P}_{\alpha} \cap D) \cup \{(x_1,y_1),(x_2, y_2)\} )
\times  {\bf 1}(E)  ] \\
& = \mathbb{E} [ \xi( (x_1, y_1),  ( (\tilde{\cal P}_{\alpha} \cap D) \cup \{(x_1,y_1) \}) \cap B((x_1,y_1), r )) \\
& \ \ \ \ \ \ \  \cdot  \xi( (x_2, y_2),  ((\tilde{\cal P}_{\alpha} \cap D) \cup \{(x_2, y_2)\})  \cap B((x_2,y_2), r )) )
\times  ( 1- {\bf 1}(E^c)) ].
\end{align*}
We consequently obtain
\begin{align*}
& | \mathbb{E} [ \xi( (x_1, y_1), (\tilde{\cal P}_{\alpha} \cap D)  \cup \{(x_1,y_1),(x_2, y_2)\})  \cdot \xi( (x_2, y_2),  (\tilde{\cal P}_{\alpha} \cap D)  \cup \{(x_1,y_1),(x_2, y_2)\} ) ]  \\
& \ \ \ \ \ \ \ - \mathbb{E} [ \xi( (x_1, y_1),  ( (\tilde{\cal P}_{\alpha} \cap D) \cup \{(x_1,y_1) \}) \cap B((x_1,y_1), r )) \\
& \ \ \ \ \  \cdot  \xi( (x_2, y_2),  ((\tilde{\cal P}_{\alpha} \cap D) \cup \{(x_2, y_2)\})  \cap B((x_2,y_2), r )) ) ]|  \\
& \leq 2 \mathbb{P}(E^c).
\end{align*}
By independence we have
\begin{align*}
& | \mathbb{E} [ \xi( (x_1, y_1), (\tilde{\cal P}_{\alpha} \cap D)  \cup \{(x_1,y_1),(x_2, y_2)\})  \cdot \xi( (x_2, y_2),  (\tilde{\cal P}_{\alpha} \cap D)  \cup \{(x_1,y_1),(x_2, y_2)\} ) ]  \\
& \ \ \ \ \ \ \ - \mathbb{E} [ \xi( (x_1, y_1),  ( (\tilde{\cal P}_{\alpha} \cap D) \cup \{(x_1,y_1) \}) \cap B((x_1,y_1), r ))] \\
& \ \ \ \ \  \cdot \mathbb{E} [ \xi( (x_2, y_2),  ((\tilde{\cal P}_{\alpha} \cap D) \cup \{(x_2, y_2)\})  \cap B((x_2,y_2), r )) ) ]|  \\
& \leq 2 \mathbb{P}(E^c).
\end{align*}
Likewise
$$
\mathbb{E} [\xi( (x_1, y_1),  (\tilde{\cal P}_{\alpha} \cap D) \cup \{(x_1,y_1)\} ) ] \mathbb{E} [\xi( (x_2, y_2), ( \tilde{\cal P}_{\alpha} \cap D) \cup \{(x_2, y_2)\})]
$$
differs from
\begin{align} \label{commonterm2}
& \mathbb{E} [ \xi( (x_1, y_1),  ( (\tilde{\cal P}_{\alpha} \cap D) \cup \{(x_1,y_1) \}) \cap B((x_1,y_1), r ))]  \nonumber \\
& \ \ \ \ \ \cdot \mathbb{E} [ \xi( (x_2, y_2),  ((\tilde{\cal P}_{\alpha} \cap D) \cup \{(x_2, y_2)\})  \cap B((x_2,y_2), r )) )]
\end{align}
by at most $2 \mathbb{P}(E^c)$.  We conclude that
$$
c^\xi( (x_1, y_1), (x_2, y_2) ) \leq 4  \P \left(M \geq  \frac{ |x_1 - x_2| } {3} \right).
$$
The bound \eqref{stab-bound} completes the proof.  \qed

\vskip.5cm

Recall that $\tilde{\cal P}_{\alpha}$ is the Poisson point process  on $\mathbb{R} \times [0, \infty)$ with intensity measure $\mu_\alpha$ as at \eqref{Pnavdef}.
Next, define $c^\xi( (x_1, y_1), (x_2, y_2);   \tilde{\cal P}_{\alpha})$ analogously as in the definition of $c^\xi( (x_1, y_1), (x_2, y_2);
 \tilde{\cal P}_{\alpha} \cap D )$.  Note that $ (\tilde{\cal P}_{\alpha} \cap D)  \tod  \tilde{\cal P}_{\alpha}$ as $n \to \infty$.
The next lemma follows from stabilization methods; see for example \cite{BY}, \cite{Pe}.

\begin{lemma} We have
\be \label{correlation-convergence}
\lim_{n \to \infty} c^\xi( (x_1, y_1), (x_2, y_2); \tilde{\cal P}_{\alpha} \cap D ) = c^\xi( (x_1, y_1), (x_2, y_2);   \tilde{\cal P}_{\alpha}).
\ee
\end{lemma}

Now we may finally prove the asserted variance asymptotics at \eqref{Varext}.  Put
\begin{equation}  \label{sigmaformula}
\begin{split}
\sigma^2 &=  2  \alpha \int_0^{\infty} \mathbb{E} [ \xi( (0, y_1), \tilde{\cal P}_{\alpha})]   e^{- \alpha y_1} dy_1\\
& \ \ \ \ \ \ + \  2\alpha \beta  \int_0^{\infty} \int_{- \infty}^{ \infty } \int_0^{\infty} c^\xi( (0, y_1), (z, y_2);  \tilde{\cal P}_{\alpha}   ) e^{- \alpha y_2}  dy_2 dz   e^{- \alpha y_1}  dy_1.
\end{split}
\end{equation}
By \eqref{eq:var_ext_equiv}, it is enough to show that
\be
\lim_{n \to \infty} \frac{ \Var [\tilde{S}^{ext}_H( \tilde{\cal P}_{\alpha} \cap D )] } {n} = \sigma^2.
\ee

\vskip.3cm

\noindent{ \em Proof of \eqref{Varext}. }  We have by~\eqref{eq:variance}
\begin{align}
& \frac{1}{n} \Var [\tilde{S}^{ext}_H( \tilde{\cal P}_{\alpha} \cap D )] \nonumber \\
&  =  \frac{\beta} {n} \int_{ - I_n  }^{ I_n } \int_0^{H}  \mathbb{E} [\xi( (x_1, y_1), (\tilde{\cal P}_{\alpha} \cap D)\cup
 \{(x_1,y_1)\} )^2] e^{- \alpha y_1} dy_1 dx_1 \nonumber \\
&  \hskip.5cm + \frac{ \beta^2}{n} \int_{- I_n}^{I_n} \int_0^{H} \int_{- I_n}^{I_n} \int_0^{H} c^\xi( (x_1, y_1), (x_2, y_2) )
e^{- \alpha y_2}  dy_2 dx_2   e^{- \alpha y_1}  dy_1 dx_1.  \label{Sliv}
\end{align}

The first integral in \eqref{Sliv} reduces to $\beta \int_0^{H}  \mathbb{E} [(\xi( (0, y_1), \tilde{\cal P}_{\alpha} \cap D )^2] e^{- \alpha y_1} dy_1$ by translation invariance of $\xi$ in the spatial $x$ coordinate. The stabilization of $\xi$ shows for all $y_1$  that
$$
 \lim_{n \to \infty} \mathbb{E}[ \xi( (0, y_1), \tilde{\cal P}_{\alpha} \cap D )]  = \mathbb{E} [ \xi( (0, y_1), \tilde{\cal P}_{\alpha})].
 $$
By the dominated convergence theorem and using $2I_n \beta/n = 2 \alpha $ and $\xi^2 = \xi$,  we obtain
$$
\lim_{n \to \infty} \frac{\beta}{n} \int_{- I_n}^{I_n} \int_0^{H} \mathbb{E} [ \xi( (x_1, y_1), \tilde{\cal P}_{\alpha} \cap D  )^2] e^{- \alpha y_1} dy_1 dx_1
$$
\be \label{sigma1}
=  2  \alpha \int_0^{\infty} \mathbb{E} [ \xi( (0, y_1), \tilde{\cal P}_{\alpha})]  e^{- \alpha y_1} dy_1.
\ee
Now we turn to the second integral in \eqref{Sliv}.  By translation invariance in the spatial coordinate we have
$$
 \frac{ \beta^2}{n} \int_{- I_n}^{I_n} \int_0^{H} \int_{- I_n}^{I_n}\int_0^{H} c^\xi( (x_1, y_1), (x_2, y_2); \tilde{\cal P}_{\alpha} \cap D ) e^{- \alpha y_2}  dy_2 dx_2   e^{- \alpha y_1}  dy_1 dx_1
$$
$$
=  \frac{ \beta^2}{n} \int_{- I_n}^{I_n} \int_0^{H} \int_{- I_n}^{I_n} \int_0^{H} c^\xi( (0, y_1), (x_2 - x_1, y_2); \tilde{\cal P}_{\alpha} \cap D ) e^{- \alpha y_2}  dy_2 dx_2   e^{- \alpha y_1}  dy_1 dx_1
$$
$$
=  \frac{ \beta^2}{n} \int_{- I_n}^{I_n} \int_0^{H} \int_{- I_n - x_1}^{I_n - x_1} \int_0^{H} c^\xi( (0, y_1), (z, y_2); \tilde{\cal P}_{\alpha} \cap D )
e^{- \alpha y_2}  dy_2 dz   e^{- \alpha y_1}  dy_1 dx_1.
$$
Let $v:=  \beta x_1/ \alpha n$, $dv:= \beta/ (\alpha n)  dx_1$.  Then
since $I_n =  \alpha n/\beta$ the above becomes
$$
=   \alpha \beta \int_{ - 1 }^{ 1 } \int_0^{H} \int_{ -I_n(1 - v) }^{ I_n(1 - v) }\int_0^{H} c^\xi( (0, y_1), (z, y_2); \tilde{\cal P}_{\alpha} \cap D ) e^{- \alpha y_2}  dy_2 dz   e^{- \alpha y_1}
dy_1 dv.
$$
For every $v \in [-1, 1]$ we have by \eqref{correlationbound} that $c^\xi( (0, y_1), (z, y_2); \tilde{\cal P}_{\alpha} \cap D ) e^{- \alpha y_2}e^{- \alpha y_1}$ is dominated by an integrable function of $(y_1, y_2, z)$.  It follows by the dominated convergence theorem that for every $v \in [-1, 1]$ we have
$$
\lim_{n \to \infty} \int_0^{H} \int_{ -I_n(1 - v) }^{ I_n(1 - v) } \int_0^{H} c^\xi( (0, y_1), (z, y_2); \tilde{\cal P}_{\alpha} \cap D
 ) e^{- \alpha y_2}  dy_2 dz   e^{- \alpha y_1}  dy_1
$$
\be 
\stackrel{\eqref{correlation-convergence}}{=}  \int_0^{\infty} \int_{- \infty}^{ \infty } \int_0^{\infty} c^\xi( (0, y_1), (z, y_2);  \tilde{\cal P}_{\alpha}   ) e^{- \alpha y_2}  dy_2 dz   e^{- \alpha y_1}  dy_1. \nonumber
\ee
The second integral in \eqref{Sliv} thus converges to
\be \label{sigma3}
2\alpha \beta    \int_0^{\infty} \int_{- \infty}^{ \infty } \int_0^{\infty} c^\xi( (0, y_1), (z, y_2);  \tilde{\cal P}_{\alpha}   ) e^{- \alpha y_2}  dy_2 dz   e^{- \alpha y_1}  dy_1.
\ee

Notice that  $\sigma^2$ is the sum of  \eqref{sigma1} and \eqref{sigma3}.
This completes the proof of  Theorem \ref{extVar}.   \qed

\section{Proof of Theorem \ref{mainCLT} }

To prove \eqref{CLTiso} and \eqref{CLText}, we first assert that it suffices to prove central limit theorems for the random variables
$\tilde{S}^{ext}_H(\tilde{\cal P}_{\alpha} \cap D)$ and $\tilde{S}^{iso}_H(\tilde{\cal P}_{\alpha} \cap D)$, defined at  \eqref{eq:sum_score_R^2} and \eqref{eq:sum_score_R^2_extr}, respectively.
We prove this assertion for $\tilde{S}_{n}:=  \tilde{S}^{iso}_H(\tilde{\cal P}_{\alpha} \cap D)$ as the proof for $\tilde{S}^{ext}_H(\tilde{\cal P}_{\alpha} \cap D)$ is identical.

Set  $S_n$ to be  $S^{iso}({\cal P}_{\alpha, n} ) :=  \tilde{S}^{iso}( \tilde{\cal P}_{\alpha, n})$. Recall that $S_n$ is determined by the Poisson process
$\tilde{\cal P}_{\alpha, n}$ on $\Rcal$ defined at \eqref{Pnav} whereas $\tilde{S}_{n}$ is determined
by $\tilde{\cal P}_{\alpha} \cap D$ defined at \eqref{eq:sum_score_R^2}. By Lemma \ref{lem:dist},
the intensities of these two processes differ by $\epsilon_n =  O(n^{ - 2 \alpha})$.
We can couple these two processes using a sprinkling argument.
Let $\hat{\cal P}$ be the Poisson process on $D$ with intensity equal to
$\lambda(x,y) :=\min \{\beta e^{-\alpha y}, \beta e^{-\alpha y} + \epsilon_n\}$  at $(x,y) \in D$ - in other words the minimum of the intensities of $\tilde{\cal P}_{\alpha} \cap D$ and $\tilde{\cal P}_{\alpha, n}$.
Now, we define two other independent processes on $D$:
$\hat{\cal P}_1$ of intensity $\beta e^{-\alpha y} - \lambda (x,y)$ at $(x,y)$ and
$\hat{\cal P}_2$ of intensity $\beta e^{-\alpha y} + \epsilon_n - \lambda (x,y)$ at $(x,y)$.
The union of $\hat{\cal P}$ and $\hat{\cal P}_1$ is distributed as $\tilde{\cal P}_{\alpha} \cap D$, whereas the
union of $\hat{\cal P}$ and $\hat{\cal P}_2$ is distributed as $\tilde{\cal P}_{\alpha, n}$.
We will use the symbols $\tilde{\cal P}_{\alpha} \cap D$ and  $\tilde{\cal P}_{\alpha, n}$ to denote the copies of these
processes in the coupling space.
For each $n$, we may define the $nth$ coupling space to be the product of the spaces on which $\hat{\cal P}$, $\hat{\cal P}_1$, and $\hat{\cal P}_2$ are all defined.
Let $\widehat{\mathbb{P}}_n$ denote the product probability measure on the coupling space.

Thus, for any $\alpha \in (1/2, \infty)$
$$
\widehat{\mathbb{P}}_n\left( \tilde{\cal P}_{\alpha, n} \not =  \tilde{\cal P}_{\alpha} \cap D \right) = \widehat{\mathbb{P}}_n\left( \hat{\cal P}_1 \cup \hat{\cal P}_2 \not = \emptyset\right)= O(1) \cdot R \cdot n^{1-2\alpha} =o(1).
$$
This implies that on the coupling space we have
$\tilde{\cal P}_{\alpha, n} = \tilde{\cal P}_{\alpha} \cap D$ with probability $\to 1$ as $n\to \infty$.
Also, the coupling will allow us to assume that $S_n$ and $\tilde{S}_{n} $ are defined on the same probability space.

Furthermore, by Lemmas~\ref{lem:variances_tilde_checked} and ~\ref{lem:exp_var_equiv}
we have
$$
|\mathbb{E} {S_n} -  \mathbb{E} {\tilde{S}_{n} }| = o(1) \ \mbox{and} \ \left| \Var S_n -  \Var \tilde{S}_{n}  \right| = o(n).
$$
In particular, the former implies that $\mathbb{P}(S_n \neq \tilde{S}_{n} ) = o(1)$.  Henceforth, if $X_n, n \geq 1$, is a sequence of random variables with $X_n$ defined on the $nth$ coupling space, then by $X_n \toP 0$ we mean that for all $\epsilon > 0$ we have $\widehat{\mathbb{P}}_n ( |X_n| \geq \epsilon ) \to 0$ as $n \to \infty$.

Thus we have $|S_n - \mathbb{E} S_n  - (\tilde{S}_{n}  - \mathbb{E} \tilde{S}_{n} ) | \toP 0$ as $n \to \infty$, whence
$$
\left| \frac{ S_n - \mathbb{E} S_n } {\sqrt{\Var S_n}} -  \frac{ \tilde{S}_{n}   - \mathbb{E} \tilde{S}_{n}  }  { \sqrt{\Var S_n} } \right| \toP 0
$$
as well. If $X_n, n \geq 1$, and $Y_n, n \geq 1,$ are sequences of random variables with $|X_n - Y_n| \toP 0$, if $\sup_n \mathbb{E}  |Y_n| < \infty$, and if $\alpha_n, n \geq 1$, is a sequence of scalars with $\lim_{n \to \infty} \alpha_n = 1$, then $|X_n - \alpha_n Y_n | \toP 0$.
Since $\lim_{n \to \infty}  {\sqrt{\Var S_n}}/ \sqrt{ \Var \tilde{S}_n } = 1$ it follows that as $n \to \infty$ we have
$$
\left| \frac{ S_n - \mathbb{E} S_n } {\sqrt{\Var S_n}} -  \frac{ \tilde{S}_{n}  - \mathbb{E} \tilde{S}_{n}  }  {\sqrt{ \Var \tilde{S}_{n}  }} \right| \toP 0.
$$
Thus the asymptotic normality for $ (\tilde{S}_{n}  - \mathbb{E} \tilde{S}_{n})/  \sqrt{ \Var \tilde{S}_{n}  }$ implies the
asymptotic normality of  $(S_n - \mathbb{E}  S_n) / {\sqrt{\Var S_n}}$, i.e., we have
as $n\to \infty$
$$
\Pro {\frac{ S_n - \mathbb{E} S_n } {\sqrt{\Var S_n}} \leq x } \to \Phi (x) := \frac{1}{\sqrt{2\pi}} \int_{-\infty}^x e^{-u^2/2}du.
$$

In the following sub-sections, we will show that $\tilde{S}^{iso}_H(\tilde{\cal P}_{\alpha} \cap D)$  and $\tilde{S}^{ext}_H(\tilde{\cal P}_{\alpha} \cap D)$  satisfy a central limit theorem, the former for all $\alpha \in (1, \infty)$ and
the latter for all $\alpha \in (1/2, \infty)$. These imply  ~\eqref{CLTiso}  and ~\eqref{CLText}.
On the other hand, in the final sub-section, we show that
$\tilde{S}^{iso}_H(\tilde{\cal P}_{\alpha} \cap D)$ does not satisfy a central limit theorem for $\alpha \in (1/2, 1)$.
The above argument implies that $S^{iso}(  {\cal P}_{\alpha, n})$ also does not satisfy a central limit theorem in the
same range of $\alpha$.

\subsection{The central limit theorem for $\tilde{S}^{ext}_H(\tilde{\cal P}_{\alpha} \cap D)$}

Consider the ball $B(p_2)$ centered at a point $p_2 \in \R{[0,H]}$.  We compute the maximum $x$-distance between $p_2$ and a
 generic point $p$  in the intersection of $B(p_2) \cap \R{[0,H]}$.  This tells us the maximum $x$-distance $x_{\max}:= x_{\max}^{ext}$ of the set given by the intersection of $B(p_2)$ with $\R{[0,H]}$.
Since both $p_2$ and $p$ have heights at most $H=4 \log R$, the inclusion
$B(p_2) \subseteq B^+(p_2)$ at ~\eqref{eq:inclusion_mapped} implies that
$$
x_{\max} = O(1) \cdot  R^4.
$$

We define a dependency graph $G_n:= G_n^{ext}:= ({\cal V}_n, {\cal E}_n)$ as follows.
Firstly, we partition the interval $(-I_n,I_n]$ into $\Theta ( n/R )$ consecutive intervals of equal length,
which we enumerate $J_1,\ldots, J_{[2 I_n/R]}$. For each $i=1,\ldots, [2I_n/R]$,
we set $C_i := J_i \times [0,R]$. The collection of axis-parallel rectangles $\{ C_i\}_{i=1,\ldots, [2I_n/R]}$ partitions $\Rcal$.
The vertex set ${\cal V}_n$ consists of the rectangles $C_1,...,C_{ [2I_n/R]}$.  We put an edge $(C_i, C_j)$   between any two rectangles  whenever  $C_i$ and $C_j$ are separated by a rectangle having $x$-distance at most
$2 x_{\max}$.  Let ${\cal E}_n$ be the collection of all such edges. Put for all $i = 1,..., [2I_n/R]$
$$
Z_i := Z_{C_i}^{ext} := \sum_{ p \in  \tPH \cap C_i} \txi^{ext} (p, \tP).
$$
By the definition of $x_{\max}$, if $\mathcal{C}_1$ and $\mathcal{C}_2$ are disjoint collections of rectangles in ${\cal V}_n$ such that no edge in ${\cal E}_n$ has one endpoint in
$\mathcal{C}_1$ and the other endpoint in $\mathcal{C}_2$, then the random variables
$\{ Z_{C_i}, C_i \in \mathcal{C}_1 \}$ and $\{ Z_{C_i}, C_i \in \mathcal{C}_2 \}$ are independent.
Note that a rectangle $C$ having $x$-side equal to $x_{\max}$ will have non-empty intersection with at most
\begin{equation} \label{conebound}
\frac{ x_{\max}} {2I_n/ [2I_n/R] } =
 O(1) \cdot R^3
\end{equation}
rectangles from the collection $C_1,...,C_{ [2I_n/R] }$. %

Thus $G_n^{ext}:= ({\cal V}_n, {\cal E}_n)$ is a dependency graph for $Z_i, i = 1,..., [2I_n/R].$
Now note that
$$
\tilde{S}^{ext}_H(\tilde{\cal P}_{\alpha} \cap D) := \sum_{i = 1}^{[2I_n/R]} Z_i^{ext}.
$$

Furthermore,
\begin{equation} \label{eq:sector}
\mathbb{E} {|C_1 \cap (\tilde{\cal P}_{\alpha} \cap D) |} = \beta \cdot 2I_n \cdot [ \frac{2I_n} {R} ]^{-1} \int_0^{R} e^{-\alpha y}dy
=O(1) \cdot  R.
 \end{equation}
Standard tail estimates for Poisson random variables give $\Pro{|C_1 \cap (\tilde{\cal P}_{\alpha} \cap D) |>R^2} \leq e^{-R^2}$, for $n$ sufficiently large.
So
\begin{align}  \label{CLT3}    \lim_{n \to \infty} \P(  {\rm{card}}( (\tilde{\cal P}_{\alpha} \cap D) \cap C_i)  \leq  R^2,\ 1 \leq i \leq [  \frac{2I_n} {R}] )
 & = 1 - O\left(\frac{n}{R} \cdot e^{-R^2/2}\right) \nonumber \\
 & =1 -o(n^{-15}).
\end{align}

Define
$$
A_n := \{  {\rm{card}}((\tilde{\cal P}_{\alpha} \cap D) \cap C_i)  \leq R^2 \ \text{for \ all} \ 1 \leq i \leq [ \frac{2I_n} {R} ]\}.
$$

The maximal degree $D_n$ of the dependency graph $G_n^{ext}$ satisfies $D_n = O(R^3).$  We also set $B_n := \max_{1 \leq i \leq [2I_n /R]} Z_{C_i}^{ext} \leq 2R$, $V_n := {\rm{card}}({\cal V}_n) = [2I_n/R]$.
Set $\sigma_n^2 := \Var [\tilde{S}^{ext}_H(\tilde{\cal P}_{\alpha} \cap D) |A_n]$.
H\"older's inequality gives
\begin{equation*}
\begin{split}
\Var [\tilde{S}^{ext}_H(\tilde{\cal P}_{\alpha} \cap D) {\bf 1}(A_n^c)]& \leq \mathbb{E} [( \tilde{S}^{ext}_H((\tilde{\cal P}_{\alpha} \cap D)) )^2 {\bf 1}(A_n^c)]
 \leq ( \mathbb{E} [ |\tilde{S}^{ext}_H(\tilde{\cal P}_{\alpha} \cap D)|^3] )^{2/3} \P (A_n^c)^{1/3} \\
& \leq n^2 \cdot n^{-5} = n^{-3}
\end{split}
\end{equation*}
and
$$
\mathbb{E} [ \tilde{S}^{ext}_H(\tilde{\cal P}_{\alpha} \cap D) {\bf 1}(A_n^c)] \leq  \mathbb{E} [ |\tilde{S}^{ext}_H(\tilde{\cal P}_{\alpha} \cap D
)|^2]^{1/2} \P (A_n^c)^{1/2} \leq n \cdot
n^{-15/2} = o(1).
$$
We thus conclude that
\begin{equation} \label{eq:var_exp_equiv}
 |\Var [\tilde{S}^{ext}_H(\tilde{\cal P}_{\alpha} \cap D)] - \Var [\tilde{S}^{ext}_H(\tilde{\cal P}_{\alpha} \cap D) | A_n] | = o(1)
\end{equation}
and
\begin{equation} \label{eq:varequiv}
\left|\mathbb{E} [{\tilde{S}^{ext}_H(\tilde{\cal P}_{\alpha} \cap D)}] - \mathbb{E} [\tilde{S}^{ext}_H((\tilde{\cal P}_{\alpha} \cap D)) | A_n] \right| = o(1).
\end{equation}

We have shown that $\Var [\tilde{S}^{ext}_H(\tilde{\cal P}_{\alpha} \cap D)] = \Theta (n)$ and thus also $\Var [\tilde{S}^{ext}_H(\tilde{\cal P}_{\alpha} \cap D) | A_n] = \sigma_n^2 = \Theta (n)$.
The Baldi-Rinott central limit theorem for dependency graphs \cite{BR}  gives
\begin{align} \label{eq:Baldi-Rinott}
& \sup_{x\in \mathbb{R}} \left| \P \left( \frac{ \tilde{S}^{ext}_H(\tilde{\cal P}_{\alpha} \cap D) - \mathbb{E} [ \tilde{S}^{ext}_H(\tilde{\cal P}_{\alpha} \cap D) | A_n ]  } { \sqrt { \Var [ \tilde{S}^{ext}_H(\tilde{\cal P}_{\alpha} \cap D) |A_n]}
 }  \leq x \ | A_n \right)
- \Phi(x) \right|   \\
& \leq 32(1 + \sqrt{6} )    \left(  \frac{D_n^2 B_n^3 V_n } { \sigma_n^3} \right)^{1/2}. \nonumber
\end{align}
Since $\sigma_n = \Theta (n^{1/2})$, we have $D_n^2 B_n^3 V_n/ \sigma_n^3 = o(1)$.
This shows a central limit theorem for $\tilde{S}^{ext}_H(\tilde{\cal P}_{\alpha} \cap D)$ conditional on $A_n$.

To deduce a central limit theorem for $\tilde{S}^{ext}_H(\tilde{\cal P}_{\alpha} \cap D)$, we write
\begin{align*}
& \Pro{\frac{ \tilde{S}^{ext}_H(\tilde{\cal P}_{\alpha} \cap D) - \mathbb{E} [ \tilde{S}^{ext}(\tilde{\cal P}_{\alpha} \cap D) ] } { \sqrt {\Var [\tilde{S}^{ext}_H(\tilde{\cal P}_{\alpha} \cap D)] } } \leq x} \\
& = \Pro{\frac{ \tilde{S}^{ext}_H(\tilde{\cal P}_{\alpha} \cap D) - \Ex {\tilde{S}^{ext}_H(\tilde{\cal P}_{\alpha} \cap D)}
 } {\sqrt {\Var [ \tilde{S}^{ext}_H(\tilde{\cal P}_{\alpha} \cap D)] }} \leq x | A_n } + o(1)\\
& \stackrel{\eqref{eq:var_exp_equiv}}{=} \Pro{ \frac{ \tilde{S}^{ext}_H(\tilde{\cal P}_{\alpha} \cap D) - \Ex{\tilde{S}^{ext}_H(\tilde{\cal P}_{\alpha} \cap D) |A_n} } { \sqrt {\Var [\tilde{S}^{ext}_H(\tilde{\cal P}_{\alpha} \cap D)|A_n] } } \leq x + o(1) | A_n } + o(1).
\end{align*}
Since $ \tilde{S}^{ext}_H(\tilde{\cal P}_{\alpha} \cap D)$ conditional on $A_n$ satisfies a central limit theorem by \eqref{eq:Baldi-Rinott}, the
probability on the right-hand side converges to $\Phi (x)$.
\qed

\subsection{The central limit theorem for $\tilde{S}^{iso}_H(\tilde{\cal P}_{\alpha} \cap D)$: the regime $\alpha \in (1, \infty)$ } \label{sec:CLTiso}
The above approach turns out to be not strong enough for showing the asymptotic normality for the
number of isolated vertices.
For a certain range of $\alpha$ a dependency graph defined as above
has high maximum degree making the bounds  \eqref{eq:Baldi-Rinott} of little use.
We will instead prove a central limit theorem for $\tilde{S}^{iso}_H(\tilde{\cal P}_{\alpha} \cap D)$
using a  Poincar\'e-type inequality for Poisson functionals
due to Last, Peccati and Schulte~\cite{ar:LPS16}.

Let $\cP$ denote a Poisson point process on a space $\Space$ having intensity
measure $\lambda$. Let $F$ denote a functional on locally finite point sets in $\Space$.
Recall that for a point $p \in \Space$ we defined the first order linear operator
$\nabla_p F:= F(\cP \cup \{ p\}) - F(\cP)$. Here, we will also use  the \emph{second order operator}
$\nabla^2_{p_1,p_2} F: = F(\cP \cup \{ p_1,p_2\})  - F(\cP \cup \{ p_1\}) -F(\cP \cup \{ p_2\})   + F(\cP)$.
The functional $F$ belongs to the \emph{domain of} $\nabla$ if
$$ \mathbb{E} [ F(\cP)^2] < \infty \ \mbox{and} \ \mathbb{E} \int_\Space (\nabla_p F(\cP))^2 \lambda (dp) <\infty.$$
Theorem 1.1 of  ~\cite{ar:LPS16} uses these differential
operators to approximate the normalised version of $F$ by the standard
normal $N$.  For two real-valued random variables $X$ and $Y$, let $d_W (X,Y)$ denote the Wasserstein distance
between the measures on $\mathbb{R}$ induced by $X$ and $Y$.

\begin{theorem} \label{thm:normal_approx}
Let $F$ be a functional defined on locally finite collections of points in $\Space$.  Assume $F$ belongs to
the domain of $\nabla$ and satisfies $\mathbb{E}{F}=0$ and $\Var F =1$.
If $N$ is a standard normally distributed random variable, then
$$
d_{W} \left( F, N\right) \leq  \gamma_1+\gamma_2+\gamma_3,
$$
where
\begin{equation*}
\begin{split}
\gamma_1 &:= 4 \left[ \int_{ \Space^3}   \left( \Ex{(\nabla_{p_2}F)^2 (\nabla_{p_3}F)^2}\right)^{1/2}
\left( \Ex{(\nabla_{p_1,p_2}^2 F)^2 (\nabla_{p_1,p_3}^2F)^2}\right)^{1/2}
\lambda^3 (d (p_1,p_2,p_3))  \right]^{1/2},  \\
\gamma_2 &:= \left[  \int_{ \Space^3}
\Ex{(\nabla_{p_1,p_3}^2 F)^2 (\nabla_{p_2,p_3}^2F)^2}
\lambda^3 (d (p_1,p_2,p_3))  \right]^{1/2}, \\
\gamma_3 & := \int_{ \Space} \mathbb{E}{|\nabla_p F|^3} \lambda(dp).
\end{split}
\end{equation*}
\end{theorem}

We will apply Theorem \ref{thm:normal_approx} on the conditional space of
the event 
\begin{equation} \label{En}
E_n :=   (\tilde{\cal P}_{\alpha} \cap D) \cap \R{[ \frac{R}{2},R]}  = \emptyset. 
\end{equation}
A calculation similar to the one in~\eqref{eq:sector} shows that  for any $\alpha \in (1, \infty)
$ we have $\P(E_n) = 1 - O(n^{1-\alpha}).$

We shall apply Theorem~\ref{thm:normal_approx} setting $\lambda$ to be $\mu_\alpha$ and letting
$$\Space:=  {D}([0, \frac{R}{2}]),  \
\cP := \tilde{\cal P}_{\alpha} \cap D  \ \ \mbox{and} \
F := \frac{\tilde{S}^{iso}_H(\tilde{\cal P}_{\alpha} \cap D)  - \mathbb{E} [ \tilde{S}^{iso}_H(\tilde{\cal P}_{\alpha} \cap D) | E_n] }
{\sqrt{\Var [\tilde{S}^{iso}(\tPH) | E_n]}}.$$
These ensure that on $E_n$, one has $\mathbb{E}[F| E_n] =0$ and $\Var  [F | E_n] = 1$.
We will verify that $F$ is on the domain of $\nabla$ later on, using the estimate on
$\gamma_3$.  We will only check the second condition; the requirement that
 $\mathbb{E}{F( \tilde{\cal P}_{\alpha} \cap D )^2} < \infty$ follows from our bounds on
 $\Var [ \tilde{S}^{iso}_H(\tilde{\cal P}_{\alpha} \cap D)]$.

Set $\sigma_n^2 : = \Var [\tilde{S}^{iso}_H(\tilde{\cal P}_{\alpha} \cap D)]$ and $\sigma_n'^2 := \Var [\tilde{S}^{iso}_H(\tilde{\cal P}_{\alpha} \cap D) | E_n]$. 
The proof of the next lemma is postponed until Section~\ref{sec:variance equivalence}.
\begin{lemma} \label{lem:variance equivalence}
For any $\alpha \in (1, \infty)$, we have
$$
\lim_{n\to \infty} \frac{\sigma_n^2}{\sigma_n'^2}  = 1 \ \mbox{and} \
\lim_{n\to \infty}
\frac{ \mathbb{E} [ \tilde{S}^{iso}_H(\tilde{\cal P}_{\alpha} \cap D)]  - \mathbb{E} [ \tilde{S}^{iso}_H(\tilde{\cal P}_{\alpha} \cap D)  | E_n]}{\sigma_n'} =0.
$$
\end{lemma}

To apply Theorem~\ref{thm:normal_approx} we shall bound  $|\nabla_p F|$
by the number of points of $(\tilde{\cal P}_{\alpha} \cap D) \cup \{p \}$ which are inside the hyperbolic ball around $p$ having height at most $H$.
By the inclusion-exclusion principle, the second order operator $\nabla_{p_1,p_2} F$
is proportional to the number of isolated points of $\tilde{\cal P}_{\alpha} \cap D$ which are contained
in the intersection of the hyperbolic balls around $p_1$ and $p_2$ and having height
at most $H$.  Thus
\begin{equation} \label{eq:bounds_diff_op}
\begin{split}
|\nabla_p F|& \leq \frac{1}{\sigma_n'} \cdot \left( | (\tilde{\cal P}_{\alpha} \cap D) \cap B(p) \cap \R{[0,H]}| +1\right)
\ \mbox{and}
\\
|\nabla_{p_1,p_2} F| &\leq \frac{1}{\sigma_n'} \cdot | (\tilde{\cal P}_{\alpha} \cap D) \cap ( B(p_1) \cap B(p_2)) \cap \R{[0,H]}|.
\end{split}
\end{equation}
Given a Borel-measurable set $A \subset \R{[0,H]}$, we have that ${\text{card}}( \tilde{\cal P}_{\alpha} \cap D  \cap A)$
is a Poisson-distributed random variable with parameter equal to the intensity measure of $A$.
The next lemma, a consequence of Lemma \ref{muasym},
bounds these intensity measures for the sets $A$ appearing in \eqref{eq:bounds_diff_op}.

\begin{lemma} \label{eq:measures_bounds}
There exists a constant $\eta>0$ depending on $\alpha$ and $\nu$ such that for all
$p =(x(p),y(p)) \in \Rcal$ we have
$$\mu_{\alpha} (B(p) \cap \R{[0,H]}) \leq \eta \cdot e^{ \frac{y(p)} {2}}. $$
Hence,
$$ \mu_{\alpha}  ( ( B(p_1) \cap B(p_2)) \cap \R{[0,H]}) \leq \eta \cdot e^{ \frac{1}{2} (y(p_1) \wedge y(p_2) )}.$$
\end{lemma}
Set $\lambda (p) : = \eta \cdot e^{y(p)/2}$.
Thus, $|\nabla_p F|$ is stochastically dominated from above by a random variable
$X(p)+1$, where $X(p)$ distributed as $\mathrm{Po}(\lambda (p))$.
Analogously, $|\nabla_{p_1,p_2}|$ is stochastically dominated by a Poisson-distributed
random variable $X(p_1,p_2)$ with parameter $\lambda (p_1,p_2) := \eta \cdot
e^{(y(p_1) \wedge y(p_2) )/2}$.  We now bound $\gamma_3$, $\gamma_2$ and $\gamma_1$ in this order.

\begin{lemma} \label{gam3} If $\alpha \in (1, \infty)$, then $\gamma_3 = o(1)$.
\end{lemma}

\noindent{\em Proof.}  For $p \in \Space$,
$$\mathbb{E}{|\nabla_p F|^3} \leq \frac{1}{\sigma_n'^{3}}
  \mathbb{E} [(X(p)+1)^3]
= O (1) \cdot \frac{1}{\sigma_n'^{3}}\cdot e^{ \frac{3y(p)} {2}}.
$$
We deduce that
$$\int \mathbb{E}{|\nabla_p F|^3} \lambda (dp)=
O\left(1\right)\frac{1}{\sigma_n'^{3}}\cdot \int_{-I_n}^{I_n} \int_0^{R/2}e^{ \frac{3y} {2} - \alpha y}dy dx
= O\left(1\right)\frac{n}{\sigma_n'^{3}}\cdot \int_0^{R/2}e^{ \frac{3y} {2} - \alpha y}dy.
 $$
Recall that $n = \nu e^{R/2}$ and $\sigma_n' = \Theta(n^{1/2})$ by the first part of Lemma~\ref{lem:variance equivalence} and \eqref{eq:variances_hat_Siso}.
Therefore if $\alpha \in (1, 3/2]$ then
$$
\gamma_3 = O\left(R\right) \frac{n}{\sigma_n'^{3}}\cdot n^{\frac{3} {2} -\alpha} = O(R)
\cdot n^{1+ \frac{3} {2}  - \alpha - \frac{3} {2}} = O(R) \cdot n^{1-\alpha} =  o(1).
$$
If $\alpha \in (3/2, \infty)$, then $\gamma_3 = o(1).$ \qed

\vskip.3cm

Let us  point out that the bound on $\int \mathbb{E}{|\nabla_p F|^3} \lambda(dp)$ is also a bound
on $\int \mathbb{E}{|\nabla_p F|^2} \lambda (dp)$ and thus $F$ is in the domain of $\nabla$.

\begin{lemma} \label{gam2} If $\alpha \in (1, \infty)$, then $\gamma_2 = o(1)$.
\end{lemma}

\noindent{\em Proof.} The second inequality in ~\eqref{eq:bounds_diff_op} implies
$$
|\nabla_{p_1,p_2} F| \leq \frac{1}{\sigma_n'} \cdot
\min \{ | \tilde{\cal P}_{\alpha} \cap D \cap B(p_1) \cap \R{[0,R/2]}|,
| \tilde{\cal P}_{\alpha} \cap D  \cap B(p_2) \cap \R{[0,R/2]}| \}.
 $$
Now, we claim that there exists a constant $\gamma> 0$ such that if
\be \label{LBgam}
|x(p_1) - x(p_2) |_{\Phi} > \gamma e^{ \frac{H} {2} } \left(e^{ \frac{y(p_1)} {2}} + e^{ \frac{y(p_2)} {2}} \right).
\ee
then $B(p_1) \cap B(p_2) \cap \R{[0,H]} = \emptyset$.
Indeed, for any $p\in \R{[0,H]}$ we have $B(p) \cap \R{[0,H]} \subseteq B_{H}^+(p)$
where $B_{H}^+(p)$ is defined at \eqref{eq:ball_plus}.
Now, Lemma~\ref{4.1} implies that there exist some constant $\gamma>0$ such that
$B_H^+(p_1) \cap B_H^+(p_2) = \emptyset$ if \eqref{LBgam} holds.
This implies that when we integrate with respect to $x(p_2)$ and $x(p_3)$, relative distances
with respect to $p_1$ greater than this quantity have no contribution to the integral defining
$\gamma_2$. In other words,
 $p_2$ (and $p_3$, respectively) contributes to this integral only if
$$ |x(p_1) - x(p_2) |_{\Phi}\leq \gamma e^{ \frac{H} {2} } \left(e^{ \frac{y(p_1)} {2}} + e^{ \frac{y(p_2)} {2}} \right) \leq
2 \gamma e^{ \frac{H} {2} } e^{  \frac{y(p_1) \vee y(p_2)}{2}}.$$

Therefore, when we integrate over the choices of $p_1,p_2,p_3$ against the intensity measure
$d\mu^3 (p_1,p_2,p_3)$ (letting $x_i=x(p_i)$ and $y_i=y(p_i)$), we will get (using that $e^{y_1/2} + e^{y_2/2} \leq 2 e^{(y_1\vee y_2)/2}$)
\begin{equation*}
\begin{split}
\gamma_2^2 &\leq
4\gamma^2 \cdot \frac{1}{ \sigma_n'^4}\int_{-I_n}^{I_n} \int_{-I_n}^{I_n} \int_{-I_n}^{I_n} \int_{[0,R/2]^3}
\Ex{(X(p_1,p_2) X(p_1,p_3))^2}   \\
& \hspace{2cm} \times
{\mathbf 1} (|x_2 - x_1|_{\Phi}\leq 2 \gamma e^{H/2} e^{(y_1 \vee y_2 )/2} )
{\mathbf 1} (|x_3 - x_1|_{\Phi}\leq 2 \gamma e^{H/2} e^{(y_1 \vee y_3 )/2} )  \\
&\hspace{4cm} \times  e^{-\alpha y_1} e^{-\alpha y_2} e^{-\alpha y_3} dy_1 dy_2  dy_3 dx_3 dx_2 dx_1.
\end{split}
\end{equation*}
By the Cauchy-Schwarz inequality we have
\begin{equation} \label{eq:joint_X}
\begin{split}
 \Ex{(X(p_1,p_2) X(p_1,p_3))^2}  & \leq    \Ex{(X(p_1,p_2)^4}^{1/2} \Ex {X(p_1,p_3))^4}^{1/2} \\
 &= O(1)  \cdot  e^{y_1\wedge y_2} \cdot e^{y_1 \wedge y_{3}}.
\end{split}
\end{equation}
Using this inequality and integrating first with respect to $x_2$ and $x_3$ we obtain
\begin{equation} \label{eq:triple_int}
\begin{split}
\gamma_2^2 &=O(1)\cdot
e^H \cdot \frac{1}{ \sigma_n'^4}\int_{-I_n}^{I_n} \int_{[0,R/2]^3}
e^{(y_1\vee y_2)/2} \cdot
e^{(y_1\vee y_3)/2} e^{y_1\wedge y_2} \cdot e^{y_1 \wedge y_3}  \\
&\hspace{4cm}  \times e^{-\alpha y_1} e^{-\alpha y_2} e^{-\alpha y_3} dy_1 dy_2  dy_3 dx_1\\
&\stackrel{I_n =\Theta(n)}{=}O(1) \cdot e^{H}\frac{n}{ \sigma_n'^4}\int_{[0,R/2]^3}
e^{(y_1\vee y_2)/2} \cdot
e^{(y_1\vee y_3)/2} e^{y_1\wedge y_2} \cdot e^{y_1 \wedge y_3}  \\
& \hspace{4cm} \times e^{-\alpha y_1} e^{-\alpha y_2} e^{-\alpha y_3} dy_1 dy_2  dy_3.
\end{split}
\end{equation}

To bound the triple integral in~\eqref{eq:triple_int} for $\alpha \in (1, \infty)$, we will split the domain of integration  into four sub-domains:
\begin{equation*}
\begin{split}
D_1& := \{(y_1,y_2,y_3)\in [0,R/2]^3 \ : \ y_1\leq y_2,y_3\},\\
D_2&:=\{(y_1,y_2,y_3)\in [0,R/2]^3 \ : \ y_2 \leq y_1 \leq y_3\}, \\
D_3&:=\{(y_1,y_2,y_3)\in [0,R/2]^3 \ : \ y_3 \leq y_1 \leq y_2\}, \\
D_4& :=\{(y_1,y_2,y_3)\in [0,R/2]^3 \ : \ y_2, y_3 \leq  y_1 \}.
\end{split}
\end{equation*}
We evaluate the integral in~\eqref{eq:triple_int} on each of these four sub-domains.
On $D_1$ we have:
\begin{align*}
& \int_{D_1} e^{(y_1\vee y_2)/2} \cdot
e^{(y_1\vee y_3)/2} e^{y_1\wedge y_2} \cdot e^{y_1 \wedge y_3} \cdot e^{-\alpha y_1} e^{-\alpha y_2} e^{-\alpha y_3} dy_1 dy_2  dy_3 \\
& = \int_{D_1} e^{y_2/2}\cdot  e^{y_3/2} \cdot e^{2y_1} \cdot e^{-\alpha y_1} e^{-\alpha y_2} e^{-\alpha y_3} dy_1 dy_2  dy_3 \\
&\leq \int_{0}^{R/2} e^{2y_1} \left(\int_{y_1}^\infty e^{(\frac{1} {2}  - \alpha ) y_2} dy_2 \right)
\left(\int_{y_1}^\infty e^{( \frac{1} {2}  - \alpha ) y_3} dy_3 \right) e^{-\alpha y_1} dy_1 \\
&=O(1)\cdot \int_0^{R/2} e^{2y_1 + (1-2\alpha)y_1-\alpha y_1} dy_1 \\
&= O(1) \cdot \int_0^{R/2} e^{3(1 - \alpha) y_1} dy_1 =O(1).
\end{align*}
For the second sub-domain $D_2$, we get:
\begin{align*}
& \int_{D_2}  e^{(y_1\vee y_2)/2} \cdot
e^{(y_1\vee y_3)/2} e^{y_1\wedge y_2} \cdot e^{y_1 \wedge y_3} \cdot e^{-\alpha y_1} e^{-\alpha y_2} e^{-\alpha y_3} dy_1 dy_2  dy_3  \\
& = \int_{D_2}e^{y_1/2} \cdot e^{y_3/2} \cdot e^{y_2} \cdot e^{y_1} \cdot e^{-\alpha y_1} e^{-\alpha y_2} e^{-\alpha y_3} dy_1 dy_2  dy_3 \\
&=\int_{D_2} e^{3 y_1/2 + y_2 + y_3/2} \cdot e^{-\alpha y_1} e^{-\alpha y_2} e^{-\alpha y_3} dy_1 dy_2  dy_3 \\
&\leq \int_{0}^{R/2}e^{3y_1/2} \left(\int_0^{y_1} e^{y_2 (1-\alpha)} dy_2 \right)
\left(\int_{y_1}^{\infty} e^{(\frac{1} {2} -\alpha)y_3} dy_3 \right) \cdot e^{-\alpha y_1} dy_1 \\
&= O(1) \cdot
\int_{0}^{R/2} e^{3y_1/2 + (\frac{1} {2}
 -\alpha)y_1 - \alpha y_1} dy_1 \\
& =O(1) \cdot \int_{0}^{R/2} e^{2(1- \alpha) y_1} dy_1 =O(1).
\end{align*}
The third sub-domain $D_3$ gives an identical result due to symmetry.

Finally, for the fourth sub-domain $D_4$  we get:
\begin{align*}
& \int_{D_4}  e^{(y_1\vee y_2)/2} \cdot
e^{(y_1\vee y_3)/2} e^{y_1\wedge y_2} \cdot e^{y_1 \wedge y_3} \cdot e^{-\alpha y_1} e^{-\alpha y_2} e^{-\alpha y_3} dy_1 dy_2  dy_3 \\
& = \int_{D_4} e^{y_1} e^{y_2 + y_3} \cdot e^{-\alpha y_1} e^{-\alpha y_2} e^{-\alpha y_3} dy_1 dy_2  dy_3 \\
&  = \int_0^{R/2} e^{y_1}  \left( \int_0^{y_1}e^{(1-\alpha) y_2}dy_2 \right)
\left( \int_0^{y_1}e^{(1-\alpha) y_3}dy_3 \right)
e^{-\alpha y_1} dy_1 =O(1).
\end{align*}
Combining the integrals for each of the four sub-domains we obtain
\begin{equation} \label{fourcases}
\int_{0}^{R/2} \int_{0}^{R/2} \int_{0}^{R/2}
e^{(y_1\vee y_2)/2} \cdot
e^{(y_1\vee y_3)/2} e^{y_1\wedge y_2} \cdot e^{y_1 \wedge y_3} \cdot e^{-\alpha y_1} e^{-\alpha y_2} e^{-\alpha y_3} dy_1 dy_2  dy_3 =O(1).
\end{equation}
Substituting the bound \eqref{fourcases} into~\eqref{eq:triple_int} we deduce for $\alpha \in (1, \infty)$ that
$\gamma_2 \leq O(1)\cdot e^{H}\cdot n^{1/2}/\sigma_n'^2 = o(1).$  This completes the proof of Lemma \ref{gam2}.
\qed

\begin{lemma} \label{gam1} If $\alpha \in (1, \infty)$, then $\gamma_1 = o(1)$.
\end{lemma}

The proof of this lemma is almost identical to the proof of the previous lemma. We postpone it to
Section~\ref{sec:lemma_gam1}.

We now establish the central limit theorem at ~\eqref{CLTiso}.
Consider the random variable
$$
\hat{S}^{iso}_H : =\frac{ \tilde{S}^{iso}_H (\tilde{\cal P}_{\alpha} \cap D) - \Ex{\tilde{S}^{iso}_H (\tilde{\cal P}_{\alpha} \cap D)| E_n } }  { \sqrt  {\Var [\tilde{S}^{iso}_H (\tilde{\cal P}_{\alpha} \cap D)|E_n]}}
$$
on the conditional space $E_n$.
Theorem~\ref{thm:normal_approx} and  Lemmas \ref{gam3}- \ref{gam1} yield
\begin{equation} \label{dW}
d_W \left( \hat{S}^{iso}_H, N \right)\leq  \gamma_1 + \gamma_2 + \gamma_3 = o(1).
\end{equation}
Recalling that $\mathbb{P} (E_n) = 1 - O(n^{ 1 - \alpha})$, we have
\begin{align*}
& \Pro{\frac{  \tilde{S}^{iso}_H (\tilde{\cal P}_{\alpha} \cap D) - \mathbb{E} [ \tilde{S}^{iso}_H (\tilde{\cal P}_{\alpha} \cap D) ] } { \sqrt {\Var [ {S}^{iso}_H (\tilde{\cal P}_{\alpha} \cap D) ]} } \leq x} \\
& = \Pro{ \frac{  \tilde{S}^{iso}_H (\tilde{\cal P}_{\alpha} \cap D) - \Ex{ \tilde{S}^{iso}_H (\tilde{\cal P}_{\alpha} \cap D) }} { \sqrt {\Var [ \tilde{S}^{iso}_H (\tilde{\cal P}_{\alpha} \cap D) ]} } \leq x | E_n } + o(1) \\
& = \Pro{ \frac{ \tilde{S}^{iso}_H (\tilde{\cal P}_{\alpha} \cap D) - \Ex{ \tilde{S}^{iso}_H (\tilde{\cal P}_{\alpha} \cap D)| E_n } }  { \sqrt  {\Var [\tilde{S}^{iso}_H (\tilde{\cal P}_{\alpha} \cap D) | E_n] } } \leq x +o(1) | E_n } +o(1),
\end{align*}
where the last equality follows by Lemma~\ref{lem:variance equivalence}. Since the bound \eqref{dW} shows that $\tilde{S}^{iso}_H (\tilde{\cal P}_{\alpha} \cap D)$ satisfies a central limit theorem on $E_n$, the
probability on the left-hand side of the above display converges to $\Phi(x)$.  Thus the central limit theorem  ~\eqref{CLTiso} holds.

\vskip.3cm

\subsection{The regime $\alpha \in (1/2,1)$}

We establish that $\tilde{S}^{iso}_H (\tilde{\cal P}_{\alpha} \cap D) $ does not exhibit normal convergence for $\alpha \in (1/2, 1)$.
We redefine $A_n$ to be the event that ${\rm{card}}( (\tilde{\cal P}_{\alpha} \cap D) \cap \R{[h_1,R]}) =0$, where now
$h_1 := R/(2\alpha) +  (\log \log R)/ 2 \alpha$. That is, on the event $A_n$ there are no points having height greater than $h_1$. An elementary calculation
shows that $\P(A_n) = 1-o(1)$.

For any $x \in \mathbb{R}$ we have
\begin{align*}
& \Pro{\frac{ \tilde{S}^{iso}_H (\tilde{\cal P}_{\alpha} \cap D)  - \mathbb{E} [ \tilde{S}^{iso}_H (\tilde{\cal P}_{\alpha} \cap D)] } { \sqrt {\Var [ \tilde{S}^{iso}_H (\tilde{\cal P}_{\alpha} \cap D) ] } } \leq x} \\
& = \Pro{\frac{ \tilde{S}^{iso}_H (\tilde{\cal P}_{\alpha} \cap D)  - \mathbb{E} [\tilde{S}^{iso}_H (\tilde{\cal P}_{\alpha} \cap D) ] } { \sqrt {\Var [ \tilde{S}^{iso}_H (\tilde{\cal P}_{\alpha} \cap D)]  } } \leq x | A_n } + o(1).
\end{align*}

We are going to show that
\begin{equation} \label{eq:vars_ratio}
\Var [\tilde{S}^{iso}_H (\tilde{\cal P}_{\alpha} \cap D)  |A_n ] = o(\Var [\tilde{S}^{iso}_H (\tilde{\cal P}_{\alpha} \cap D)] ).
\end{equation}
Since
$$
 \Var \left[ \frac{ \tilde{S}^{iso}_H (\tilde{\cal P}_{\alpha} \cap D)  -  \mathbb{E} [ \tilde{S}^{iso}_H (\tilde{\cal P}_{\alpha} \cap D)]  } { \sqrt {\Var [\tilde{S}^{iso}_H (\tilde{\cal P}_{\alpha} \cap D)] }} |A_n \right] = \frac{\Var [\tilde{S}^{iso}_H (\tilde{\cal P}_{\alpha} \cap D)  |A_n]}{\Var [\tilde{S}^{iso}_H (\tilde{\cal P}_{\alpha} \cap D)] } =o(1)
 $$
this  implies that
$$
\Pro{\frac{ \tilde{S}^{iso}_H (\tilde{\cal P}_{\alpha} \cap D)  - \mathbb{E} [ \tilde{S}^{iso}_H (\tilde{\cal P}_{\alpha} \cap D) ] }
{ \sqrt {\Var [\tilde{S}^{iso}_H (\tilde{\cal P}_{\alpha} \cap D)]  } } \leq x | A_n} $$ cannot converge to $\Phi(x)$ and therefore
$$
\frac{ \tilde{S}^{iso}_H (\tilde{\cal P}_{\alpha} \cap D)  -  \mathbb{E} [ \tilde{S}^{iso}_H (\tilde{\cal P}_{\alpha} \cap D) ]} { \sqrt {\Var [\tilde{S}^{iso}_H (\tilde{\cal P}_{\alpha} \cap D)] }}
$$
cannot converge in distribution to a standard normally distributed random variable $N$.

We now show \eqref{eq:vars_ratio}. We will bound $\Var [\tilde{S}^{iso}_H (\tilde{\cal P}_{\alpha} \cap D)  |A_n]$ using the Poincar\'e inequality
$$
 \Var F  \leq \Ex{\int_\Space (\nabla_p F(\cP))^2 \lambda  (dp)}.$$
We put $\lambda$ to be $\mu_{\alpha}$ and set
$$
F: = \tilde{S}^{iso}_H (\tilde{\cal P}_{\alpha} \cap D)  |  A_n, \ \Space:=\R{[0,h_1]}, \ \mbox{and} \
\cP := \tilde{\cal P}_{\alpha} \cap D .
$$

By Lemma~\ref{eq:measures_bounds} and the discussion immediately after its statement, we have that
$|\nabla_p F(\cP)|$ is stochastically bounded by
$X(p)+1$ where $X(p)$ is a Poisson-distributed random variable with parameter $\eta \cdot e^{y(p)/2}$.
Hence,
\begin{align*}
\Ex{\int_\Space (\nabla_p F(\cP))^2 \lambda (dp)}  & \leq  \int_\Space \Ex{(\nabla_p F(\cP))^2} \lambda (dp) \\
& = O(1) \cdot n \int_0^{h_1} e^{(1-\alpha )y}dy \\
& = O(1) \cdot n \cdot e^{(1-\alpha)h_1}.
\end{align*}
Recalling  $h_1 = R/2\alpha  + (\log \log R)/\alpha$
and $n = \nu e^{R/2}$ we obtain
$$
\int_\Space \Ex{(\nabla_p F(\cP))^2} \lambda (dp)  = O(\log^{\frac{1-\alpha}{\alpha}} R) \cdot n^{1+ (1-\alpha)/\alpha}
\stackrel{\frac{1-\alpha}{\alpha} <2}{=} O(\log^2 R) \cdot n^{ \frac{1} {\alpha} }.
$$
We conclude that
$
\Var [\tilde{S}^{iso}_H (\tilde{\cal P}_{\alpha} \cap D)  |A_n] =O(\log^2 R) \cdot n^{1/\alpha}.
$
By Theorem \ref{prop:var_Shat} we have $\Var [\tilde{S}^{iso}_H (\tilde{\cal P}_{\alpha} \cap D)]  = \Theta (1) \cdot n^{3-2\alpha}$.
But $1/\alpha < 3 -2 \alpha$, for $\alpha \in (1/2, 1)$.
Thus~\eqref{eq:vars_ratio} follows, concluding the proof of Theorem \ref{mainCLT}.
\qed

\subsubsection*{Acknowledgement}
This work started in June of 2017 at the Fields Institute in Toronto, Canada, during the workshop on~ \emph{Random Geometric Graphs and their Applications to Complex Networks}.  The authors  thank the Fields Institute for its hospitality and support.

\bibliographystyle{plain}
\bibliography{IP}

\appendix

\section{Proof of Lemmas~\ref{lem:relAngle_generic} and~\ref{lem:dist}} \label{sec:basic_lemmas}

\begin{proof}[Proof of Lemma~\ref{lem:relAngle_generic}]
The expression for $\theta_R (r,r')$ is a consequence of the hyperbolic law of cosines at  \eqref{RE}.
We first prove (i).  We compute:
\begin{eqnarray*}
 \frac{\cosh r \cosh r' - \cosh R}{\sinh r \sinh r'}
 & = &
\frac{
 \frac14(e^{r+r'}+e^{r-r'}+e^{r'-r}+e^{-(r+r')}) -\frac12(e^R+e^{-R})
 }{
 \frac14(e^{r+r'}-e^{r-r'}-e^{r'-r}+e^{-(r+r')})
 } \\
 & = &
 1
 +
 2 \frac{
 e^{r-r'}+e^{r'-r} - e^R - e^{-R}
 }{
 e^{r+r'}-e^{r-r'}-e^{r'-r}+e^{-(r+r')}
 }  \\
 & = &
 1- 2e^{-(r+r' -R)}\left(\frac{1 - e^{r-r'-R} - e^{r'-r-R} + e^{-2R}}{1-e^{-2r}-e^{-2r'}+e^{-2(r+r')}}\right) \\
 & = &
  1 - x,
\end{eqnarray*}
where
\be \label{defx}
x:= 2e^{-(r+r' -R)} \cdot \frac{(1 - e^{r-r'-R})(1-e^{r'-r-R})}{(1-e^{-2r})(1-e^{-2r'})}.
\ee

By definition of $\Delta (r,r')$ it suffices to bound $\frac12 e^{R/2} \arccos (1 - x)$ above and below.
 First, we remark that $r-r'-R > -2r'$ since $r+r' > R$. This implies that $(1 - e^{r-r'-R}) / (1-e^{-2r'}) \in (0,1)$.
Similarly, we have $(1 - e^{r'-r-R}) / (1-e^{-2r}) < 1$.  Let $\eps \in (0,1)$.
Since $r,r' \in [C,R]$, it follows that if $C:= C(\eps)$ is large enough then we have
$$
\frac{1 - e^{r-r'-R}}{1-e^{-2r'}}> \frac{1 - e^{-r'}}{1- e^{-2r'}} > 1 -\eps,
$$
and
$$\frac{1 - e^{r-r'-R}}{1-e^{-2r}}> 1- \eps. $$
With $s:= 2 e^{-(r+r'-R)}$, this shows that
\begin{equation} \label{eq:x_bounds}  s(1-\eps)^2 < x <s.
\end{equation}
Taylor's expansion of $\arccos(\cdot )$ implies  there exists a constant $K>0$ such that
\begin{equation} \label{eq:arccos_expansion}
\sqrt{2v} - K v^{3/2} < \arccos (1 - v) < \sqrt{2v} + Kv^{3/2}, \ \ v \in (0,1).
\end{equation}
Replacing $v$ with $x$, inequality ~\eqref{eq:x_bounds} implies that
$$ (1 -\eps )\sqrt{2s} - K s^{3/2} < \arccos (1 - x) < \sqrt{2s} + Ks^{3/2}.$$
Now, since $r+r' > R + C$ it follows that $s < 2 e^{-C}$ and thus $s^{3/2}= s^{1/2} s
< s^{1/2} 2^{1/2}  2^{1/2}e^{-C}.$
If $C:= C(\eps)$ is large enough so that $2^{1/2} e^{-C} < \eps /K$, we have
\be \label{Ineq}
Ks^{3/2} \leq \eps \sqrt{2s}.
\ee
This yields
$$
(1 -2\eps ) \sqrt{2s} < \arccos (1 - x) < \sqrt{2s} (1 + \eps) < \sqrt{2s} (1 + 2\eps).
$$
Note that
\begin{equation}\label{eq:bbad2}
\frac12 e^{R/2} \sqrt{2s} = e^{R/2 - (r+r'-R)/2} = e^{R-(r+r')/2} = e^{(y+y')/2},
\end{equation}
\noindent
where we recall $y :=R-r$ and $y':=R-r'$.
So for $\eps  \in (0,1/2)$ we obtain
$$
(1 -2\eps ) e^{(y+y')/2} < \frac12 e^{R/2} \arccos (1 - x) < (1 + 2\eps) e^{(y+y')/2}.
$$
Replacing $2 \eps$ by $\eps$, the inequality  \eqref{eq:asymp1.0} follows.
\vskip.2cm
We now show (ii). The assumption  $r,r'  \in [R - H, R]$ implies that $|r-r'| \leq H$. Thus,
$$ e^{r-r' -R}  \leq e^{H-R}, \ \ e^{r'-r - R} \leq e^{H-R} \ \mbox{and} \ e^{-2r} \leq e^{-2(R-H)}, \ \ e^{-2r'} \leq  e^{-2(R-H)}.$$
The definition of $x$ gives
$x = 2 e^{-(r+r' - R)} (1 + \delta_n (r,r') ),$
where $\delta_n (r,r') = o(1)$ uniformly over all $r,r' \in [R-H,R]$. Thus,~\eqref{eq:arccos_expansion} implies that
\begin{equation}  \label{lam} \begin{split}
\arccos (1-x) &= 2 e^{-\frac{1}{2} (r+r'-R)} (1+\delta_n (r,r'))^{1/2} \left( 1 +
\Theta (x)\right)\\
 &=: 2 e^{-\frac{1}{2} (r+r'-R)} (1+ \lambda_n (r,r')).
\end{split}
\end{equation}
The result then follows by~\eqref{eq:bbad2}.

We now prove (iii). To see this, recall from that $C$ is chosen to satisfy
$2^{1/2} e^{-C} < \eps /K$. Thus, this implies that $\eps$ as a function of $C$ can be selected such that
$\eps = \Theta( e^{-C})$.
\end{proof}

\begin{proof}[Proof of Lemma~\ref{lem:dist}]
The proof involves  elementary calculations, included here for completeness.
For the lower bound, we have
\begin{equation*}
\begin{split}
\bar{\rho}_{\alpha, n}(y) & > \frac{\alpha (e^{\alpha (R-y)} - e^{-\alpha (R-y)})}{e^{\alpha R} + e^{-\alpha R}} =
\frac{\alpha e^{\alpha (R-y)}}{e^{\alpha R} + e^{-\alpha R}} - \frac{\alpha e^{-\alpha (R-y)}}{e^{\alpha R} + e^{-\alpha R}} \\
&= \frac{\alpha e^{\alpha (R-y)}}{e^{\alpha R}} - \frac{\alpha e^{\alpha (R-y)}}{e^{\alpha R}} + \frac{\alpha e^{\alpha (R-y)}}{e^{\alpha R} + e^{-\alpha R}} -  \frac{\alpha e^{-\alpha (R-y)}}{e^{\alpha R} + e^{-\alpha R}} \\
&= \frac{\alpha e^{\alpha (R-y)}}{e^{\alpha R}} + \alpha e^{\alpha (R-y)}\left(- \frac{1}{e^{\alpha R}} + \frac{1}{e^{\alpha R} + e^{-\alpha R}} \right) -  \frac{\alpha e^{-\alpha (R-y)}}{e^{\alpha R} + e^{-\alpha R}} \\
&= \alpha e^{-\alpha y} -\alpha e^{\alpha (R-y)}\frac{e^{-\alpha R}}{e^{\alpha R}(e^{\alpha R} + e^{-\alpha R})}
- \frac{\alpha e^{-\alpha (R-y)}}{e^{\alpha R} + e^{-\alpha R}}  \\
&> \alpha e^{-\alpha y} - \frac{\alpha e^{-\alpha y}}{e^{\alpha R}(e^{\alpha R} + e^{-\alpha R})}
- \frac{\alpha e^{-\alpha y}}{e^{\alpha R} + e^{-\alpha R}} \\
&> \alpha e^{-\alpha y} - \frac{2\alpha}{e^{\alpha R}} >  \alpha e^{-\alpha y} - \frac{2\alpha}{e^{\alpha R}-2}.
\end{split}
\end{equation*}
The upper bound is derived similarly:
\begin{equation*}
\begin{split}
\bar{\rho}_{\alpha, n}(y)  &< \frac{\alpha e^{\alpha (R-y)}}{e^{\alpha R} -2} = \frac{\alpha e^{\alpha R} e^{-\alpha y}}{e^{\alpha R} -2} \\
&= \frac{\alpha (e^{\alpha R}-2 + 2) e^{-\alpha y}}{e^{\alpha R} -2} \\
&=  \frac{\alpha (e^{\alpha R}-2) e^{-\alpha y}}{e^{\alpha R} -2} +  \frac{2 \alpha e^{-\alpha y}}{e^{\alpha R} -2} \\
&=\alpha e^{-\alpha y} + \frac{2 \alpha e^{-\alpha y}}{e^{\alpha R} -2} \\
&<\alpha e^{-\alpha y} + \frac{2 \alpha }{e^{\alpha R} -2}.
\end{split}
\end{equation*}
\end{proof}

\section{ Proof of Lemma ~\ref{lem:variance equivalence} } \label{sec:variance equivalence}

For a  point process $\tilde{\cal P}$ on $D$ and a point $p \in D([0,R/2]) \cap  \tilde{\cal P}$, we define $\hat{\xi}^{iso}(p, \tilde{\cal P})$ to be equal to 1 if and only if
$B(p) \cap D([0,R/2]) \cap \{ \tilde{\cal P}  \setminus \{p\} \} = \emptyset$.
In other words, $\hat{\xi}^{iso}(p, \tilde{\cal P})$ is equal to 1 precisely when $B(p)$ does not contain any other
points of $ \tilde{\cal P} $ of height at most $R/2$.
Otherwise we put $\hat{\xi}^{iso}(p, \tilde{\cal P})= 0.$  For such $p$ we have
\begin{equation}\label{eq:scores_ineq}
\txi^{iso}(p,  \tilde{\cal P} )  \leq \hat{\xi}^{iso}(p, \tilde{\cal P}).
\end{equation}

We write $\xi(p, \tilde{\cal P})$ instead of $\txi^{iso} (p, \tilde{\cal P} )$ and we write $\hat{\xi}(p, \tilde{\cal P})$ instead of
$\hat{\xi}^{iso}(p, \tilde{\cal P})$.  With this definition, we set
$$
\widehat{S}^{iso}_H(\tilde{\cal P}_{\alpha} \cap D )
 := \sum_{p \in \tilde{\cal P}_{\alpha} \cap D([0,H]) } \hat{\xi}(p, \tilde{\cal P}_{\alpha} \cap D).
 $$
Observe  that $\widehat{S}^{iso}_H(\tilde{\cal P}_{\alpha} \cap D ) $ is distributed as
$\tilde{S}^{iso}_H (\tilde{\cal P}_{\alpha} \cap D) $ conditional on $E_n$.

Thus
$$
\Var [\tilde{S}^{iso}_H (\tilde{\cal P}_{\alpha} \cap D)  | E_n] = \Var [\widehat{S}^{iso}_H( \tilde{\cal P}_{\alpha} \cap D)].
$$
We will show that
\begin{equation} \label{eq:key-identity}
 |\Var [\tilde{S}^{iso}_H (\tilde{\cal P}_{\alpha} \cap D)]  -\Var[\widehat{S}^{iso}_H (\tilde{\cal P}_{\alpha} \cap D)] | =o(n).
\end{equation}
By \eqref{eq:variances_hat_Siso} we have  $\Var [ \tilde{S}^{iso}_H (\tilde{\cal P}_{\alpha} \cap D)]   = \Omega (n)$ for $\alpha \in (1, \infty)$  and thus the first part of Lemma~\ref{lem:variance equivalence} will then follow.

With $E:= \{(p_1,p_2) \in D : y(p_2) \leq y(p_1) \leq H\}$, using~\eqref{eq:variance}
we write the difference~\eqref{eq:key-identity} as follows:
\begin{equation*}
\begin{split}
&|\Var [\tilde{S}^{iso}_H (\tilde{\cal P}_{\alpha} \cap D)]  -\Var[\widehat{S}^{iso}_H (\tilde{\cal P}_{\alpha} \cap D)] |\leq \\
& 2\beta^2 \cdot \left|\int_{E}
( c^{\hat{\xi}} ((x_1,y_1),(x_2,y_2))  - c^{\xi} ((x_1,y_1),(x_2,y_2)) )  e^{-\alpha y_1} e^{-\alpha y_2} dx_2 dy_2 dy_1 dx_1 \right|\\
&+\Ex{ \sum_{p \in \tilde{\cal P}_{\alpha} \cap D } (\hat{\xi}(p, \tilde{\cal P}) - \xi(p, \tilde{\cal P}))}.
\end{split}
\end{equation*}

Observe now that $B(p) \cap \R {[0,R/2]} \subseteq B_{R/2}^+(p)$ (cf.~\eqref{eq:ball_plus}).
Furthermore, Lemma~\ref{4.1} implies that if $p_1,p_2 \in E$, then
$B_{R/2}^+(p_1) \cap B_{R/2}^+(p_2) = \emptyset$, when
$$|x(p_2) -x(p_1)| >_\Phi 2  e^{R/4}  ( e^{y(p_1)/2} + e^{y(p_2)/2} ).$$
If this condition holds, we have $B(p_1) \cap B(p_2) \cap \R{[0,R/2]} = \emptyset$,
which in turn implies that
$c^{\hat{\xi}} (p_1,p_2) =0$.
As we did before (see Lemma 3.3), we set $Y_i = e^{y(p_i)/2}$ for $i=1,2$ (we will be using this notation inside several integrals - there, we will be writing $Y_i =e^{y_i/2}$, for
$i=1,2$).
This observation motivates us to split $E$ into two sets:
$$ E_1 := \{(p_1,p_2) \in E \ : \
 0<_\Phi |x(p_2) - x(p_1)| \leq_\Phi 2 e^{\frac{R} {4}} (Y_1 +Y_2)
 \}$$
 and its complement inside $E$.
 In particular, it will suffice to show
 \begin{equation} \label{first}  \left| \int_{E_1}
( c^{\hat{\xi}} ((x_1,y_1),(x_2,y_2))  - c^{\xi} ((x_1,y_1),(x_2,y_2)) )  e^{-\alpha y_1} e^{-\alpha y_2} dx_2 dy_2 dy_1 dx_1
\right| = O(n^{2 - \alpha})
 \end{equation}
 and
 \begin{equation} \label{second}
 \left| \int_{E\setminus E_1}
 c^{\xi} ((x_1,y_1),(x_2,y_2))   e^{-\alpha y_1} e^{-\alpha y_2} dx_2 dy_2 dy_1 dx_1 \right| = O(n^{2 - \alpha})
 \end{equation}
 as on $E\setminus E_1$ the other covariance vanishes.

 Let us first show \eqref{first}.  For any $p_1,p_2 \in \R{[0,H]}$, we write
\begin{equation*}
\begin{split}
&
c^{\hat{\xi}} (p_1,p_2) - c^{\xi} (p_1,p_2) \\
& = \mathbb{E} [ \hat{\xi}(p_1, ( \tilde{\cal P}_{\alpha} \cap D) \cup \{p_1,p_2\}) \hat{\xi}(p_2, ( \tilde{\cal P}_{\alpha} \cap D) \cup \{p_1,p_2\}) \\
&\hspace{5cm}-
{\xi}(p_1,( \tilde{\cal P}_{\alpha} \cap D) \cup \{p_1,p_2\}) \xi(p_2, ( \tilde{\cal P}_{\alpha} \cap D) \cup \{p_1,p_2\}) ] \\
& - \left( \mathbb{E} [ \hat{\xi}(p_1,( \tilde{\cal P}_{\alpha} \cap D) \cup \{p_1\})] \mathbb{E} [  \hat{\xi}(p_2, ( \tilde{\cal P}_{\alpha} \cap D) \cup \{p_2\})] \right. \\
&\hspace{4.5cm}\left.-
  \mathbb{E} [ \xi(p_1, ( \tilde{\cal P}_{\alpha} \cap D) \cup \{p_1\})]  \mathbb{E} [ \xi(p_2, ( \tilde{\cal P}_{\alpha} \cap D) \cup \{p_2\}) ]\right)
\end{split}
\end{equation*}
But
\begin{equation*}
\begin{split}
& \hat{\xi}(p_1, ( \tilde{\cal P}_{\alpha} \cap D) \cup \{p_1,p_2\}) \hat{\xi}(p_2, ( \tilde{\cal P}_{\alpha} \cap D) \cup \{p_1,p_2\})\\
&\hspace{5cm}-
\xi(p_1, ( \tilde{\cal P}_{\alpha} \cap D) \cup \{p_1,p_2\}) \xi(p_2, ( \tilde{\cal P}_{\alpha} \cap D) \cup \{p_1,p_2\})  \\
& \leq \mathbf{1} (B (p_1) \cap (\tilde{\cal P}_{\alpha} \cap D)   \cap \R {[R/2,R]}\not =\emptyset)
+
\mathbf{1} (B (p_2) \cap \tilde{\cal P}_{\alpha} \cap D \cap \R {[R/2,R]}\not =\emptyset).
\end{split}
\end{equation*}
In other words, if the left-hand side is 1, then $\tilde{\cal P}_{\alpha} \cap D
$ has a point
in $B (p_1)$ or in $B(p_2)$ that has height at least $R/2$.
But by~\eqref{eq:inclusion_mapped} we have
$$ \mu (B (p_1)\cap D([ \frac{R} {2},R ] )) \leq \mu (B^+ (p_1) \cap \R{[  \frac{R} {2} , R]}).$$
Recall that $B^+ (p_1)$ is the union of two disjoint sets - so its measure naturally splits into two terms.
The first term is
\begin{align*}
 \mu (B^+ (p_1)\cap \R{[\frac{R} {2} ,R -y(p_1)-C]})  & = O(1)\cdot  Y_1 \cdot \int_{R/2}^R e^{y/2 -\alpha y}dy \\
 &  = O(1) \cdot Y_1\cdot e^{( \frac{1} {2} -\alpha) R/2}  \\
 & = O(n^{-(\alpha - \frac{1} {2} )})\cdot Y_1.
 \end{align*}
Now, the second term is (using $y(p_1) \leq H$)
\begin{align*}
 \mu (B^+ (p_1) \cap \R{[R-y(p_1) -C,R]})) &  = O(1) \cdot n \int_{R-H}^R e^{-\alpha y}dy \\
&  = O(R^{4\alpha}) \cdot n^{1- 2\alpha}  \\
& = o(n^{-(\alpha - \frac{1} {2} )}).
\end{align*}
Therefore, since $y(p_2) \leq y(p_1)$ we deduce that
\begin{equation} \label{eq:measure_high_up}
 \mu (B (p_2)\cap \R {[R/2,R]}) \leq \mu (B (p_1)\cap \R {[ \frac{R} {2}, R]}) = O(n^{-(\alpha -\frac{1} {2} )})\cdot Y_1.
\end{equation}
Using these upper bounds, we obtain:
\begin{align}
&\int_{E_1} \mathbb{E} \left[
 \hat{\xi}((x_1,y_1), (\tilde{\cal P}_{\alpha} \cap D)\cup \{(x_1,y_1),(x_2,y_2)\})
 \hat{\xi}((x_2,y_2),  (\tilde{\cal P}_{\alpha} \cap D)\cup \{(x_1,y_1),(x_2,y_2)\}) \right. \nonumber \\
&\hspace{1.5cm} \left.
-\xi((x_1,y_1), (\tilde{\cal P}_{\alpha} \cap D)\cup \{(x_1,y_1),(x_2,y_2)\})
\xi((x_2,y_2),  (\tilde{\cal P}_{\alpha} \cap D)\cup \{(x_1,y_1),(x_2,y_2)\}) \right] \times
\nonumber \\
&\hspace{1cm} e^{-\alpha y_1} e^{-\alpha y_2} dx_2 dy_2 dy_1 dx_1
\nonumber \\
& = O(1)\cdot  n^{-(\alpha - \frac{1} {2} )} \int_{-I_n}^{I_n} \int_{0}^H \int_0^{y_1}
\int_{-I_n}^{I_n} {\mathbf 1}
\left(|x_2  -x_1|_\Phi< 2e^{R/4} (Y_1 + Y_2)\right) \nonumber  \\
& \ \ \ \ \ \
\times Y_1 e^{-\alpha y_1} e^{-\alpha y_2} dx_2 dy_2 dy_1 dx_1  \nonumber \\
& = O(1)\cdot  n^{-(\alpha - \frac{1} {2} )} \cdot e^{R/4} \cdot
\int_{-I_n}^{I_n} \int_{0}^H \int_0^{y_1} Y_1 (Y_1+Y_2) e^{-\alpha y_1} e^{-\alpha y_2}dy_2 dy_1 dx  \nonumber \\
&  \stackrel{e^{R/4} = O(n^{1/2}), Y_2 \leq Y_1}{=} O(1) \cdot n^{1 - \alpha} \int_{-I_n}^{I_n} \int_{0}^H \int_0^{y_1}
 Y_1^2 e^{-\alpha y_1} e^{-\alpha y_2} dy_2 dy_1 dx_1  \nonumber \\
& \stackrel{\alpha >1}{=} O(1) \cdot n^{2-\alpha}.  \label{covar1}
\end{align}
Note also that for any $p \in \R{[0,H]}$
\begin{equation*}
\begin{split}
\mathbb{E} [ \xi (p, ( \tilde{\cal P}_{\alpha} \cap D) \cup \{p\}) ]
&= \mathbb{E} [ \hat{\xi} (p,  ( \tilde{\cal P}_{\alpha} \cap D) \cup \{p\}) ]
\cdot \exp \left( - \mu ( B (p)\cap \R {[\frac{R} {2} ,R]} )\right)  \\
&= \mathbb{E} [ \hat{\xi} (p, ( \tilde{\cal P}_{\alpha} \cap D) \cup \{p\}) ] \left( 1+ O(n^{-(\alpha - \frac{1}{2} )}
)\right).
\end{split}
\end{equation*}
Therefore,
\begin{align*}
& \left| \mathbb{E} [  \hat{\xi} (p_1,  ( \tilde{\cal P}_{\alpha} \cap D) \cup \{p_1\})  ] \mathbb{E} [ \hat{\xi} (p_2,  ( \tilde{\cal P}_{\alpha} \cap D) \cup \{p_2\}) ] - \right. \\
&\hspace{5cm}\left. \mathbb{E} [  \xi (p_1,  ( \tilde{\cal P}_{\alpha} \cap D) \cup \{p_1\})  ]  \mathbb{E} [ \xi (p_1,  ( \tilde{\cal P}_{\alpha} \cap D) \cup \{p_2\}) ]
\right|  \\
& = O(n^{-(\alpha - \frac{1} {2}  )}) \cdot \Ex{ \xi (p_1,  ( \tilde{\cal P}_{\alpha} \cap D) \cup \{p_1\})  }
 \Ex{ \xi (p_2,  ( \tilde{\cal P}_{\alpha} \cap D) \cup \{p_2\}) } \\
&  = O(n^{-(\alpha - \frac{1} {2})  }).
\end{align*}
So
\begin{align}
& \int_{E_1}\left|
\Ex{\hat{\xi} ((x_1,y_1),  (\tilde{\cal P}_{\alpha} \cap D)\cup \{(x_1,y_1)\})}
\Ex{ \hat{\xi} ((x_2,y_2),  (\tilde{\cal P}_{\alpha} \cap D)\cup \{(x_2,y_2) \})} \right.
 \nonumber \\
&\hspace{1.5cm}
-\left.
\Ex{ \xi((x_1,y_1),  (\tilde{\cal P}_{\alpha} \cap D)\cup \{(x_1,y_1)\})}
\Ex{\xi ((x_2,y_2),  (\tilde{\cal P}_{\alpha} \cap D)\cup \{(x_2,y_2)\})}
\right|\times   \nonumber \\
&\hspace{1cm} e^{-\alpha y_1} e^{-\alpha y_2} dx_2 dy_2 dy_1 dx_1 \nonumber\\
&=O(n^{-(\alpha - \frac{1} {2} )})\cdot  \int_{-I_n}^{I_n} \int_{0}^H \int_0^{y_1}
\int_{-I_n}^{I_n} {\mathbf 1}\left(|x_2  -x_1|_\Phi< 2e^{R/4} (Y_1 + Y_2)\right) \nonumber  \\
& \ \ \ \ \ \ \times e^{-\alpha y_1} e^{-\alpha y_2} dx_2 dy_2 dy_1 dx_1  \nonumber  \\
&=O(n^{-(\alpha - \frac{1} {2} )})\cdot  e^{R/4} \int_{-I_n}^{I_n} \int_{0}^H \int_0^H
(Y_1 + Y_2) e^{-\alpha y_1} e^{-\alpha y_2} dy_2 dy_1 dx_1  \nonumber \\
&\stackrel{Y_1,Y_2\geq 1, Y_i = e^{y_i/2}}{=} O(n^{-(\alpha - \frac{1} {2} )})\cdot  e^{R/4} \int_{-I_n}^{I_n} \int_{0}^H \int_0^H
Y_1\cdot  Y_2 Y_1^{-1-2\alpha} Y_2^{-1-2\alpha} dY_2 dY_1 dx_1  \nonumber  \\
&\stackrel{e^{R/4} = O(n^{1/2})}{=}
O(n^{1+ \frac{1} {2} - (\alpha - \frac{1} {2} )}) = O(n^{2-\alpha}).  \label{covar2}
\end{align}
Combining \eqref{covar1} and \eqref{covar2} we obtain \eqref{first} as desired.

Now we establish \eqref{second}. In particular,  Lemma~\ref{lem:mu_intersection}(ii) implies that
for any $(p_1,p_2) \in E \setminus E_1$,
with $t= |x(p_1) - x(p_2)|_\Phi$ we have
$$\mu (\InterSec{p_1p_2}) =O(1)\cdot  t^{1-2\alpha} (Y_1 + Y_2)^{2\alpha} + n^{1-2\alpha}\cdot Y_1^{2\alpha}. $$
But $t> 2 e^{R/4} (Y_1 + Y_2)$ which implies that
$t^{1-2\alpha} (Y_1 +Y_2)^{2\alpha} = O(1) \cdot n^{-(\alpha - \frac{1} {2} )} (Y_1 +Y_2) $.
Since $Y_2 \leq Y_1 \leq R^2$, we deduce that
$$  \mu (\InterSec{p_1p_2})  = O(1)\cdot  t^{1-2\alpha} Y_1 = O(R^2) \cdot n^{-(\alpha -1/2)}= o(1).$$
So by~\eqref{eq:cov_intersection} we conclude that
$$c^{\xi} (p_1,p_2) = O(1)\cdot  t^{1-2\alpha} Y_1 \mathbb{E} [
\xi(p_1,(\tilde{\cal P}_{\alpha} \cap D )\cup \{p_1\})]   \mathbb{E} [
\xi(p_2,(\tilde{\cal P}_{\alpha} \cap D )\cup \{p_2\})] .
$$
Therefore, setting $t=|x_1-x_2|$
\begin{align*}
& \left| \int_{E\setminus E_1} c^{\xi} ((x_1,y_1),(x_2,y_2))e^{-\alpha y_1} e^{-\alpha y_2}  dx_2 dy_2 dy_1 dx_1\right|  \\
& \stackrel{t>2e^{R/4} (Y_1+Y_2)> e^{R/4}}{=} O(1)\cdot \int_{-I_n}^{I_n}  \int_0^H \int_0^H \int_{e^{R/4}}^{I_n} t^{1-2\alpha}
 Y_1 e^{-\alpha y_1} e^{-\alpha y_2}  dt dy_2 dy_1 dx_1 \\
 & = O(1) \cdot n \cdot e^{2(1-\alpha)R/4} = O(n^{2-\alpha}),
\end{align*}
where we recall $Y_1 := e^{y_1/2}, \alpha \in (1, \infty)$. Thus \eqref{second} holds.

To finish the bound on $|\Var [\widehat{S}^{iso}_H (\tilde{\cal P}_{\alpha} \cap D)]  -\Var[\tilde{S}^{iso}_H (\tilde{\cal P}_{\alpha} \cap D)] |$, we also need to bound
$\left|\mathbb{E} [ \widehat{S}^{iso}_H (\tilde{\cal P}_{\alpha} \cap D) ] - \mathbb{E} [ \tilde{S}^{iso}_H (\tilde{\cal P}_{\alpha} \cap D)] \right|$, from which will conclude
 the second part of Lemma~\ref{lem:variance equivalence}.

Applying the Campbell-Mecke formula~\eqref{eq:Campbell-Mecke}, we get
\begin{equation*}
\begin{split}
&\left| \mathbb{E} [\widehat{S}^{iso}_H (\tilde{\cal P}_{\alpha} \cap D)]  - \mathbb{E} [\tilde{S}^{iso}_H (\tilde{\cal P}_{\alpha} \cap D) ]  \right| \\
&= \left|\mathbb{E}  \sum_{ p \in \tilde{\cal P}_{\alpha} \cap D([0,H])  } \left( \hat{\xi} (p, \tilde{\cal P}_{\alpha} \cap D)  - \xi (p, \tilde{\cal P}_{\alpha} \cap D) \right) \right| \\
& \leq \int_{-I_n}^{I_n} \int_0^H
\mathbb{E}[ \hat{\xi}((x,y), (\tilde{\cal P}_{\alpha} \cap D)  \cup \{(x,y)\}) - \xi((x,y), (\tilde{\cal P}_{\alpha} \cap D)  \cup \{(x,y)\})]
e^{-\alpha y} dy dx.
\end{split}
\end{equation*}

But by~\eqref{eq:measure_high_up}, for any $p\in \R{[0,H]}$ we have
$$
\mathbb{E} [ \hat{\xi} (p, \tilde{\cal P}_{\alpha} \cap D) - \xi (p, \tilde{\cal P}_{\alpha} \cap D)]  \leq
\mu ( B(p) \cap \R{[R/2,R]}) = O(n^{-(\alpha - \frac{1} {2} )})\cdot e^{y(p)/2}.
$$
Substituting this into the above integral we get
\begin{equation*}
\left| \mathbb{E}[ \widehat{S}^{iso}_H (\tilde{\cal P}_{\alpha} \cap D)] - \mathbb{E} [\tilde{S}^{iso}_H (\tilde{\cal P}_{\alpha} \cap D)] \right| =
O(1) \cdot n^{ \frac{3} { 2}  -\alpha}.
\end{equation*}

Combining \eqref{first} and \eqref{second} with this, we deduce that
$|\Var [\widehat{S}^{iso}_H (\tilde{\cal P}_{\alpha} \cap D)]  -\Var[\tilde{S}^{iso}_H (\tilde{\cal P}_{\alpha} \cap D)] | = O(n^{2-\alpha}) \stackrel{\alpha >1}{=} o(n),$
which shows~\eqref{eq:key-identity}.

Furthermore, since $\sigma_n' = \sqrt{\Var [\tilde{S}^{iso}_H (\tilde{\cal P}_{\alpha} \cap D) ]} = O(n^{1/2})$ we obtain for $\alpha \in (1, \infty)$
$$
\frac{\left| \mathbb{E} [\widehat{S}^{iso}_H (\tilde{\cal P}_{\alpha} \cap D) ] - \mathbb{E} [ \tilde{S}^{iso}_H (\tilde{\cal P}_{\alpha} \cap D)] \right| }{\sigma_n'} = O(n^{\frac{3} {2} -\alpha - \frac{1} {2}} ) = O(n^{1-\alpha}) =o(1),
$$
which concludes the proof of the second part of the lemma.
\qed

\section{Proof of Lemma~\ref{gam1}} \label{sec:lemma_gam1}

For $\gamma_1$ we have:
\begin{equation*}
\begin{split}
\gamma_1 &= O(1)\cdot  \int_{-I_n}^{I_n} \int_{-I_n}^{I_n}\int_{-I_n}^{I_n} \int_{[0,R/2]^3}
\left( \Ex{(\nabla_{p_2}F)^2 (\nabla_{p_3}F)^2}\right)^{1/2}
\left( \Ex{(\nabla_{p_1,p_2}^2 F)^2 (\nabla_{p_1,p_3}^2F)^2}\right)^{1/2}  \\
& \hspace{4cm} \times e^{-\alpha y_1} e^{-\alpha y_2} e^{-\alpha y_3} dy_1 dy_2  dy_3 dx_3 dx_2 dx_1 \\
&= O(1)\cdot  \int_{-I_n}^{I_n} \int_{-I_n}^{I_n}\int_{-I_n}^{I_n} \int_{[0,R/2]^3}
\left( \Ex{(\nabla_{p_2}F)^2 (\nabla_{p_3}F)^2}\right)^{1/2}
\left( \Ex{(\nabla_{p_1,p_2}^2 F)^2 (\nabla_{p_1,p_3}^2F)^2}\right)^{1/2}  \\
& \hspace{2cm} \times
{\mathbf 1} (|x_2 - x_1|_{\Phi}\leq 2 \gamma e^{H/2} e^{(y_1 \vee y_2 )/2} )
{\mathbf 1} (|x_3 - x_1|_{\Phi}\leq 2 \gamma e^{H/2} e^{(y_1 \vee y_3 )/2} ) \\
& \hspace{4cm} \times e^{-\alpha y_1} e^{-\alpha y_2} e^{-\alpha y_3} dy_1 dy_2  dy_3 dx_3 dx_2 dx_1.
\end{split}
\end{equation*}
By the Cauchy-Schwarz inequality
\begin{equation*}
\begin{split}
\Ex{(\nabla_{p_2}F)^2 (\nabla_{p_3}F)^2} &\leq \Ex{(\nabla_{p_2}F)^4}^{1/2} \Ex{(\nabla_{p_3}F)^4}^{1/2} \\
&\leq
\frac{1}{\sigma_n'^4} \Ex{(X(p_2)+1)^4}^{1/2} \Ex{(X(p_3)+1)^4}^{1/2} \\
& = O(1) \cdot \frac{1}{\sigma_n'^4} e^{y_2 + y_3},
\end{split}
\end{equation*}
where the last equality follows by Lemma~\ref{eq:measures_bounds}.
Therefore,
\begin{equation} \label{eq:prod_operators}
\Ex{(\nabla_{p_2}F)^2 (\nabla_{p_3}F)^2}^{1/2} = O(1) \cdot \frac{1}{\sigma_n'^2} e^{ \frac{1} {2} (y_2 + y_3) }.
\end{equation}
Using ~\eqref{eq:joint_X} and \eqref{eq:prod_operators} and integrating first with respect to $x_2$ and $x_3$, we get
\begin{equation} \label{eq:triple_int_gamma_1}
\begin{split}
\gamma_1^2 &= O(1) \cdot e^{2H}\frac{1}{\sigma_n'^4} \int_{-I_n}^{I_n}\int_{[0,R/2]^3}
e^{y_2/2 + y_3/2}
e^{(y_1\vee y_2)/2} \cdot
e^{(y_1\vee y_3)/2}\cdot  e^{(y_1\wedge y_2)/2} \cdot e^{(y_1 \wedge y_3)/2}    \\
& \hspace{4cm} \times e^{-\alpha y_1} e^{-\alpha y_2} e^{-\alpha y_3} dy_1 dy_2  dy_3 dx_1 \\
& = O(1) \cdot e^{2H}\frac{n}{\sigma_n'^4} \int_{[0,R/2]^3}
e^{y_2/2 + y_3/2}
e^{(y_1\vee y_2)/2} \cdot
e^{(y_1\vee y_3)/2}\cdot e^{(y_1\wedge y_2)/2} \cdot e^{(y_1 \wedge y_3)/2}   \\
& \hspace{4cm} \times e^{-\alpha y_1} e^{-\alpha y_2} e^{-\alpha y_3} dy_1 dy_2  dy_3,
\end{split}
\end{equation}
where we use $I_n = \Theta(n)$ in the last equality.
We will bound the integral in \eqref{eq:triple_int_gamma_1} by considering the four sub-domains we considered for the bound
on $\gamma_2$.
We start with $D_1$ on which $y_1\leq y_2,y_3$:
\begin{equation*}
\begin{split}
\int_{D_1} &
e^{y_2/2 + y_3/2} \cdot
e^{(y_1\vee y_2)/2} \cdot
e^{(y_1\vee y_3)/2} e^{(y_1\wedge y_2)/2} \cdot e^{(y_1 \wedge y_3)/2} \cdot e^{-\alpha y_1} e^{-\alpha y_2} e^{-\alpha y_3} dy_1 dy_2  dy_3 \\
& =\int_{D_1} e^{y_2/2 + y_3/2}  e^{y_2/2}\cdot  e^{y_3/2} \cdot e^{y_1} \cdot e^{-\alpha y_1} e^{-\alpha y_2} e^{-\alpha y_3} dy_1 dy_2  dy_3 \\
&=\int_{D_1} e^{(1-\alpha) y_1} e^{(1-\alpha)y_2} e^{(1-\alpha)y_3} dy_2 dy_3 dy_1 \\
&\leq \int_{[0,R/2]^3} e^{(1-\alpha) y_1} e^{(1-\alpha)y_2} e^{(1-\alpha)y_3} dy_2 dy_3 dy_1
\stackrel{\alpha >1}{=}O(1).
\end{split}
\end{equation*}
On $D_2$, where $y_2 \leq y_1 \leq y_3$ we have:
\begin{align*}
& \int_{D_2}  e^{y_2/2 + y_3/2} \cdot e^{(y_1\vee y_2)/2} \cdot
e^{(y_1\vee y_3)/2} e^{(y_1\wedge y_2)/2} \cdot e^{(y_1 \wedge y_3)/2} \cdot e^{-\alpha y_1} e^{-\alpha y_2} e^{-\alpha y_3} dy_1 dy_2  dy_3  \\
& = \int_{D_2} e^{y_2/2 + y_3/2} \cdot
e^{y_1/2} \cdot e^{y_3/2} \cdot e^{y_2/2} \cdot e^{y_1/2} \cdot e^{-\alpha y_1} e^{-\alpha y_2} e^{-\alpha y_3} dy_1 dy_2  dy_3 \\
&= \int_{D_2} e^{(1-\alpha) y_1} e^{(1-\alpha)y_2} e^{(1-\alpha)y_3} dy_2 dy_3 dy_1 \\
& \leq \int_{[0,R/2]^3} e^{(1-\alpha) y_1} e^{(1-\alpha)y_2} e^{(1-\alpha)y_3} dy_2 dy_3 dy_1
=O(1).
\end{align*}
By symmetry, integration on $D_3$ gives the same upper bound.
Finally, on $D_4$ where $y_2,y_3 \leq y_1$ we get
\begin{align*}
& \int_{D_4} e^{y_2/2 + y_3/2} \cdot  e^{(y_1\vee y_2)/2} \cdot
e^{(y_1\vee y_3)/2} e^{(y_1\wedge y_2)/2} \cdot e^{(y_1 \wedge y_3)/2} \cdot e^{-\alpha y_1} e^{-\alpha y_2} e^{-\alpha y_3} dy_1 dy_2  dy_3  \\
& = \int_{D_4} e^{y_2/2 + y_3/2} \cdot e^{y_1} e^{y_2/2 + y_3/2} \cdot e^{-\alpha y_1} e^{-\alpha y_2} e^{-\alpha y_3} dy_1 dy_2  dy_3  \\
&= \int_{D_4} e^{(1-\alpha) y_1} e^{(1-\alpha)y_2} e^{(1-\alpha)y_3} dy_2 dy_3 dy_1 \\
& \leq \int_{[0,R/2]^3} e^{(1-\alpha) y_1} e^{(1-\alpha)y_2} e^{(1-\alpha)y_3} dy_2 dy_3 dy_1
=O(1).
\end{align*}
Combining these four upper bounds into~\eqref{eq:triple_int_gamma_1} we obtain
$\gamma_1 = O(1) \cdot e^{H} n^{1/2}/\sigma_n'^2 =  o(1)$ as desired, where the last equality follows since  $\sigma_n' = \Theta (n^{1/2})$.
This completes the proof of Lemma \ref{gam1}.  \qed


\end{document}